\documentclass[twoside]{amsart}
\usepackage[foot]{amsaddr}
\usepackage{array}
\usepackage{amssymb}
\usepackage{amsmath}
\usepackage{amsthm}
\usepackage{latexsym}
\usepackage{bm}
\usepackage{enumerate}
\usepackage[dvips]{graphicx}
\usepackage{epsf}
\usepackage{wrapfig}
\usepackage{euscript}
\usepackage{textcomp}
\usepackage{setspace}
\usepackage[margin=1.5in]{geometry}
\usepackage{marginnote}

\usepackage{color}
\newtheorem{theorem}{Theorem}
\newtheorem{lemma}{Lemma}
\newtheorem{definition}{Definition}
\newtheorem{proposition}{Proposition}
\newtheorem{example}{Example}
\newtheorem{corollary}{Corollary}
\begin{document}
\binoppenalty = 10000 %
\relpenalty   = 10000 %

\pagestyle{headings} \makeatletter

\renewcommand{\@oddhead}{\raisebox{0pt}[\headheight][0pt]{\vbox{\hbox to\textwidth{{Lindstr\"{o}m theorem for intuitionistic first-order logic}\hfill \strut\thepage}\hrule}}}
\makeatother

\title{A Lindstr\"{o}m theorem \\
 for intuitionistic first-order logic}
\author{Grigory Olkhovikov}

\address[a]{ Department of Philosophy I\\
Ruhr University Bochum\\
Bochum, Germany }

\address[a]{ Department of Philosophy \\
Ural Federal University\\
Ekaterinburg, Russia }

\email[a]{grigory.olkhovikov@rub.de, grigory.olkhovikov@gmail.com}

\author{Guillermo Badia}

\address[b]{
School of Historical and Philosophical Inquiry\\
University of Queensland\\
Brisbane, Australia }

\address[b]{
Institute of Philosophy and Scientific Method\\
Johannes Kepler University\\
Linz, Austria } 

 \email[b]{guillebadia89@gmail.com}

\author{Reihane Zoghifard}

\address[c]{ School of Mathematics\\
	Institute for Research in Fundamental Sciences (IPM), PO Box 19395-5746\\
	Tehran, Iran}

\email[c]{r.zoghi@gmail.com}

\date{}
\maketitle

\begin{abstract}
We extend the main result of \cite{baok} to the first-order
	intuitionistic logic (with and without equality), showing that it is the maximal (with respect
	to expressive power) abstract logic satisfying a certain form of
	compactness, the Tarski union property and preservation under
	asimulations. A similar result is also shown for the
	intuitionistic logic of constant domains.
	
	\bigskip
	
	
	\bigskip
	
	\emph{Math subject classification:}  Primary 03C95; Secondary 03B55.
\end{abstract}

\section{Introduction}

In 1969, 
Per Lindstr\"{o}m published his most important contribution to mathematical logic \cite{lin}.  This article contained a new theoretical framework where a logic  was  treated, in an abstract manner, as a language equipped with a
semantics (sometimes also called a ``model-theoretic language'' \cite{fe}) rather than as a set of theorems or derivations in a
certain formal system. Lindstr\"{o}m's central result
characterized classical first-order logic as a maximal language, in a
wide class of languages equipped with a similar semantics,
enjoying a combination of desirable properties which basically
allowed to  derive the bulk of classical model-theoretic results
known about this logic. Thus, classical first-order logic was,
roughly, the largest logic allowing for the Compactness  and  L\"owenheim-Skolem theorems,
the first results of first-order model theory (for a textbook presentation of the result see \cite{flum}). Lindstr\"om's landmark 11-page paper subsequently generated several other model-theoretic characterizations of classical first-order logic, all obtained by Lindstr\"om himself \cite{lin2, lin3, lin4}. These results contributed to clarify
 the very notion of first-order model theory in the context of the model theory of other logics with more expressive power.  \emph{Abstract model theory} \cite{bar1, barfer} was born out of these investigations as a field where model theory itself is the object of meta-mathematical investigation.

Naturally, Lindstr\"om's breakthrough result inspired similar
characterizations for other types of logic, including modal
and non-classical ones \cite{mm, vanB, otto, enqvist}. Many logics in this group can be described
as the fragments of a corresponding classical first-order logic preserved
under an appropriate simulation notion by a variant of the methods
used in the proof of van Benthem modal characterization theorem.
Therefore, if one only cares about placing a given modal or
non-classical logic among its possible extensions that are still
fragments of classical first-order logic, this problem has been addressed in substantial generality by the first author in \cite{o2}.

In \cite{baok}, the first and second authors (inspired by \cite{otto, enqvist}) provided a Lindstr\"om theorem for intuitionistic propositional  logic. In particular, they showed that only the model theory of  intuitionistic propositional  logic enjoyed a Compactness property, the Tarski Union Property and preservation under asimulations. Asimulations had been introduced by the first author in studying the expressive power of intuitionistic propositional  logic in \cite{o}, and they turned out to be  equivalent to the intuitionistic bisimulations of  \cite{kurto}. More or less simultaneously with the writing of \cite{baok} by the first two authors, the third author had successfully co-authored an extension of the methods of \cite{otto} to prove a Lindstr\"om theorem for first-order modal logic in \cite{zopo}. It was also the third author, who brought in renewed enthusiasm to the efforts by the first and second authors to extend the techniques from  \cite{baok} to cover the first-order intuitionistic case.
This generalization ended up being rather non-trivial due to the complexities of intuitionistic first-order logic, and hence  the present article.

This paper is arranged as follows: Section \ref{S:Prel} is concerned with the technical  preliminaries on intuitionistic first-order logic without equality, in particular, notions of asimulation, unravellings, submodels and theories (among others) are introduced. Section \ref{S:Standard} describes what we will call \emph{standard intuitionistic logics}, specifically, six logical
systems including the intuitionistic logic of constant domains and the addition of the most common variants of equality.  Section \ref{S:Abstract} presents the notion of an \emph{abstract intuitionistic logic} and the three properties of logics that will be involved in our main result. Section \ref{S:main} (the main section) contains the relatively  lengthy proof  of the central theorem of the paper, showing that the Tarski Union Property, Compactness and preservation under asimulations characterize all the most typical versions of first-order intuitionistic logic (with and without identiy), including the intuitionistic logic of constant domains. Finally, Section \ref{S:conclusion} offers directions for further research.

\section{Intuitionistic first-order logic without identity}\label{S:Prel}

\subsection{Preliminaries}

In this subsection, we introduce the intuitionistic first-order
logic $\mathsf{IL}$ without identity, based on any set of
predicate letters of positive arity and individual constants. We
allow neither functions, nor improper predicate letters (i.e.
propositional letters) although the main results of our paper are
straightforwardly extendable to the languages with propositional
letters, and the same appears to be true for the languages with
functional symbols.

An ordered couple of sets $\Sigma = (Pred_\Sigma, Const_\Sigma)$
comprising all the predicate letters and constants allowed in a
given version of intuitionistic language will be called the
\emph{signature} of this language. Signatures will be denoted by
letters $\Sigma$ and $\Theta$. For a given signature $\Sigma$, the
elements of $Pred_\Sigma$ will be denoted by $P^n$ and $Q^n$,
where $n
> 0$ indicates the arity, and the elements of $Const_\Sigma$ will
be denoted by $c$, $d$, and $e$. All these notations and all of
the other notations introduced in this section can be decorated by
all types of sub- and superscripts. We will often use the notation
$\Sigma_n$ for the set of $n$-ary predicates in a given signature
$\Sigma$.

Even though we have defined signatures as ordered pairs, we will
somewhat abuse the notation in the interest of suggestivity and,
given a set $\Pi$ of predicates and a set $\Delta$ of individual
constants, we will denote by $\Sigma \cup \Pi \cup \Delta$ the
signature where the predicates from $\Pi$ are added to
$Pred_\Sigma$ with their respective arities and the constants from
$\Delta$ are added to $Const_\Sigma$.

If $\Sigma$ is a signature, then the set of formulas generated
from $\Sigma$ using the set of logical symbols $\{ \wedge, \vee,
\bot, \to, \forall, \exists \}$ and the set of (individual)
variables $Var := \{ v_i \mid i < \omega \}$ will be denoted
$IL(\Sigma)$ and will be called the set of intuitionistic
first-order $\Sigma$-formulas without identity. We will also use
the connectives $\neg$ and $\leftrightarrow$ under their standard
definitions. The elements of $Var$ will be also denoted by $x,y,z,
w$, and the elements of $IL(\Sigma)$ by Greek letters like
$\varphi$, $\psi$ and $\theta$.

For any $n > 0$, we will denote by $\bar{o}_n$ the sequence
$(o_1,\ldots, o_n)$ of objects of any kind; moreover, somewhat
abusing the notation, we will sometimes denote $\{o_1,\ldots,
o_n\}$ by $\{\bar{o}_n\}$. We will denote by
$(\bar{o}_n;\bar{r}_m)$ the ordered couple of two sequences
$\bar{o}_n$ and $\bar{r}_m$; the notation
$(\bar{o}_n)^{\frown}(\bar{r}_m)$ will be reserved for the
concatenation of $\bar{o}_n$ and $\bar{r}_m$. If $f$ is a
function, then we will denote the $n$-tuple $(f(o_1),\ldots,
f(o_n))$ by $f\langle\bar{o}_n\rangle$ (by contrast,
$f(\bar{o}_n)$ will denote the application of $f$ to $\bar{o}_n$
viewed as a separate argument of $f$). The ordered $1$-tuple will
be identified with its only member and the ordered $0$-tuple will
be denoted by $\Lambda$.

If $f$ is any function, then we will denote by $dom(f)$ its domain
and by $rang(f)$ the image of $dom(f)$ under $f$; if $rang(f)
\subseteq M$, we will also write $f: dom(f) \to M$.

For a given set $\Omega$ and $k < \omega$, the notation $\Omega^k$
(resp. $\Omega^{\neq k}$) will denote the $k$-th Cartesian power
of $\Omega$ (resp. the set of all $k$-tuples from $\Omega^k$ such
that their elements are pairwise distinct). If $X, Y$ are sets,
then we will write $X \Subset Y$, if $X \subseteq Y$ and $X$ is
finite.

For any given signature $\Sigma$, and any given $\varphi \in
IL(\Sigma)$, we define $BV(\varphi)$ and $FV(\varphi)$, its sets
of bound and free variables, by the usual inductions. These sets
are always finite. There is no problem to extend these definitions
to an arbitrary $\Gamma \subseteq IL(\Sigma)$, although in this
case $BV(\Gamma)$ and $FV(\Gamma)$ need not be finite.

We will denote the set of $IL(\Sigma)$-formulas with free
variables among the elements of $\bar{x}_n$ by
$IL_{\bar{x}_n}(\Sigma)$; in particular, $IL_\emptyset(\Sigma)$
will stand for the set of $\Sigma$-sentences. If $\varphi \in
IL_{\bar{x}_n}(\Sigma)$ ($\Gamma \subseteq
IL_{\bar{x}_n}(\Sigma)$), then we will also express this by
writing $\varphi(\bar{x}_n)$ ($\Gamma(\bar{x}_n)$).

For a given $\Sigma$, the elements of $IL(\Sigma)$ will be
interpreted over first-order Kripke $\Sigma$-models (called just
$\Sigma$-models below) of the form $\mathcal{M} = \langle W,
\prec, \mathfrak{A}, \mathbb{H}\rangle$. The tuple of the form
$\mathcal{M} = \langle W, \prec, \mathfrak{A}, \mathbb{H}\rangle$,
is a first-order Kripke $\Sigma$-model iff:
\begin{itemize}
    \item $W$ is a non-empty set of \emph{worlds} or \emph{nodes};
    \item $\prec \subseteq W \times W$ is a partial order;
    \item $\mathfrak{A}$ is a function defined on $W$ such that for all $w \in W$, $\mathfrak{A}(w) = (A_w, I_w)$ (also
    denoted $\mathfrak{A}_w$) is a classical first-order
    $\Sigma$-model, with $A_w$ as its domain and $I_w$ as its
    interpretation of the predicates and constants in $\Sigma$.
    The function $\mathfrak{A}$ is assumed to satisfy the
    following condition:
    $$
(\forall w,v \in W)(w \neq v \Rightarrow A_w \cap A_v =
\emptyset).
$$
    \item $\mathbb{H}$ is a function defined on $\{ (w,v) \in W \times W\mid w \prec v\}$ such
    that $\mathbb{H}(w,v):\mathfrak{A}_w \to
    \mathfrak{A}_v$ is a homomorphism (also denoted $\mathbb{H}_{wv}$), that is to say, we have
    $$
    I_w(P)(\bar{a}_n) \Rightarrow
    I_v(P)(\mathbb{H}_{wv}\langle\bar{a}_n\rangle), I_v(c) =
    \mathbb{H}_{wv}(I_w(c))
    $$
    for all $\bar{a}_n \in A_w^n$, every $c \in Const_\Sigma$, and every $P \in \Sigma_n$. In addition, the system of
    homomorphisms associated with $\mathbb{H}$ must satisfy the
    following conditions for all $w,v,u \in W$ such that $w \prec v \prec u$:
    \begin{itemize}
        \item $\mathbb{H}_{ww} = id_{A_w}$;
        \item $\mathbb{H}_{wu} = \mathbb{H}_{vu}\circ \mathbb{H}_{wv}$, in other words, $\mathbb{H}_{wu}(a) = \mathbb{H}_{vu}(\mathbb{H}_{wv}(a))$
        for all $a \in A_w$.
    \end{itemize}
\end{itemize}

As is usual, we denote the reduct of a $\Sigma$-model
$\mathcal{M}$ to a smaller signature $\Theta$ by
$\mathcal{M}\upharpoonright\Theta$.

We will denote by $Mod$ (resp. $Mod_\Theta$) the class of all
models in all signatures (resp. in the signature $\Theta$). For a
given $\mathcal{M} \in Mod_\Theta$ and a given $w \in W$, $A_w$ is
called the domain of $w$ and $I_w$ is called the interpretation
assigned to $w$. We will also lift this up to the model level, and
will say that $I$ and $A$ are the interpretation- and the
domain-function assigned to $\mathcal{M}$. Moreover, we will call
the domain of $\mathcal{M}$ (and write $\mathbb{A}$) the set
$\bigcup_{w \in W}A_w$.

We will strive to denote model components in a way that is
consistent with the notation for the model itself. Some examples of
typical notations for models and their components are given below:
$$
\mathcal{M} = \langle W, \prec,\mathfrak{A}, \mathbb{H}\rangle,
\mathcal{M}' = \langle W', \prec', \mathfrak{A}',
\mathbb{H}'\rangle, \mathcal{M}_n = \langle W_n, \prec_n,
\mathfrak{A}_n, \mathbb{H}_n\rangle,
$$
$$
\mathcal{N} = \langle U, \lhd,\mathfrak{B}, \mathbb{G}\rangle,
\mathcal{N}' = \langle U', \lhd',\mathfrak{B}',
\mathbb{G}'\rangle,\mathcal{N}_n = \langle U_n,
\lhd_n,\mathfrak{B}_n, \mathbb{G}_n\rangle,
$$
where $n \in \omega$, and for $v \in U$ we will assume that
$\mathfrak{B}_v = (B_v, J_v)$ and similarly in other cases. We
will also sometimes use the notation of the form $\mathcal{M} =
\langle W_\mathcal{M}, \prec_\mathcal{M},
\mathfrak{A}_\mathcal{M}, \mathbb{H}_\mathcal{M}\rangle$.
\begin{definition}\label{D:submodel}
For any given two $\Sigma$-models $\mathcal{M}$ and $\mathcal{N}$,
we will say that $\mathcal{M}$ is a submodel of $\mathcal{N}$ and
write $\mathcal{M} \subseteq \mathcal{N}$ iff all of the following
conditions hold:
\begin{itemize}
    \item $W \subseteq U$;
    \item $\prec = \lhd \cap (W \times W)$;
    \item $(\forall w \in W) (\mathfrak{A}_w \subseteq
    \mathfrak{B}_w)$, i.e. $\mathfrak{A}_w$ is a classical submodel of $\mathfrak{B}_w$;
    \item $(\forall w,v \in W)(w \lhd v \Rightarrow \mathbb{G}_{wv}\upharpoonright A_w
    = \mathbb{H}_{wv})$.
\end{itemize}
\end{definition}
Now we can prove the following lemma:
\begin{lemma}\label{L:submodel}
   Let $\Sigma$ be a signature, let $\mathcal{N}$ be a
    $\Sigma$-model, let $W \subseteq U$, and let the set $X$ be such that
    $\{ J_w(c) \mid c  \in Const_\Sigma,\,w \in W\} \subseteq X
\subseteq \bigcup_{w \in W}B_w \subseteq \mathbb{B}$ and be closed
under every function in
$\{\mathbb{G}_{wv} \mid w,v \in W\,\&\, w \lhd v \}$. Then
consider $\mathcal{M}$, such that for all $w,v \in W$, every $P
\in \Sigma_n$ and every $c \in Const_\Sigma$:
\begin{itemize}
    \item $\prec = \lhd \cap (W \times W)$;
    \item $A_w = X \cap B_w$;
    \item $I_w(P) = J_w(P) \cap A_w^n$;
    \item $I_w(c) = J_w(c)$;
    \item $w \lhd v \Rightarrow \mathbb{H}_{wv} = \mathbb{G}_{wv}\upharpoonright
    A_w$.
\end{itemize}
Then $\mathcal{M}$ is a $\Sigma$-model, we have that $\mathbb{A} = X$,
and $\mathcal{M} \subseteq \mathcal{N}$, and, moreover,
if $\mathcal{M}' \subseteq \mathcal{N}$ is such that $W' = W$ and
$\mathbb{A}' = X$, then also $\mathcal{M}' =
\mathcal{M}$. Therefore, we will also denote $\mathcal{M}$ by
$\mathcal{N}(W, \mathbb{A})$.
\end{lemma}
\begin{proof} It is immediate to check that, under the settings
given in the Lemma, we have both $\mathcal{M}\in Mod_\Sigma$ and $\mathcal{M} \subseteq \mathcal{N}$. As for $\mathbb{A}$, note that we have
$$
\mathbb{A} = \bigcup_{w \in W}A_w = \bigcup_{w \in W}(X \cap B_w) = X\cap \bigcup_{w \in W}B_w = X,
$$
by the definiton of $A_w$ and the choice of $X$.

Now, assume that $\mathcal{M}' \subseteq \mathcal{N}$ is such that
$W' = W$ and $\mathbb{A}' = \mathbb{A}$. Then we will have $\prec'
=  \lhd \cap (W' \times W') =  \lhd \cap (W \times W) = \prec$.

Next, if $w \in W' = W$ and $a \in A'_w$, then $a \in \mathbb{A}'
= \mathbb{A}$, but also $A'_w \subseteq B_w$ by the assumption
that $\mathcal{M}' \subseteq \mathcal{N}$. Therefore $A'_w
\subseteq \mathbb{A} \cap B_w$. In the other direction, if $a \in
\mathbb{A} \cap B_w = \mathbb{A}' \cap B_w$, then we know that
there is a $v \in W = W'$ such that $a \in A'_v$. If $v = w$ then
we are done, otherwise, by $A'_v \subseteq \mathbb{A}' \cap B_v$,
we know that $a \in B_v$ for some $v \neq w$. But also we have $a
\in \mathbb{A} \cap B_w \subseteq B_w$. Thus $a \in B_w \cap B_v$,
which contradicts the fact that $\mathcal{N}$ is a model.
Therefore, for all $w \in W' = W$, we have $A'_w = \mathbb{A} \cap
B_w = A_w$.

But then, if $P \in \Sigma_n$, $c \in Const_\Sigma$, and $w \in W'
= W$, we have, by
 $\mathcal{M}' \subseteq \mathcal{N}$, that $\mathfrak{A}'_w \subseteq \mathfrak{B}_w$ and thus:
$$
I'_w(P) = J_w(P) \cap (A'_w)^{n} = J_w(P) \cap A_w^n = I_w(P),
$$
and:
$$
I'_w(c) = J_w(c) = I_w(c).
$$
It follows then that $\mathfrak{A}'_w = \mathfrak{A}_w$ for all $w
\in W' = W$.

Finally, for all $w,v \in W' = W$ such that $w \lhd v$ we have
that
$$
\mathbb{H}'_{wv} = \mathbb{G}_{wv}\upharpoonright
    A'_w = \mathbb{G}_{wv}\upharpoonright
    A_w = \mathbb{H}_{wv},
$$
so that $\mathcal{M}' = \mathcal{M}$ follows.
\end{proof}
Note that the same proof does not go through if the domain
function in a model is allowed to have non-disjoint values.

Moreover, we note, in passing, that one can even show the existence of a bijection between the family of submodels of any given 	$\Sigma$-model $\mathcal{N}$ and the family of all pairs $(W,X)$ satisfying the conditions of the above lemma. We do not do this, though, since it is not necessary for our main result.

It follows from the above lemma that if $W \subseteq U$ is
arbitrary, then there exists a unique $\mathcal{M} \subseteq
\mathcal{N}$ such that $\mathcal{M} = \mathcal{N}(W, \bigcup_{w
\in W}B_w)$. For simplicity, we will denote such a node-set induced
submodel of $\mathcal{N}$ by $\mathcal{N}(W)$.

If $\mathcal{M}_1 \subseteq,\ldots, \subseteq \mathcal{M}_n
\subseteq,\ldots$ is a countable chain of $\Sigma$-models then the
model:
$$
\bigcup_{n \in \omega}\mathcal{M}_n = (\bigcup_{n \in \omega}W_n,
\bigcup_{n \in \omega}\prec_n, \bigcup_{n \in
\omega}\mathfrak{A}_n, \bigcup_{n \in \omega}\mathbb{H}_n)
$$
is again a $\Sigma$-model and we have:
$$
\mathcal{M}_k \subseteq \bigcup_{n \in \omega}\mathcal{M}_n
$$
for every $k \in \omega$.

A \emph{pointed} $\Theta$-model is a pair of the form
$(\mathcal{M}, w)$ such that $w \in W$. The class of all pointed
models in all signatures (resp. in the signature $\Theta$) will be
denoted by $Pmod$ (resp. by $Pmod_\Theta$). A pair $(\mathcal{M},
w) \in Pmod_\Theta$ is called a \emph{rooted} $\Theta$-model iff
we have $w\mathrel{\prec}v$ for all $v \in W$. The class of rooted
$\Theta$-models will be denoted by $Rmod_\Theta$.

If $(\mathcal{M}, w)$ is a pointed $\Theta$-model, then one can
canonically generate from it a rooted model $([\mathcal{M}, w],
w)$ such that $[\mathcal{M}, w] = \mathcal{M}([W, w]) \subseteq
\mathcal{M}$ where,  $[W, w] = \{ v \in W\mid w \prec v\}$. In
what follows, we will also use the following alternative notation
for the generated submodels: $[\mathcal{M}, w] = \langle [W, w],
[\mathfrak{A}, w], [\prec, w], [\mathbb{H},w] \rangle$.


Next, we need a definition of isomorphism between models:
\begin{definition}\label{D:isomorphism}
{\em Let $\mathcal{M}$, $\mathcal{N}$ be $\Theta$-models. A pair
of functions $(g,h)$ such that $g: \mathbb{A} \to \mathbb{B}$ and
$h:W \to U$ is called an \emph{ isomorphism from $\mathcal{M}$
onto $\mathcal{N}$} (write $(g,h):\mathcal{M} \cong \mathcal{N}$)
iff $g$, $h$ are bijections, and for all $v,u \in W$, all $n <
\omega$, and all $a \in A_{v}$ it is true that
\begin{align}
&v\mathrel{\prec}u \Leftrightarrow h(v)\mathrel{\lhd}h(u)
\label{E:ic1}\tag{\text{i-rel}}\\
&(g\upharpoonright A_v):\mathfrak{A}_v \cong \mathfrak{B}_{h(v)}\label{E:ic2}\tag{\text{i-dom}}\\
&v\mathrel{\prec}u \Rightarrow (g(\mathbb{H}_{vu}(a)) =
\mathbb{G}_{h(v)h(u)}(g(a)))\label{E:ic2a}\tag{\text{i-map}}
\end{align}
where $(g\upharpoonright A_v):\mathfrak{A}_v \cong
\mathfrak{B}_{h(v)}$ means that $(g\upharpoonright A_v)$ is a
classical isomorphism from $\mathfrak{A}_v$ onto
$\mathfrak{B}_{h(v)}$.
 }
\end{definition}
If we have both $(g,h):\mathcal{M} \cong \mathcal{N}$ and, for a $w \in W$ we also have $h(w) = v \in U$, then we can express this more concisely by writing $(g,h):(\mathcal{M}, w) \cong (\mathcal{N}, v)$.

If $\Sigma$ is a signature, $1 \leq n < \omega$, and we have
$Const_\Sigma \cap \{\bar{c}_n\} = \emptyset$, then for any
$(\mathcal{M}, w) \in Rmod_\Sigma$, and any $\bar{a}_n \in A^n_w$,
we can canonically (and uniquely) extend $\mathcal{M}$ to the
model $\mathcal{M}' =(\mathcal{M}, \bar{c}_n/\bar{a}_n) \in
Mod_{\Sigma \cup \{\bar{c}_n\}}$ such that:
$$
I'_w\langle\bar{c}_n\rangle = \bar{a}_n.
$$
It follows then easily that for all $v \in W$ we will have:
$$
I'_v\langle\bar{c}_n\rangle =
\mathbb{H}_{wv}\langle\bar{a}_n\rangle.
$$
Of course, we will also have $(\mathcal{M}', w) \in Rmod_{\Sigma
\cup \{\bar{c}_n\}}$ under these settings. In case $(\mathcal{M},
w) \in Pmod_\Sigma \setminus Rmod_\Sigma$, a similar extension
might be still possible, but, generally speaking, it will not be
unique, since a given $a \in A_w$ may have more than one
$\mathbb{H}_{vw}$-preimage for a given $v\mathrel{\prec}w$, and
also $\mathcal{M}$ may consist of more than one $\prec$-connected
component. Even worse, if we make a wrong choice of $\bar{a}_n \in
A^n_w$ there might be no possible extensions of $\mathcal{M}$
making $\bar{a}_n$ the values of an $n$-tuple of fresh constants
at $w$. Consider the case when there is a $v \in W$ such that
$v\mathrel{\prec}w$, but $\{\bar{a}_n\}$ is disjoint from
$rang(\mathbb{H}_{vw})$. Therefore, in order to make a desired
extension possible, we have to allow the restriction of
$\mathcal{M}$ to a suitable submodel, provided that such a
restriction does not affect the $\prec$-successors of $w$ in
$\mathcal{M}$.

More precisely, we define the set $(\mathcal{M}, w)\oplus(\bar{c}_n/\bar{a}_n)$ of possible constant extensions
 of a given
$\mathcal{M}$ such that $(\mathcal{M}, w) \in Pmod_\Sigma$ for a
fresh tuple of constants $\bar{c}_n$ with the values $\bar{a}_n
\in A^n_w$ at $w$ as follows:
$$
(\mathcal{M}, w)\oplus(\bar{c}_n/\bar{a}_n) = \{\mathcal{N}\in
Mod_{\Sigma \cup \{\bar{c}_n\}}\mid [\mathcal{M}, w] \subseteq
\mathcal{N}\upharpoonright\Sigma \subseteq
\mathcal{M}\,\&\,J_w\langle\bar{c}_n\rangle = \bar{a}_n\}.
$$
It is easy to see that under this definition, the set
$(\mathcal{M}, w)\oplus(\bar{c}_n/\bar{a}_n)$ is always non-empty
since we will always have $([\mathcal{M},w], \bar{c}_n/\bar{a}_n)
\in (\mathcal{M},w)\oplus(\bar{c}_n/\bar{a}_n)$. Moreover, the
differences between the models in this set will only affect the
nodes in $W$ which are not $\prec$-accessible from $w$, so that
the following holds:
\begin{lemma}\label{L:cutoff}
$(\mathcal{M}, w) \in Pmod_\Sigma$, let $\bar{c}_n \notin \Sigma$,
let $\bar{a}_n \in A^n_w$, and let $\mathcal{N}\in
\mathcal{M}\oplus(\bar{c}_n/\bar{a}_n)$. Then we have:
$$
[\mathcal{N}, w] = ([\mathcal{M}, w], \bar{c}_n/\bar{a}_n).
$$
\end{lemma}
Another convenient  property of the generated submodels is that two
and more successive transitions to a generated submodel display a
sort of transitivity. We state this property, together with some other straightforward consequences of the definitions given thus far, in the following
lemma:
\begin{lemma}\label{L:cutoff-constants}
$(\mathcal{M}, w) \in Pmod_\Sigma$, let $\Theta \subseteq \Sigma$, let $n < \omega$, let
$\bar{c}_{n +1}$ be a tuple of pairwise distinct constants outside
$\Sigma$, let $\bar{a}_{n +1} \in A^{n +1}_w$, and let
$v\mathrel{\succ}w$ (or, equivalently, let $v \in [W,w]$). Then
the following statements hold:
\begin{enumerate}
\item $[[\mathcal{M}, w], v] = [\mathcal{M}, v]$.
    \item $
[([\mathcal{M}, w], \bar{c}_n/\bar{a}_n), v] = ([\mathcal{M}, v],
\bar{c}_n/\mathbb{H}_{wv}\langle\bar{a}_n\rangle)$.
\item $(([\mathcal{M}, w], \bar{c}_n/\bar{a}_n), c_{n +1}/a_{n
	+1}) = ([\mathcal{M}, w], \bar{c}_{n +1}/\bar{a}_{n +1}) = ([([\mathcal{M}, w], \bar{c}_n/\bar{a}_n), w], c_{n +1}/a_{n
	+1})$.
\item $([\mathcal{M}, w], \bar{c}_n/\bar{a}_n)\upharpoonright \Theta = [\mathcal{M}\upharpoonright \Theta, w]$.

\item $([\mathcal{M}, w], \bar{c}_n/\bar{a}_n)\upharpoonright(\Theta\cup \{\bar{c}_n\}) = ([\mathcal{M}\upharpoonright \Theta, w], \bar{c}_n/\bar{a}_n)$.
\end{enumerate}
\end{lemma}
All parts of this Lemma are rather straightforward, except maybe the second equation in Part 3. But note that this equation follows from the other parts of the Lemma by the following reasoning:
\begin{align*}
	([([\mathcal{M}, w], \bar{c}_n/\bar{a}_n), w], c_{n +1}/a_{n
		+1}) = (([\mathcal{M}, w], \bar{c}_n/\bar{a}_n), c_{n +1}/a_{n
		+1}) = ([\mathcal{M}, w], \bar{c}_{n +1}/\bar{a}_{n +1}),
\end{align*}
applying first Part 2 and the fact that $\mathbb{H}_{ww} = id_{A_w}$, and then the first equation of Part 3.

Let $(\bar{x}_n)^\frown(\bar{y}_m) \in Var^{\neq n + m}$, and let
$\varphi \in IL_{(\bar{x}_n)^\frown(\bar{y}_m)}(\Sigma)$. If
$\bar{c}_n$ is a tuple of individual constants, then we can
define, by the usual induction, the result
$\varphi\binom{\bar{c}_n}{\bar{x}_n}\in IL_{\bar{y}_m}(\Sigma \cup
\{\bar{c}_n\})$ of simultaneously replacing every free occurrence
of $x_i$ with an occurrence of $c_i$ for every $1 \leq i \leq n$.

Now we can give the following definition of satisfaction.

Let $(\mathcal{M}, w)$ be a pointed $\Theta$-model, let
$\varphi\in IL_{\bar{x}_n}(\Theta)$, let $\bar{a}_n \in A^n_w$ and
let $c, \bar{c}_n \notin Const_\Theta$ and be pairwise distinct. Then we
say that $\bar{a}_n$ $\mathsf{IL}$-satisfies $\varphi$ at
$(\mathcal{M}, w)$ and write $\mathcal{M}, w \models_{\mathsf{IL}}
\varphi[\bar{a}_n]$, where the $\models_{\mathsf{IL}}$ is defined
inductively as follows:
\begin{align*}
&\mathcal{M}, w \models_{\mathsf{IL}} P(\bar{d}_m)[\bar{a}_n]
\Leftrightarrow \langle I_w\langle\bar{d}_m\rangle\rangle
\in I_w(P) \,\,\,\,\, \,\,\,\,\,\,\,\,\,\,\,\,\,\,\,  \,\,\,\,\, \,\,\,\,\,\,\,\,\,\,\,\,\,\,\, \,\,\,\,\, \,\,\,\,\,\,\,\,\,\,\,\,\,\,\,  \,\,\,\,\, \,\,\,\,\,\, P \in \Theta_m,\, \bar{d}_m\in Const^m_\Theta;&&\\
&\mathcal{M}, w \models_{\mathsf{IL}} \varphi[\bar{a}_n]
\Leftrightarrow (\exists \mathcal{N} \in
(\mathcal{M},w)\oplus(\bar{c}_n/\bar{a}_n))(\mathcal{N}, w
\models_{\mathsf{IL}}
\varphi\binom{\bar{c}_n}{\bar{x}_n}[\bar{a}_n]) \,\,\,\,\, \,\,\,\,\,\,\,\,\,\,\,\,\,\,\,  \,\,\,\,\, \,\,\,\,\,\,\,\,\,\,\,\,\,\,\, \varphi\text{ atomic};
&&\\
&\mathcal{M}, w \models_{\mathsf{IL}} (\varphi \wedge
\psi)[\bar{a}_n] \Leftrightarrow
\mathcal{M}, w \models_{\mathsf{IL}}\varphi[\bar{a}_n]\textup{ and } \mathcal{M}, w \models_{\mathsf{IL}}\psi[\bar{a}_n];\\
&\mathcal{M}, w \models_{\mathsf{IL}} (\varphi \vee
\psi)[\bar{a}_n] \Leftrightarrow
\mathcal{M}, w \models_{\mathsf{IL}}\varphi[\bar{a}_n]\textup{ or } \mathcal{M}, w \models_{\mathsf{IL}}\psi[\bar{a}_n];\\
&\mathcal{M}, w \models_{\mathsf{IL}} (\varphi \to
\psi)[\bar{a}_n] \Leftrightarrow \forall v(w\prec v \Rightarrow
(\mathcal{M}, v \not\models_{\mathsf{IL}}
\varphi[\mathbb{H}_{wv}\langle\bar{a}_n\rangle]\textup{ or }
\mathcal{M}, v \models_{\mathsf{IL}}
\psi[\mathbb{H}_{wv}\langle\bar{a}_n\rangle]));\\
&\mathcal{M}, w \not\models_{\mathsf{IL}}\bot[\bar{a}_n];\\
&\mathcal{M}, w \models_{\mathsf{IL}}(\exists x\varphi)[\bar{a}_n]
\Leftrightarrow (\exists a \in A_w)(\exists \mathcal{N} \in
(\mathcal{M},w)\oplus(c/a))(\mathcal{N}, w \models_{\mathsf{IL}} \varphi\binom{c}{x}[\bar{a}_n]);\\
&\mathcal{M}, w \models_{\mathsf{IL}} (\forall
x\varphi)[\bar{a}_n] \Leftrightarrow (\forall v \succ w)(\forall a
\in A_v)(\exists \mathcal{N} \in
(\mathcal{M},v)\oplus(c/a))(\mathcal{N}, v \models_{\mathsf{IL}}
\varphi\binom{c}{x}[\mathbb{H}_{wv}\langle\bar{a}_n\rangle]).
\end{align*}
 We also define, as usual, for any given $\Gamma
\subseteq IL_{\bar{x}_n}(\Theta)$ that $\mathcal{M}, w
\models_{\mathsf{IL}} \Gamma[\bar{a}_n]$ iff $\mathcal{M}, w
\models_{\mathsf{IL}} \varphi[\bar{a}_n]$ for all $\varphi \in
\Gamma$. In case $n = 0$ we write $\mathcal{M}, w
\models_{\mathsf{IL}} \varphi$ (resp. $\mathcal{M}, w
\models_{\mathsf{IL}} \Gamma$) in place of $\mathcal{M}, w
\models_{\mathsf{IL}} \varphi[\Lambda]$ (resp. $\mathcal{M}, w
\models_{\mathsf{IL}} \Gamma[\Lambda]$).

It is clear that our definition of $\models_{\mathsf{IL}}$ is somewhat too involved for the
needs of intuitionistic logic alone; the reason for these
complications is that we want to place intuitionistic logic within
a more general framework applicable to many other systems, either
based on it or related to it. However, it is easy to infer some
simplifications of the above definition by making the following
observations:
\begin{lemma}\label{L:satisfaction}
Let $(\mathcal{M}, w)$ be a pointed $\Theta$-model, let
$\varphi\in IL_{\bar{x}_n}(\Theta)$, let $v \in [W,w]$, let
$\bar{a}_n \in A^n_w$, let $\bar{c}_n$ be a tuple of pairwise
distinct individual constants outside $\Theta$, let $\bar{t}_m$ be
$\Theta$-terms, and let $P \in \Theta_m$. Then we have:
\begin{enumerate}
    \item  $\mathcal{M}, w \models_{\mathsf{IL}} P(\bar{t}_m)[\bar{a}_n]
\Leftrightarrow \alpha\langle\bar{t}_m\rangle \in I_w(P)$, where
for $1 \leq i \leq m$ and $1 \leq j \leq n$ we set that:
$$
\alpha(t_i):= \left\{%
\begin{array}{ll}
    I_w(t_i), & \hbox{if $t_i$ is a constant;} \\
    a_j, & \hbox{if $t_i = x_j$.} \\
\end{array}%
\right.
$$
    \item $\mathcal{M}, w \models_{\mathsf{IL}}\varphi[\bar{a}_n]
    \Leftrightarrow [\mathcal{M}, w], w
    \models_{\mathsf{IL}}\varphi[\bar{a}_n]$.

    \item $\mathcal{M}, w \models_{\mathsf{IL}}\varphi[\bar{a}_n]
    \Rightarrow \mathcal{M}, v
    \models_{\mathsf{IL}}\varphi[\mathbb{H}_{wv}\langle\bar{a}_n\rangle]$.

    \item $\mathcal{M}, w \models_{\mathsf{IL}}\varphi[\bar{a}_n]
    \Leftrightarrow ([\mathcal{M},w], \bar{c}_n/\bar{a}_n), w
    \models_{\mathsf{IL}}\varphi\binom{\bar{c}_n}{\bar{x}_n}$.
\end{enumerate}
\end{lemma}
Besides the statement expressed in Lemma \ref{L:satisfaction}.4
(often termed `Substitution Lemma' in the literature), we will
need some further properties of substitutions which we formulate
below.
\begin{lemma}\label{L:substitution}
Let $\Theta$ be a signature, let $\bar{c}_n$ be a tuple of
pairwise distinct individual constants outside $\Theta$, let
$\varphi\in IL_{\bar{x}_m}(\Theta \cup \bar{c}_n)$, let $\Gamma
\subseteq IL_{\bar{x}_m}(\Theta \cup \bar{c}_n)$ and let
$\bar{y}_{n}, \bar{z}_{n} \in Var^{\neq {n}}$ be such that
$\{\bar{y}_{n}\} \subseteq Var\setminus (\{\bar{x}_m\} \cup
BV(\varphi))$ and $\{\bar{z}_{n}\} \subseteq Var\setminus
(\{\bar{x}_m\} \cup BV(\Gamma))$. Then all of the following hold:
\begin{enumerate}
\item There exists a $\psi \in  IL_{(\bar{x}_m)^\frown
    (\bar{y}_{n})}(\Theta)$ such that $\varphi =
    \psi\binom{\bar{c}_n}{\bar{y}_{n}}$.

    \item There exists a $\Delta \subseteq  IL_{(\bar{x}_m)^\frown
    (\bar{z}_{n})}(\Theta)$ such that $\Gamma =
    \Delta\binom{\bar{c}_n}{\bar{z}_{n}} = \{\psi\binom{\bar{c}_n}{\bar{z}_{n}}\mid \psi \in
    \Delta\}$.

    \item If, in addition, $\Gamma$ is finite, then $\Delta$ can
    also be chosen to be finite.
\end{enumerate}
\end{lemma}
\begin{definition}\label{D:renaming}
For any signatures $\Theta$ and $\Sigma$, assume that
$(\mathcal{M},w) \in Pmod_\Theta$ and $(\mathcal{N},w) \in
Pmod_\Sigma$ are such that $U = W$, $\prec = \lhd$, and $A_v =
B_v$ for all $v \in W$. Assume, further, that a couple of
bijections $f:Pred_\Theta \to Pred_\Sigma$ and $g: Const_\Theta
\to Const_\Sigma$ are given and that $f(\Theta_n) = \Sigma_n$ for
all $n < \omega$. Assume, finally, that for these bijections it is
true that, for every $v \in W$, every $c \in Const_\Theta$, every
$n < \omega$, and every $P \in \Theta_n$, we have $I_v(P) =
J_v(f(P))$ and $I_v(c) = J_v(g(c))$. Then $\Sigma$ is called an
$(f,g)$-\emph{renaming of }$\Theta$ and $(\mathcal{N},w)$ is
called an $(f,g)$-\emph{renaming of} $(\mathcal{M},w)$.
\end{definition}

\begin{lemma}\label{L:renaming}
For any signatures $\Theta$ and $\Sigma$ such that $\Sigma$ is an
$(f,g)$-renaming of $\Theta$, consider the mapping
$\tau_{(f,g)}:IL(\Theta) \to IL(\Sigma)$ defined inductively as
follows. We first set:
\begin{align*}
    t_g: = \left\{%
\begin{array}{ll}
    g(t), & \hbox{if $t \in Const_\Theta$;} \\
    t, & \hbox{otherwise.} \\
\end{array}%
\right.
\end{align*}
for any $\Theta$-term $t$ and then set:
\begin{align*}
    \tau_{(f,g)}(P(t_1,\ldots,t_n))&: =
    f(P)((t_1)_g,\ldots,(t_n)_g)&&P \in \Theta_n\\
    \tau_{(f,g)}(\bot)&: = \bot\\
    \tau_{(f,g)}(\varphi\circ\psi)&:=
    \tau_{(f,g)}(\varphi)\circ\tau_{(f,g)}(\psi)&&\circ\in\{\wedge, \vee,
    \to\}\\
    \tau_{(f,g)}(Qx\varphi)&:=
    Qx\tau_{(f,g)}(\varphi)&&x\in Var,\,Q\in\{\exists, \forall\}
\end{align*}
Then the following statements are true:
\begin{enumerate}
    \item $\tau_{(f,g)}$ is a bijection of $IL(\Theta)$ onto
    $IL(\Sigma)$.

    \item For any $\varphi \in IL(\Theta)$, $FV(\varphi) =
    FV(\tau_{(f,g)}(\varphi))$ and $BV(\varphi) =
    BV(\tau_{(f,g)}(\varphi))$.

    \item For any $\varphi \in IL(\Theta)$, if $\Theta_\varphi
    \subseteq \Theta$ is the smallest subsignature for which $\varphi \in
    IL(\Theta_\varphi)$, and $\Sigma_{\tau_{(f,g)}(\varphi)}
    \subseteq \Sigma$ is the smallest subsignature for which
    $\tau_{(f,g)}(\varphi)\in IL(\Sigma_{\tau_{(f,g)}(\varphi)})$,
    then for $f': = f\upharpoonright Pred_{\Theta_\phi}$ and $g': = g\upharpoonright
    Const_{\Theta_\varphi}$, it is true that
    $\Sigma_{\tau_{(f,g)}(\varphi)}$ is the $(f',g')$-renaming of
    $\Theta_\varphi$, and that $\tau_{(f,g)}(\varphi) =
    \tau_{(f',g')}(\varphi)$.

    \item If $(\mathcal{N},w) \in
Pmod_\Sigma$ is an $(f,g)$-renaming of $(\mathcal{M},w) \in
Pmod_\Theta$, $\varphi\in IL_{\bar{x}_n}(\Theta)$, and $\bar{a}_n
\in A^n_w$, then we have:
$$
\mathcal{M},w \models_{\mathsf{IL}} \varphi[\bar{a}_n]
\Leftrightarrow \mathcal{N},w \models_{\mathsf{IL}}
\tau_{(f,g)}(\varphi)[\bar{a}_n].
$$

\item If, moreover, $v \in [W,w]$ and $u \in [W, v]$, then
$([\mathcal{N},v], u) \in Pmod_\Sigma$ is an $(f,g)$-renaming of
$([\mathcal{M},v],u) \in Pmod_\Theta$.
\end{enumerate}
\end{lemma}
One corollary to Lemma \ref{L:satisfaction}.4, together with
Lemma \ref{L:renaming}.4, is that the choice of new constants
is to some extent irrelevant to the truth of formulas, provided
that these constants denote one and the same tuple of objects.

Another corollary to the above lemmas says that the generated
models can be put to a sort of universal use when evaluating
intuitionistic formulas:
\begin{corollary}\label{C:cutoff}
Let $\Theta$ be a signature, let $(\mathcal{M},w) \in
Pmod_\Theta$. Then for every $\bar{a}_n \in A^n_w$, $\bar{b}_m \in A^m_w$, and $\bar{x}_m \in Var^m$, every tuple
$\bar{c}_n$ of pairwise distinct constants outside $\Theta$, and
every  $\varphi \in IL_{\bar{x}_m}(\Theta \cup \{\bar{c}_n\})$  we have:
\begin{align*}
(\exists\mathcal{N}\in (\mathcal{M},
w)\oplus(\bar{c}_n/\bar{a}_n))(\mathcal{N}, w
\models_{\mathsf{IL}}\varphi[\bar{b}_m])&\Leftrightarrow
(\forall\mathcal{N}\in (\mathcal{M},
w)\oplus(\bar{c}_n/\bar{a}_n))(\mathcal{N}, w
\models_{\mathsf{IL}}\varphi[\bar{b}_m])\\
&\Leftrightarrow ([\mathcal{M},w], \bar{c}_n/\bar{a}_n),
w\models_{\mathsf{IL}}\varphi[\bar{b}_m]
\end{align*}
\end{corollary}
\begin{proof}
Since it is always the case that $([\mathcal{M},w],
\bar{c}_n/\bar{a}_n) \in
(\mathcal{M},w)\oplus(\bar{c}_n/\bar{a}_n)$, it follows trivially
that
\begin{align*}
(\forall\mathcal{N}\in (\mathcal{M},
w)\oplus(\bar{c}_n/\bar{a}_n))(\mathcal{N}, w
\models_{\mathsf{IL}}\varphi[\bar{b}_m])&\Rightarrow([\mathcal{M},w],
\bar{c}_n/\bar{a}_n), w\models_{\mathsf{IL}}\varphi[\bar{b}_m]\\
&\Rightarrow (\exists\mathcal{N}\in (\mathcal{M},
w)\oplus(\bar{c}_n/\bar{a}_n))(\mathcal{N}, w
\models_{\mathsf{IL}}\varphi[\bar{b}_m])
\end{align*}
Now, if $\mathcal{N}\in (\mathcal{M},
w)\oplus(\bar{c}_n/\bar{a}_n)$ is such that $\mathcal{N}, w
\models_{\mathsf{IL}}\varphi$, then let $\mathcal{N}'\in
(\mathcal{M}, w)\oplus(\bar{c}_n/\bar{a}_n)$ be chosen
arbitrarily. We reason as follows:
\begin{align*}
\mathcal{N}, w \models_{\mathsf{IL}}\varphi[\bar{b}_m] &\Leftrightarrow
[\mathcal{N}, w], w\models_{\mathsf{IL}}\varphi[\bar{b}_m] &&\text{(by Lemma
\ref{L:satisfaction}.2)}\\
&\Leftrightarrow ([\mathcal{M}, w], \bar{c}_n/\bar{a}_n),
w\models_{\mathsf{IL}}\varphi[\bar{b}_m] &&\text{(by Lemma
\ref{L:cutoff})}\\
&\Leftrightarrow [\mathcal{N}', w], w,\models_{\mathsf{IL}}\varphi[\bar{b}_m]
&&\text{(by Lemma
\ref{L:cutoff})}\\
&\Leftrightarrow \mathcal{N}', w
\models_{\mathsf{IL}}\varphi[\bar{b}_m]&&\text{(by Lemma
\ref{L:satisfaction}.2)}
\end{align*}
\end{proof}
The latter corollary allows to somewhat simplify the clauses for
the intuitionistic quantifiers bringing them closer to the
familiar pattern:
\begin{corollary}\label{C:satisfaction}
Let $(\mathcal{M}, w)$ be a pointed $\Theta$-model, let $c$ be a constant outside $\Theta$, let
$\varphi\in IL_{\bar{x}_n^\frown x}(\Theta)$, and let $\bar{a}_n
\in A^n_w$. Then we have:
\begin{align*}
&\mathcal{M}, w \models_{\mathsf{IL}}(\exists x\varphi)[\bar{a}_n]
\Leftrightarrow (\exists a \in A_w)(([\mathcal{M}, w], c/a), w \models_{\mathsf{IL}} \varphi\binom{c}{x}[\bar{a}_n]);\\
&\mathcal{M}, w \models_{\mathsf{IL}} (\forall
x\varphi)[\bar{a}_n] \Leftrightarrow (\forall v \succ w)(\forall a
\in A_v)(([\mathcal{M}, v], c/a), v \models_{\mathsf{IL}}
\varphi\binom{c}{x}[\mathbb{H}_{wv}\langle\bar{a}_n\rangle]).
\end{align*}
\end{corollary}
We may bring the simplification one step further in the particular case of the generated models:
\begin{corollary}\label{C:satisfaction-cutoff}
	Let $(\mathcal{M}, w)$ be a pointed $\Theta$-model, let $\bar{b}_m
	\in A^m_w$, let $\bar{c}_{m + 1}$ be a tuple of pairwise distinct individual constants outside $\Theta$, let
	$\varphi\in IL_{(\bar{x}_n)^\frown x}(\Theta \cup \{\bar{c}_{m}\})$, and let $\bar{a}_n
	\in A^n_w$. Then we have:
	\begin{align*}
		&([\mathcal{M}, w], \bar{c}_m/\bar{b}_m), w\models_{\mathsf{IL}}(\exists x\varphi)[\bar{a}_n]
		\Leftrightarrow (\exists b \in A_w)(([\mathcal{M}, w], \bar{c}_{m + 1}/(\bar{b}_m)^\frown b), w \models_{\mathsf{IL}} \varphi\binom{c_{m + 1}}{x}[\bar{a}_n]);\\
		&([\mathcal{M}, w], \bar{c}_m/\bar{b}_m), w \models_{\mathsf{IL}} (\forall
		x\varphi)[\bar{a}_n] \Leftrightarrow (\forall v \succ w)(\forall b
		\in A_v)(([\mathcal{M}, v], \bar{c}_{m + 1}/\mathbb{H}_{wv}\langle\bar{b}_m\rangle^\frown b), v \models_{\mathsf{IL}}
		\varphi\binom{c_{m + 1}}{x}[\mathbb{H}_{wv}\langle\bar{a}_n\rangle]).
	\end{align*}
\end{corollary}
\begin{proof}
	(Part 1). We have $([\mathcal{M}, w], \bar{c}_m/\bar{b}_m), w\models_{\mathsf{IL}}(\exists x\varphi)[\bar{a}_n]$ iff, by Corollary \ref{C:satisfaction}.1, for some $b \in A_w$ we have $([([\mathcal{M}, w], \bar{c}_m/\bar{b}_m), w], c_{m + 1}/b), w \models_{\mathsf{IL}} \varphi\binom{c_{m + 1}}{x}[\bar{a}_n]$. But we also have:
	$$
	([([\mathcal{M}, w], \bar{c}_m/\bar{b}_m), w], c_{m + 1}/b) = ([\mathcal{M}, w], \bar{c}_{m + 1}/(\bar{b}_m)^\frown b)	
	$$
	by Lemma \ref{L:cutoff-constants}.3, so the Corollary follows.
	
	(Part 2). We have $([\mathcal{M}, w], \bar{c}_m/\bar{b}_m), w \models_{\mathsf{IL}} (\forall
	x\varphi)[\bar{a}_n]$ iff for every $v \succ w$ and every $ b
	\in A_v$ we have $([([\mathcal{M}, w], \bar{c}_m/\bar{b}_m), v], c_{m + 1}/b), v \models_{\mathsf{IL}} \varphi\binom{c_{m + 1}}{x}[\mathbb{H}_{wv}\langle\bar{a}_n\rangle]$. But we also have:
	\begin{align*}
		([([\mathcal{M}, w], \bar{c}_m/\bar{b}_m), v], c_{m + 1}/b) &= (([\mathcal{M}, v], \bar{c}_m/\mathbb{H}_{wv}\langle\bar{b}_m\rangle), c_{m + 1}/b) &&\text{by Lemma \ref{L:cutoff-constants}.2}\\
		&= ([\mathcal{M}, v], \bar{c}_{m + 1}/\mathbb{H}_{wv}\langle\bar{b}_m\rangle^\frown b)&&\text{by Lemma \ref{L:cutoff-constants}.3}
	\end{align*}
     Thus the Corollary follows.
\end{proof}

Moreover, it is clear that renaming of bound variables in an
intuitionistic formula results in an equivalent formula.
Therefore, we can adopt the following notational convention
allowing us to directly quantify over constants.
\begin{definition}\label{D:constant-quantification}
Let $\Theta$ be a signature, let $c \notin Const_\Theta$, let
$\bar{x}_n \in Var^{\neq n}$ and let $\varphi \in
IL_{\bar{x}_n}(\Theta \cup\{c\})$. Then we denote by $\exists
c\varphi \in IL_{\bar{x}_n}(\Theta)$ (resp. $\forall c\varphi \in
IL_{\bar{x}_n}(\Theta)$) any formula $\exists x\psi$ (resp.
$\forall x\psi$), where $x \in Var \setminus (BV(\varphi) \cup
\{\bar{x}_n\})$ and $\psi \in IL_{(\bar{x}_n)^\frown x}(\Theta)$ is
such that $\psi\binom{c}{x} = \varphi$.
\end{definition}
It follows then from Lemma \ref{L:substitution}.1 that $\exists
c\varphi \in IL_{\bar{x}_n}(\Theta)$ (resp. $\forall c\varphi \in
IL_{\bar{x}_n}(\Theta)$) is defined for every $\varphi \in
IL_{\bar{x}_n}(\Theta \cup\{c\})$ and it follows from Corollary
\ref{C:satisfaction} that whenever $\chi$ and $\theta$ are
different representatives of $\exists c\varphi$ (resp. $\forall
c\varphi$), then for every $(\mathcal{M}, w) \in Pmod_\Theta$ and
every $\bar{a}_n \in A^n_w$ we will have:
$$
\mathcal{M}, w \models_{\mathsf{IL}}
\chi[\bar{a}_n]\Leftrightarrow \mathcal{M}, w
\models_{\mathsf{IL}} \theta[\bar{a}_n],
$$
so, semantically $\chi$ and $\theta$ are one and the same formula.
In addition, we can formulate the following variants of Corollary
\ref{C:satisfaction} and  Corollary
\ref{C:satisfaction-cutoff} for quantifiers over constants:
\begin{corollary}\label{C:satisfaction-constants}
Let $(\mathcal{M}, w)$ be a pointed $\Theta$-model,  let $c \notin
Const_\Theta$, let $\varphi\in IL_{\bar{x}_n}(\Theta\cup \{c\})$,
and let $\bar{a}_n \in A^n_w$. Then we have:
\begin{align*}
&\mathcal{M}, w \models_{\mathsf{IL}}(\exists c\varphi)[\bar{a}_n]
\Leftrightarrow (\exists a \in A_w)(([\mathcal{M}, w], c/a), w \models_{\mathsf{IL}} \varphi[\bar{a}_n]);\\
&\mathcal{M}, w \models_{\mathsf{IL}} (\forall
c\varphi)[\bar{a}_n] \Leftrightarrow (\forall v \succ w)(\forall a
\in A_v)(([\mathcal{M}, v], c/a), v \models_{\mathsf{IL}}
\varphi[\mathbb{H}_{wv}\langle\bar{a}_n\rangle]).
\end{align*}
\end{corollary}
\begin{corollary}\label{C:satisfaction-cutoff-constants}
	Let $(\mathcal{M}, w)$ be a pointed $\Theta$-model, let $\bar{b}_m
	\in A^m_w$, let $\bar{c}_{m + 1}$ be a tuple of pairwise distinct individual constants outside $\Theta$, let
	$\varphi\in IL_{\bar{x}_n}(\Theta\cup \{\bar{c}_{m + 1}\})$, and let $\bar{a}_n\in A^n_w$. Then we have:
	\begin{align*}
		&([\mathcal{M}, w], \bar{c}_m/\bar{b}_m), w\models_{\mathsf{IL}}(\exists c_{m + 1}\varphi)[\bar{a}_n]
		\Leftrightarrow (\exists b \in A_w)(([\mathcal{M}, w], \bar{c}_{m + 1}/(\bar{b}_m)^\frown b), w \models_{\mathsf{IL}} \varphi[\bar{a}_n]);\\
		&([\mathcal{M}, w], \bar{c}_m/\bar{b}_m), w \models_{\mathsf{IL}} (\forall
		c_{m + 1}\varphi)[\bar{a}_n] \Leftrightarrow (\forall v \succ w)(\forall b
		\in A_v)(([\mathcal{M}, v], \bar{c}_{m + 1}/\mathbb{H}_{wv}\langle\bar{b}_m\rangle^\frown b), v \models_{\mathsf{IL}}
		\varphi[\mathbb{H}_{wv}\langle\bar{a}_n\rangle]).
	\end{align*}
\end{corollary}
We now adopt the standard definitions for the notions of
satisfiable and valid formula (resp. sentence) in any given class
of (pointed) models, including the singleton sets of models and
the class of all models. Thus, for any class $Cls \subseteq Mod$
and for any $\varphi \in IL_\emptyset(\Theta)$, we write $Cls
\models_{\mathsf{IL}} \varphi$ iff $\varphi$ is valid in every model in the
class $Pmod_\Theta(Cls):= \{ (\mathcal{M}, w) \in Pmod_\Theta\mid \mathcal{M} \in
Cls\}$. In particular, for a $\Theta$-model $\mathcal{M}$ and a
$\varphi \in IL_\emptyset(\Theta)$, we will write $\mathcal{M}
\models_{\mathsf{IL}} \varphi$ (resp. $\models_{\mathsf{IL}}
\varphi$) iff $\varphi$ is valid in $\{ \mathcal{M}\}$ (resp. in
$Mod_\Theta$).

We will distinguish the following subclasses of Kripke models:
\begin{itemize}
	\item $In$, the class of all models $\mathcal{M}$ such that for all $w,v \in W$, if $w \prec v$, then $\mathbb{H}_{wv}$ is an injection (one-to-one function).
	
	\item $Su$, the class of all models $\mathcal{M}$ such that for all $w,v \in W$, if $w \prec v$, then $\mathbb{H}_{wv}$ is a surjection ($rang(\mathbb{H}_{wv}) = A_v$).
	
	\item $Bi := In \cap Su$.
\end{itemize}
If we restrict our attention to the models in $Bi$, where all the canonical homomorphisms are bijections, we obtain $\mathsf{CD}$, the intuitionistic logic of constant domains.

On the other hand, we observe that in the context of equality-free first-order language, the presence or absence of injectivity of canonical isomorphisms does not affect the validity and satisfiability of formulas. Therefore, if the validity and satisfiability is all we care about when defining our logics, we might just as well interpret $\mathsf{CD}$ on the basis of class $Su$ of Kripke models; symmetrically, we could have assumed $In$ as the class of intended models for our basic intuitionistic logic $\mathsf{IL}$. The following lemma explains the matter in a bit more detail:
\begin{lemma}\label{L:injectivity}
Let $\Theta$ be a signature, let $n < \omega$, and let $\varphi \in IL_{\bar{x}_n}(\Theta)$. Then we have:
\begin{enumerate}
	\item $\models_{\mathsf{IL}}
	\varphi$ iff $In \models_{\mathsf{IL}}
	\varphi$.
	
	\item $Bi \models_{\mathsf{IL}}
	\varphi$ iff $Su \models_{\mathsf{IL}}
	\varphi$. 
\end{enumerate}
\end{lemma}
\begin{proof}[Proof (a sketch)] For a given signature $\Theta$, and a given $\mathcal{M} \in Mod_\Theta$, if $w \in W$ and $a \in A_w$, we define the set of nodes $(w\downarrow a) : = \{v \mid v\mathrel{\prec} w\,\&\,(\exists b \in A_v)(\mathbb{H}_{vw}(b) = a)\}$ and the following set of functions:
	$$
f_{a} =  \{ f:(w\downarrow a) \to \bigcup_{v \in (w\downarrow)}A_v \mid (\forall v \in (w\downarrow a))(f(v)\in A_v\,\&\,\mathbb{H}_{vw}(f(v)) = a) \}.	
	$$
We then consider $\Theta$-model $\mathcal{M}' = \langle W, \prec,\mathfrak{A}', \mathbb{H}'\rangle$ such that for all $w,v \in W$ we have:
	\begin{align*}
		\mathfrak{A}'_w &:= (A'_w, I'_w)\\
		A'_w &:= \bigcup\{ f_{a}\mid  a \in A_w \}\\
		I'_w(P) :&= \{ \bar{f}_n\mid \langle f_1(w),\ldots f_n(w)\rangle\in I_w(P)\} &&\text{$P \in \Theta_n$}\\
		I'_w(c) &:= f \Leftrightarrow (\forall v \in (w\downarrow))(f(v) = I_v(c)) &&\text{$c \in Const_\Theta$}\\
		\mathbb{H}'_{vw}(f) &:= g \Leftrightarrow g \supseteq f\cup \{(w, \mathbb{H}_{vw}(f(v)))\}&& v\mathrel{\prec} w
	\end{align*}
It is straightforward (although somewhat tiresome) to show by induction on the construction of a formula, that $\mathcal{M}' \in In$ and that, given an $n < \omega$, for every $w \in W$, for all $\bar{a}_n$, for all $\bar{f}_n \in A'_w$ such that $\bar{a}_n = \langle f_1(w),\ldots f_n(w)\rangle$ for every $\varphi \in IL_{\bar{x}_n}(\Theta)$, we have:
$$
\mathcal{M}, w \models_{\mathsf{IL}} \varphi[\bar{a}_n]\Leftrightarrow \mathcal{M}', w \models_{\mathsf{IL}} \varphi[\bar{f}_n].
$$
(This is due to the fact that our language is lacking equality, otherwise formulas like $x_1 \equiv c$ would provide an easy counterexample to our claim).

Moreover, it is easy to verify that for an $\mathcal{M} \in Su$ we also have $\mathcal{M}' \in Su$, whence Part 2 of our Lemma follows.
\end{proof}
For the sake of uniformity, we need to either omit or enforce injectivity for the default classes of intended models assigned to both $\mathsf{IL}$ and $\mathsf{CD}$, and since we did not enforce the injectivity for $\mathsf{IL}$, we will not do this for $\mathsf{CD}$ either, even though this decision gives $\mathsf{CD}$ the semantics of non-constant models as its default semantics. We thus set that $Cls
\models_{\mathsf{CD}} \varphi$ iff $Cls \cap Su
\models_{\mathsf{IL}} \varphi$ and that $\models_{\mathsf{CD}} \varphi$ iff $Su
\models_{\mathsf{IL}} \varphi$. In situations when we will need to distinguish between two logics with the same sets of theorems but with different intended semantics, we will denote by $\mathsf{In}$ the logic of injective Kripke models and by $\mathsf{Bi}$ the logic of bijective Kripke models.

We note, further, that $\mathsf{CD}$ is different
from $\mathsf{IL}$, since we have, e.g. $\models_{\mathsf{CD}} \forall x(\varphi \vee \psi) \to (\varphi
\vee \forall x\psi)$, provided that $x \notin FV(\varphi)$, but
$\not\models_{\mathsf{IL}} \forall x(\varphi \vee \psi) \to
(\varphi \vee \forall x\psi)$.

\subsection{Intuitionistic theories and intuitionistic types}
Since classical negation is not available in intuitionistic logic,
it often makes sense to include falsehood assumptions along with
truth assumptions when defining theories in intuitionistic logic
(this is done in e.g. \cite[p. 110]{GabbayMaksimova}; \cite{chagrov}). Thus, given
a $(\mathcal{M}, w) \in Pmod_\Theta$, we define
$Th_{\mathsf{IL}}(\mathcal{M}, w)$, the complete
$\mathsf{IL}$-theory of $(\mathcal{M}, w)$, to be the following
pair:
\begin{align*}
Th_{\mathsf{IL}}(\mathcal{M}, w) := (\{\varphi\in
IL_{\emptyset}(\Theta) \mid \mathcal{M}, w\models_{\mathsf{IL}}
\varphi \},\{\varphi\in IL_{\emptyset}(\Theta) \mid \mathcal{M},
w\not\models_{\mathsf{IL}} \varphi \})
\end{align*}
We also introduce a special notation for the left and right
projection of $Th_{\mathsf{IL}}(\mathcal{M}, w)$, that is to say,
for the \emph{positive} and for the \emph{negative} part of this
theory, denoting them by $Th^+_{\mathsf{IL}}(\mathcal{M}, w)$ and
$Th^-_{\mathsf{IL}}(\mathcal{M}, w)$, respectively.

Since we are going to put intuitionistic logic in the context of
abstract model theory which typically avoids mentioning individual
variables explicitly, it is convenient to construe intuitionistic
types as a variant of intuitionistic theories, allowing fresh
individual constants to take on the role of free variables of a
given type. In this fashion, given a $(\mathcal{M}, w) \in
Pmod_\Theta$, a tuple $\bar{a}_n \in A^n_w$, and a tuple
$\bar{c}_n$ of pairwise distinct individual constants outside
$\Theta$, we define $Tp_{\mathsf{IL}}(\mathcal{M}, w,
\bar{c}_n/\bar{a}_n)$, the $\bar{c}_n$-complete $\mathsf{IL}$-type
of $(\mathcal{M}, w, \bar{a}_n)$, as
$Th_{\mathsf{IL}}(([\mathcal{M}, w], \bar{c}_n/\bar{a}_n),w)$. It
follows then from Corollary \ref{C:cutoff} that a complete type is
exactly what holds or fails in all (or any) possible extensions of
a model by a tuple of constants with a given tuple of denotations:
\begin{corollary}\label{C:types-and-constants}
Let $\Theta$ be signature, let $(\mathcal{M}, w) \in Pmod_\Theta$,
let $\bar{a}_n \in A^n_w$, let $\bar{c}_n$ be a tuple of pairwise
distinct individual constants outside $\Theta$, and let $\varphi
\in IL_\emptyset(\Theta \cup \{\bar{c}_n\})$. Then:
\begin{enumerate}
    \item $\varphi \in Tp^+_{\mathsf{IL}}(\mathcal{M}, w,
\bar{c}_n/\bar{a}_n) \Leftrightarrow (\exists\mathcal{N}\in
(\mathcal{M}, w)\oplus(\bar{c}_n/\bar{a}_n))(\mathcal{N}, w
\models_{\mathsf{IL}}\varphi)\Leftrightarrow
(\forall\mathcal{N}\in (\mathcal{M},
w)\oplus(\bar{c}_n/\bar{a}_n))(\mathcal{N}, w
\models_{\mathsf{IL}}\varphi)$.
    \item $\varphi \in Tp^-_{\mathsf{IL}}(\mathcal{M}, w,
\bar{c}_n/\bar{a}_n) \Leftrightarrow (\exists\mathcal{N}\in
(\mathcal{M}, w)\oplus(\bar{c}_n/\bar{a}_n))(\mathcal{N}, w
\not\models_{\mathsf{IL}}\varphi)\Leftrightarrow
(\forall\mathcal{N}\in (\mathcal{M},
w)\oplus(\bar{c}_n/\bar{a}_n))(\mathcal{N}, w
\not\models_{\mathsf{IL}}\varphi)$.
\end{enumerate}
\end{corollary}

More generally, we define the inclusion of pairs of sets of
formulas componentwise, and say that if
$(\Gamma, \Delta)$ and $(\Gamma', \Delta')$ are two such pairs
then we have
\[
(\Gamma, \Delta) \subseteq (\Gamma', \Delta') \Leftrightarrow
\Gamma \subseteq \Gamma'\,\&\, \Delta \subseteq \Delta'.
\]
In a similar way, we will use pair of sets in connection with the
relation of $\mathsf{IL}$-satisfaction, and, given a
$(\mathcal{M}, w) \in Pmod_\Theta$, a tuple $\bar{a}_n \in A^n_w$,
and a pair $(\Gamma, \Delta) \subseteq (IL_{\bar{x}_n}(\Theta),
IL_{\bar{x}_n}(\Theta))$, we will say that $(\Gamma, \Delta)$ is
satisfied at $(\mathcal{M}, w)$ by $\bar{a}_n$ (and write
$\mathcal{M}, w\models_{\mathsf{IL}} (\Gamma, \Delta)[\bar{a}_n]$)
iff it is true that (1) $\mathcal{M}, w\models_{\mathsf{IL}}
\Gamma[\bar{a}_n]$ and (2) $(\forall\varphi \in
\Delta)(\mathcal{M}, w\not\models_{\mathsf{IL}}
\varphi[\bar{a}_n])$. In case $\Gamma = \emptyset$, we will say
that $\Delta$ is falsified at $(\mathcal{M}, w)$ by $\bar{a}_n$.

In the particular case when $n = 0$, we immediately see
that, for every $(\mathcal{M}, w) \in Pmod_\Theta$ and every
$(\Gamma, \Delta) \subseteq (IL_{\emptyset}(\Theta),
IL_{\emptyset}(\Theta))$, we have $\mathcal{M},
w\models_{\mathsf{IL}} (\Gamma, \Delta)$ iff $(\Gamma, \Delta)
\subseteq Th_{\mathsf{IL}}(\mathcal{M}, w)$. More generally, for every $\bar{x}_n \in Var^{\neq n}$, every $\bar{a}_n \in A_w^{n}$ and every tuple
$\bar{c}_n$ of pairwise distinct individual constants outside
$\Theta$, if $(\Gamma, \Delta) \subseteq (IL_{\bar{x}_n}(\Theta),
IL_{\bar{x}_n}(\Theta))$, then we have $\mathcal{M},
w\models_{\mathsf{IL}} (\Gamma, \Delta)[\bar{a}_n]$ iff $(\Gamma, \Delta)\binom{\bar{c}_n}{\bar{x}_n}
\subseteq Tp_{\mathsf{IL}}(\mathcal{M}, w, \bar{c}_n/\bar{a}_n)$.

If $\mathcal{M} \subseteq \mathcal{N}$ are two $\Theta$-models and
for all $w \in W$, all $\bar{a}_n \in A^n_w$  and every
(equivalently, any) tuple $\bar{c}_n$ of pairwise distinct
individual constants outside $\Theta$ it is true that
$Tp_{\mathsf{IL}}(\mathcal{M}, w, \bar{c}_n/\bar{a}_n) =
Tp_{\mathsf{IL}}(\mathcal{N}, w, \bar{c}_n/\bar{a}_n)$, then we
say that $\mathcal{M}$ is an $\mathsf{IL}$\emph{-elementary
submodel} of $\mathcal{N}$ and write $\mathcal{M}
\preccurlyeq_{\mathsf{IL}} \mathcal{N}$. All of the notions defined in this subsection thus far are robust in the sense that they can be shown to interact with reducts to a subsignature in a very regular and predictable way. More precisely, the following lemma holds:
\begin{lemma}\label{L:el-submodels}
	Let $\Theta$ be a signature and let $\Sigma \subseteq \Theta$. For a given $n < \omega$, any $\mathcal{M}, \mathcal{N} \in Mod_\Theta$, $w \in W$, and $\bar{a}_n \in A^n_w$ and any tuple $\bar{c}_n$ of pairwise distinct individual constants outside $\Theta$ it is true that:
	\begin{enumerate}
		\item $Th_{\mathsf{IL}}(\mathcal{M}, w)\cap IL_\emptyset(\Sigma) = Th_{\mathsf{IL}}(\mathcal{M}\upharpoonright\Sigma, w)$.
		
		\item $Tp_{\mathsf{IL}}(\mathcal{M}, w, \bar{c}_n/\bar{a}_n) \cap IL_\emptyset(\Sigma\cup \{\bar{c}_n\}) = Tp_{\mathsf{IL}}(\mathcal{M}\upharpoonright\Sigma, w, \bar{c}_n/\bar{a}_n)$.
		
		\item If $\mathcal{M}
		\preccurlyeq_{\mathsf{IL}} \mathcal{N}$, then $(\mathcal{M}\upharpoonright\Sigma)
		\preccurlyeq_{\mathsf{IL}} (\mathcal{N}\upharpoonright\Sigma)$.
	\end{enumerate}
\end{lemma}
\begin{proof}
	Part 1 is immediate from the definition, as for Part 2, we reason as follows:
	\begin{align*}
	Tp_{\mathsf{IL}}(\mathcal{M}, w, \bar{c}_n/\bar{a}_n) \cap IL_\emptyset(\Sigma\cup \{\bar{c}_n\}) &= Th_{\mathsf{IL}}(([\mathcal{M}, w], \bar{c}_n/\bar{a}_n),w) \cap IL_\emptyset(\Sigma\cup \{\bar{c}_n\})\\
	&= Th_{\mathsf{IL}}(([\mathcal{M}, w], \bar{c}_n/\bar{a}_n)\upharpoonright(\Sigma\cup \{\bar{c}_n\}),w) &&\text{(by Part 1)}\\
	&= Th_{\mathsf{IL}}(([\mathcal{M}\upharpoonright\Sigma, w], \bar{c}_n/\bar{a}_n),w) &&\text{(by Lemma \ref{L:cutoff-constants}.5)}\\
	&= Tp_{\mathsf{IL}}(\mathcal{M}\upharpoonright\Sigma, w, \bar{c}_n/\bar{a}_n).
	\end{align*}
Finally, as for Part 3, it is immediate that if $\mathcal{M} \subseteq \mathcal{N}$, then also $(\mathcal{M}\upharpoonright\Sigma)
\subseteq (\mathcal{N}\upharpoonright\Sigma)$. It remains to show the coincidence of complete types in the reducts of $\mathcal{M}$ and $\mathcal{N}$. We reason as follows:
		\begin{align*}
		Tp_{\mathsf{IL}}(\mathcal{M}\upharpoonright\Sigma, w, \bar{c}_n/\bar{a}_n) &= Tp_{\mathsf{IL}}(\mathcal{M}, w, \bar{c}_n/\bar{a}_n) \cap IL_\emptyset(\Sigma\cup \{\bar{c}_n\})&&\text{(by Part 2)}\\
		&= Tp_{\mathsf{IL}}(\mathcal{N}, w, \bar{c}_n/\bar{a}_n) \cap IL_\emptyset(\Sigma\cup \{\bar{c}_n\}) &&\text{(by $\mathcal{M}
			\preccurlyeq_{\mathsf{IL}} \mathcal{N}$)}\\
		&= Tp_{\mathsf{IL}}(\mathcal{N}\upharpoonright\Sigma, w, \bar{c}_n/\bar{a}_n)&&\text{(by Part 2)}
	\end{align*}
Thus we get that $(\mathcal{M}\upharpoonright\Sigma)
\preccurlyeq_{\mathsf{IL}} (\mathcal{N}\upharpoonright\Sigma)$.	
\end{proof}

In what follows, we will also need the concept of an
$\mathsf{IL}$-elementary embedding defined as follows:
\begin{definition}\label{D:embedding}
{\em Let $\mathcal{M}$, $\mathcal{N}$ be $\Theta$-models. A pair
of functions $(g,h)$ such that $g: \mathbb{A} \to \mathbb{B}$ and
$h:W \to U$ is called an \emph{ $\mathsf{IL}$-elementary embedding of
$\mathcal{M}$ into $\mathcal{N}$}
 iff $g$, $h$ are injective, and
for all $v,u \in W$, all $n < \omega$, all $\bar{a}_n \in
A^n_{v}$, and every (equivalently, any) tuple $\bar{c}_n$ of
pairwise distinct individual constants outside $\Theta$ it is true
that

\begin{align}
&v\mathrel{\prec}u \Leftrightarrow h(v)\mathrel{\lhd}h(u)
\label{E:c1}\tag{\text{rel}}\\
&(g\upharpoonright A_v):A_v \to B_{h(v)}\label{E:c2}\tag{\text{dom}}\\
&v\mathrel{\prec}u \Rightarrow (g(\mathbb{H}_{vu}(a_1)) = \mathbb{G}_{h(v)h(u)}(g(a_1)))\label{E:c2a}\tag{\text{map}}\\
 &Tp_{\mathsf{IL}}(\mathcal{M}, v,
\bar{c}_n/\bar{a}_n) = Tp_{\mathsf{IL}}(\mathcal{N}, h(v),
\bar{c}_n/g\langle\bar{a}_n\rangle)\label{E:c3}\tag{\text{c-types}}
\end{align}
 }
\end{definition}
We collect some of the easy consequences of the latter definition in the following lemma:
\begin{lemma}\label{L:embedding}
	Let $\mathcal{M}$, $\mathcal{N}$ be $\Theta$-models. Then the following statements hold:
	\begin{enumerate}
		\item If $\mathcal{M}
		\preccurlyeq_{\mathsf{IL}} \mathcal{N}$, then the pair $(id_\mathbb{A},
		id_W)$ is an $\mathsf{IL}$-elementary embedding of $\mathcal{M}$ into $\mathcal{N}$.
		\item If $(g,h):\mathcal{M}\cong\mathcal{N}$, then $(g,h)$ is an $\mathsf{IL}$-elementary embedding of $\mathcal{M}$ into $\mathcal{N}$; in particular, for any $w \in W$, all $n < \omega$, all $\bar{a}_n \in
		A^n_{w}$,
		 and every (equivalently, any) tuple $\bar{c}_n$ of
		pairwise distinct individual constants outside $\Theta$ we have $$Tp_{\mathsf{IL}}(\mathcal{M}, w,
		\bar{c}_n/\bar{a}_n) = Tp_{\mathsf{IL}}(\mathcal{N}, h(w),
		\bar{c}_n/g\langle\bar{a}_n\rangle).$$
		
		\item A pair of functions $(g,h)$ is an
		$\mathsf{IL}$-elementary embedding of $\mathcal{M}$ into $\mathcal{N}$ iff
		there exists a (unique) $\mathcal{N}' \preccurlyeq_{\mathsf{IL}}
		\mathcal{N}$ such that $(g,h): \mathcal{M} \cong \mathcal{N}'$. We will denote $\mathcal{N}'$ by $(g,h)(\mathcal{M})$.
	\end{enumerate}
\end{lemma}

For the later parts of this paper we will also need the following intuitionistic version of a well-known principle from classical model theory (sometimes called the principle of isomorphic correction):
\begin{lemma}\label{L:isomorphic-correction}
	Let $\mathcal{M}$, $\mathcal{N}$ be $\Theta$-models, and let
	functions $g$, $h$ be such that $(g,h)$ is an
	$\mathsf{IL}$-elementary embedding of $\mathcal{M}$ into $\mathcal{N}$. Then
	there exists a $\Theta$-model $\mathcal{M}'$ and functions $g',
	h'$, such that we have:
	\begin{enumerate}
		\item $\mathcal{M} \preccurlyeq_{\mathsf{IL}} \mathcal{M}'$.
		\item $g \subseteq g'$ and $h \subseteq h'$.
		\item $(g',h'):\mathcal{M}' \cong
		\mathcal{N}$.
	\end{enumerate}
\end{lemma}
\begin{proof}[Proof (a sketch)] As in the classical case, we
	choose two disjoint sets $C$ and $D$ outside $W \cup U \cup
	\mathbb{A} \cup \mathbb{B}$ such that there exist bijections
	$h'':C \to (U \setminus h(W))$ and $g'':D \to (\mathbb{B}
	\setminus g(\mathbb{A}))$. We then set $g': = g \cup g''$ and $h':
	= h \cup h''$, and define $\mathcal{M}'$ on the basis of $W': = W
	\cup C$ and $\mathbb{A}' := \mathbb{A} \cup D$, setting, e.g.,
	$u\mathrel{\prec'}u' :\Leftrightarrow g'(u)\mathrel{\lhd}g'(u')$
	for arbitrary $u, u' \in W'$, and similarly for the world domains
	and predicate extensions.
\end{proof}
The notion of a complete $\mathsf{IL}$-type, though useful, turns
out to be somewhat too narrow for our purposes. We therefore adopt
the following classification of intuitionistic types:
\begin{definition}\label{D:types}
Let $(\mathcal{M}, w) \in Pmod_\Theta$, let $n < \omega$, let
$\bar{a}_n \in A^n_w$, let $\bar{c}_{n + 1}$, be a tuple of
pairwise distinct individual constants outside $\Theta$, let
$\Gamma, \Delta\subseteq IL_{\emptyset}(\Theta \cup
\{\bar{c}_n\})$ and $\Xi \subseteq IL_{\emptyset}(\Theta \cup
\{\bar{c}_{n+1}\})$. Then we say that:
\begin{itemize}
    \item $(\Gamma, \Delta)$ is a $\bar{c}_n$-successor $\mathsf{IL}$-type of $(\mathcal{M},
    w, \bar{a}_n)$ iff:
    $$
(\forall \Gamma' \Subset \Gamma)(\forall \Delta' \Subset
\Delta)(\exists v \succ
    w)((\Gamma',
    \Delta')\subseteq Tp_{\mathsf{IL}}(\mathcal{M}, v,
\bar{c}_n/\mathbb{H}_{wv}\langle\bar{a}_n\rangle)).
    $$
    \item $\Xi$ is an $\bar{c}_{n+1}$-existential $\mathsf{IL}$-type of $(\mathcal{M},
    w, \bar{a}_n)$ iff:
    $$
(\forall \Xi' \Subset \Xi)(\exists a_{n+1} \in
A_w)((\Xi',\emptyset)\subseteq Tp_{\mathsf{IL}}(\mathcal{M}, w,
\bar{c}_{n+1}/\bar{a}_{n+1})).
    $$
    \item $\Xi$ is a $\bar{c}_{n+1}$-universal $\mathsf{IL}$-type of $(\mathcal{M},
    w, \bar{a}_n)$ iff:
    $$
(\forall \Xi' \Subset \Xi)(\exists v \succ
    w)(\exists b \in
    A_v)((\emptyset,\Xi') \subseteq Tp_{\mathsf{IL}}(\mathcal{M}, v,
\bar{c}_{n+1}/\mathbb{H}_{wv}\langle\bar{a}_n\rangle^\frown b)).
    $$
\end{itemize}
\end{definition}
It is easy to see that, due to the reflexivity of the
intuitionistic accessibility relation and the fact that
$\mathbb{H}_{ww}$ is always an identical mapping, a complete
$\mathsf{IL}$-type of a given tuple $\bar{a}_{n} \in A_w^{n}$ is
a particular case of a successor $\mathsf{IL}$-type of
$(\mathcal{M}, w, \bar{a}_n)$; our definition, therefore, extends
the notion of a complete type. We will say that a set or a pair of
sets of sentences is an $\mathsf{IL}$-type of
$(\mathcal{M}, w, \bar{a}_n)$ iff it is either a successor, or
an existential, or a universal type of $(\mathcal{M}, w, \bar{a}_n)$ for some (equivalently, any) tuple of pairwise distinct fresh constants;
we will say that it is an $\mathsf{IL}$-type of $(\mathcal{M}, w)$
iff it is a type of $(\mathcal{M}, w, \bar{a}_n)$ for some $n <
\omega$ and $\bar{a}_n \in A^n_w$; finally, we will say that it is
a type of $\mathcal{M}$ iff there exists a $w \in W$ such that it
is a type of $(\mathcal{M}, w)$.

Another straightforward observation is that the intuitionistic
types, as given by Definition \ref{D:types}, are related to
certain classes of intuitionistic formulas. We state this
observation as a corollary (omitting the obvious proof):
\begin{corollary}\label{C:types-formulas}
Let $(\mathcal{M}, w) \in Pmod_\Theta$, let $n < \omega$, let
$\bar{a}_n \in A^n_w$, let $\bar{c}_{n + 1}$, be a tuple of
pairwise distinct individual constants outside $\Theta$, let
$\Gamma, \Delta\subseteq IL_{\emptyset}(\Theta \cup
\{\bar{c}_n\})$ and $\Xi \subseteq IL_{\emptyset}(\Theta \cup
\{\bar{c}_{n+1}\})$. Then all of the following statements hold:
\begin{enumerate}
    \item $(\Gamma, \Delta)$ is a $\bar{c}_n$-successor $\mathsf{IL}$-type of $(\mathcal{M},
    w, \bar{a}_n)$ iff:
$$
(\forall \Gamma' \Subset \Gamma)(\forall \Delta' \Subset
\Delta)(([\mathcal{M}, w], \bar{c}_n/\bar{a}_n), w
\not\models_{\mathsf{IL}} \bigwedge\Gamma'\to
    \bigvee\Delta')
    $$
    \item $\Xi$ is an $\bar{c}_{n+1}$-existential $\mathsf{IL}$-type of $(\mathcal{M},
    w, \bar{a}_n)$ iff:
        $$
(\forall \Xi' \Subset \Xi)(([\mathcal{M}, w], \bar{c}_n/\bar{a}_n),
w \models_{\mathsf{IL}} \exists c_{n+1}\bigwedge\Xi'),
    $$

    \item $\Xi$ is a  $\bar{c}_{n+1}$-universal $\mathsf{IL}$-type of $(\mathcal{M},
    w, \bar{a}_n)$ iff:
    \begin{align*}
(\forall \Xi' \Subset \Xi)(([\mathcal{M}, w], \bar{c}_n/\bar{a}_n),
w \not\models_{\mathsf{IL}} \forall c_{n+1}\bigvee\Xi'),
\end{align*}
\end{enumerate}
\end{corollary}

We now  need to extend the definition of type realization from
complete types to the more general notion of an $\mathsf{IL}$-type
given in Definition \ref{D:types}:
\begin{definition}\label{D:types-realization}
Let $(\mathcal{M}, w) \in Pmod_\Theta$, let $\mathcal{M}
\preccurlyeq_{\mathsf{IL}} \mathcal{N}$, let $n < \omega$, let
$\bar{a}_n \in A^n_w$, let $\bar{c}_{n + 1}$, be a tuple of
pairwise distinct individual constants outside $\Theta$, let
$\Gamma, \Delta\subseteq IL_{\emptyset}(\Theta \cup
\{\bar{c}_n\})$ and $\Xi \subseteq IL_{\emptyset}(\Theta \cup
\{\bar{c}_{n+1}\})$. Then we say that:
\begin{itemize}
    \item If $(\Gamma, \Delta)$ is a $\bar{c}_n$-successor $\mathsf{IL}$-type of $(\mathcal{M},
    w, \bar{a}_n)$, then $(\Gamma, \Delta)$ is realized in $\mathcal{N}$ iff there exists $v \in U$ such that
    $w\mathrel{\lhd}v$ and we have $(\Gamma,
    \Delta)\subseteq Tp_{\mathsf{IL}}(\mathcal{N}, v,
\bar{c}_n/\mathbb{G}_{wv}\langle\bar{a}_n\rangle)$.

    \item If $\Xi$ is an $\bar{c}_{n+1}$-existential $\mathsf{IL}$-type of $(\mathcal{M},
    w, \bar{a}_n)$, then $\Xi$ is realized in $\mathcal{N}$ iff there exists  $b \in B_{w}$ such that $(\Xi,\emptyset)\subseteq Tp_{\mathsf{IL}}(\mathcal{N}, w,
\bar{c}_{n+1}/\bar{a}_n^\frown b)$.

    \item If $\Xi$ is a $\bar{c}_{n+1}$-universal $\mathsf{IL}$-type of $(\mathcal{M},
    w, \bar{a}_n)$, then $\Xi$ is realized in $\mathcal{N}$ iff there exist  $v \in U$ and  $b \in B_v$ such that
    $w \lhd v$ and we have $$(\emptyset,\Xi) \subseteq Tp_{\mathsf{IL}}(\mathcal{N}, v,
\bar{c}_{n+1}/\mathbb{G}_{wv}\langle \bar{a}_n\rangle^\frown b).$$
\end{itemize}
\end{definition}
We observe that Definition \ref{D:submodel} allows us to equivalently replace every occurrence of $\mathbb{G}_{wv}$ in Definition \ref{D:types-realization} by an occurrence of $\mathbb{H}_{wv}$.

Every model is trivially an $\mathsf{IL}$-elementary submodel of
itself and trivially realizes all of its own complete types.
However, for an intuitionistic model to realize all of its types
in the sense of Definition \ref{D:types-realization} is a much
rarer property that merits a special name:
\begin{definition}\label{D:saturation}
Let $\mathcal{M}$ be a $\Theta$-model. We say that $\mathcal{M}$
is $\mathsf{IL}$-saturated iff it realizes all $\mathsf{IL}$-types
of $\mathcal{M}$.
\end{definition}

Another interesting situation occurs when a model happens to realize every type of its proper $\mathsf{IL}$-elementary submodel. We collect some properties of type realization in the following lemma:
\begin{lemma}\label{L:type-realization}
Let $\mathcal{M}'
\preccurlyeq_{\mathsf{IL}}\mathcal{M}
\preccurlyeq_{\mathsf{IL}} \mathcal{N}$ be an $\mathsf{IL}$-elementary chain of $\Theta$-models, let $\mathcal{N}'$ be a  $\Theta$-model such that for some functions $g$ and $h$ we have $(g,h):\mathcal{N}\cong \mathcal{N}'$. Finally, let $\Sigma \subseteq \Theta$. Then the following statements hold:
\begin{enumerate}
	\item Every $\mathsf{IL}$-type of $\mathcal{M}'$, that is realized in $\mathcal{M}$, is also realized in $\mathcal{N}$.
	
	\item If $\mathcal{N}$ realizes every type of $\mathcal{M}$, then $\mathcal{N}\upharpoonright\Sigma$ realizes every type of $\mathcal{M}\upharpoonright\Sigma$.
	
	\item For every $n < \omega$, every $v \in U$, and every $\bar{b}_n \in B^n_v$, $(\Gamma, \Delta)$ (resp. $\Xi$) is an $\mathsf{IL}$-type of $(\mathcal{N}, v, \bar{b}_n)$ iff it is an $\mathsf{IL}$-type of $(\mathcal{N}', h(v), g\langle\bar{b}_n\rangle)$.
	
	\item $\mathcal{N}$ realizes every type of $\mathcal{M}$ iff $\mathcal{N}'$ realizes every type of $(g,h)(\mathcal{M})$.    
\end{enumerate} 	
\end{lemma}
\begin{proof}[Proof (a sketch)]
	(Part 1). Assume that $(\Gamma, \Delta)$ (resp. $\Xi$) is an $\mathsf{IL}$-type of $\mathcal{M}'$ that is realized in $\mathcal{M}$. This means that for an appropriate $n < \omega$, tuple of constants $\bar{c}_n$, $w \in W$, and $\bar{a}_n \in A^n_w$ 
	 we have that $(\Gamma, \Delta)$ (resp. $(\Xi, \emptyset)$, $(\emptyset, \Xi)$) is a subset of  $Tp_{\mathsf{IL}}(\mathcal{M}, w,
	\bar{c}_{n}/\bar{a}_{n})$. But we have $Tp_{\mathsf{IL}}(\mathcal{M}, w,
	\bar{c}_{n}/\bar{a}_{n}) = Tp_{\mathsf{IL}}(\mathcal{N}, w,
	\bar{c}_{n}/\bar{a}_{n})$, therefore our type is also realized in $\mathcal{N}$.
	
	(Part 2). Assume that $(\Gamma, \Delta)$ (resp. $\Xi$) is an $\mathsf{IL}$-type of $\mathcal{M}\upharpoonright\Sigma$. Then it follows from the respective part of Corollary \ref{C:types-formulas}, that $(\Gamma, \Delta)$ (resp. $\Xi$) is also an $\mathsf{IL}$-type of $\mathcal{M}$. But then, $(\Gamma, \Delta)$ (resp. $\Xi$) must be realized in $\mathcal{N}$, which means, by definition, that for an appropriate $n < \omega$, tuple of constants $\bar{c}_n$, $v \in U$, and $\bar{b}_n \in B^n_v$, we have that $(\Gamma, \Delta)$ (resp. $(\Xi, \emptyset)$, $(\emptyset, \Xi)$) is a subset of  $Tp_{\mathsf{IL}}(\mathcal{N}, v,
	\bar{c}_{n}/\bar{b}_{n})$. But this means, further, that we also have:
	\begin{align*}
			(\Gamma, \Delta) &= (\Gamma\cap IL_\emptyset(\Sigma\cup \{\bar{c}_{n}\}), \Delta\cap IL_\emptyset(\Sigma\cup \{\bar{c}_{n}\})) = (\Gamma, \Delta)\cap IL_\emptyset(\Sigma\cup \{\bar{c}_{n}\})\\ 
			&\subseteq Tp_{\mathsf{IL}}(\mathcal{N}, v,
		\bar{c}_{n}/\bar{b}_{n})\cap IL_\emptyset(\Sigma\cup \{\bar{c}_{n}\}) = Tp_{\mathsf{IL}}(\mathcal{N}\upharpoonright\Sigma, v,
		\bar{c}_{n}/\bar{b}_{n}),
	\end{align*}
where the last equality holds by Lemma \ref{L:el-submodels}.2. This means that $(\Gamma, \Delta)$ is also realized in $\mathcal{N}\upharpoonright\Sigma$. We argue similarly for  $(\Xi, \emptyset)$ and $(\emptyset, \Xi)$, and for the converse.

(Part 3). By Corollary \ref{C:types-formulas} and Lemma \ref{L:embedding}.2.

(Part 4). Assume that, for some  $n < \omega$, $w \in W$,
$\bar{a}_n \in A^n_w$, and a tuple of fresh constants $\bar{c}_{n}$,  $(\Gamma, \Delta)$ is a $\bar{c}_n$-successor $\mathsf{IL}$-type of $(\mathcal{M}, w, \bar{a}_n)$. Since $(\Gamma, \Delta)$ is realized in $\mathcal{N}$, let $v \in U$ be such that we have both
$w \mathrel{\lhd} v$ and $(\Gamma,
\Delta)\subseteq Tp_{\mathsf{IL}}(\mathcal{N}, v,
\bar{c}_n/\mathbb{G}_{wv}\langle\bar{a}_n\rangle)$. Note that we have, of course, $(g,h):\mathcal{M}\cong (g,h)(\mathcal{M})$, therefore, by Part 3, $(\Gamma, \Delta)$ is a $\bar{c}_n$-successor $\mathsf{IL}$-type of $((g,h)(\mathcal{M}), h(w), g\langle\bar{a}_n\rangle)$. Moreover, by Lemma \ref{L:embedding}.2, we have: 
$$
(\Gamma,
\Delta)\subseteq Tp_{\mathsf{IL}}(\mathcal{N}, v,
\bar{c}_n/\mathbb{G}_{wv}\langle\bar{a}_n\rangle) = Tp_{\mathsf{IL}}(\mathcal{N}', h(v),
\bar{c}_n/g\langle\mathbb{G}_{wv}\langle\bar{a}_n\rangle\rangle),
$$	
and thus $(\Gamma,
\Delta)$ is also realized in $\mathcal{N}'$. We argue similarly for the cases of existential and universal type, and for the converse.   
\end{proof}

\subsection{Asimulations}

We devote this subsection to a treatment of world-object
asimulations introduced in \cite{o1} and adapted to the particular details of the version of
intuitionistic first-order logic introduced in the previous
sections:
\begin{definition}\label{D:asimulation}
{\em Let $(\mathcal{M}_1, w_1)$, $(\mathcal{M}_2, w_2)$ be pointed
$\Theta$-models and let, for $n \geq 0$, $\bar{a}_n \in
(A_1)^n_{w_1}$ and $\bar{b}_n \in (A_2)^n_{w_2}$. Moreover, fix a
set $\{c_i\mid i < \omega\}$ of pairwise distinct individual
constants outside $\Theta$. A binary relation $A$ is called an
\emph{$\mathsf{IL}$-asimulation from $(\mathcal{M}_1,w_1,
\bar{a}_n)$ to $(\mathcal{M}_2,w_2, \bar{b}_n)$} iff for any $i,j$
such that $\{ i,j \} = \{ 1, 2 \}$, any $w \in W_i$, $v,t \in
W_j$, any $l \geq 0$, $\bar{\alpha}_l \in
(A_i)^l_{w}$ and $\bar{\beta}_l \in (A_j)^l_{v}$ 
and any atomic $\varphi \in
IL_{\emptyset}(\Theta \cup\{\bar{c}_l\})$ the following conditions
hold:
\begin{align}
&A \subseteq \bigcup_{l \geq 0}(((W_1 \times \mathbb{A}_1^l)
\times (W_2 \times \mathbb{A}_2^l)) \cup ((W_2 \times
\mathbb{A}_2^l) \times (W_1 \times
\mathbb{A}_1^l)))\label{E:c22}\tag{\text{type}}\\
&(w_1; \bar{a}_n)\mathrel{A}(w_2;\bar{b}_n)\label{E:c11}\tag{\text{elem}}\\
&((w;\bar{\alpha}_l)A(v;\bar{\beta}_l) \& ([\mathcal{M}_i,w],
\bar{c}_l/\bar{\alpha}_l), w\models_{\mathsf{IL}} \varphi)
\Rightarrow ([\mathcal{M}_j,v], \bar{c}_l/\bar{\beta}_l),
v\models_{\mathsf{IL}} \varphi\label{E:c33}\tag{\text{atom}}\\
&((w;\bar{\alpha}_l)A(v;\bar{\beta}_l) \& v\mathrel{\prec_j}t)
\Rightarrow\notag\\
&\qquad\Rightarrow  \exists u \in W_i(w\mathrel{\prec_i}u \,\&\,
(u;(\mathbb{H}_i)_{wu}\langle\bar{\alpha}_l\rangle)A(t;(\mathbb{H}_j)_{vt}\langle\bar{\beta}_l\rangle)
\,\&\,
(t;(\mathbb{H}_j)_{vt}\langle\bar{\beta}_l\rangle)A(u;(\mathbb{H}_i)_{wu}\langle\bar{\alpha}_l\rangle)))\label{E:c44}\tag{\text{s-back}}\\
&((w;\bar{\alpha}_l)A(v;\bar{\beta}_l) \& \alpha' \in (A_i)_w)
\Rightarrow
(\exists \beta'\in(A_j)_v)((w;(\bar{\alpha}_l)^\frown \alpha')A(v;(\bar{\beta}_l)^\frown \beta'));\label{E:c55}\tag{\text{obj-forth}}\\
&((w;\bar{\alpha}_l)A(v;\bar{\beta}_l) \&
v\mathrel{\prec_j}t\wedge \beta' \in (A_j)_t))
\Rightarrow\notag\\
&\qquad\Rightarrow (\exists u,\alpha')(w\mathrel{\prec_i}u \,\&\,
\alpha' \in (A_i)_u \&
(u;(\mathbb{H}_i)_{wu}\langle\bar{\alpha}_l\rangle^\frown
\alpha')A(t;(\mathbb{H}_j)_{vt}\langle\bar{\beta}_l\rangle^\frown
\beta') ).\label{E:c66}\tag{\text{obj-back}}
\end{align}
}
\end{definition}
It is easy to see  that the above definition is correct in that
the fact that $A$ is an asimulation does not depend on the choice
of  $\{c_i\mid i < \omega\}$.

Our first series of statements about asimulations is that they
scale down neatly w.r.t. subsequences of objects, generated
submodels, and constant extensions. More precisely:
\begin{lemma}\label{L:asimulations-generated}
Let $(\mathcal{M}_1, w_1)$, $(\mathcal{M}_2, w_2)$ be pointed
$\Theta$-models and let, for $n \geq 0$, $\bar{a}_n \in
(A_1)^n_{w_1}$ and $\bar{b}_n \in (A_2)^n_{w_2}$. Moreover, fix a
tuple $\bar{c}_n$ of pairwise distinct individual constants
outside $\Theta$ and let a binary relation $A$ be an
$\mathsf{IL}$-asimulation from $(\mathcal{M}_1,w_1, \bar{a}_n)$ to
$(\mathcal{M}_2,w_2, \bar{b}_n)$. Then the following statements
hold:
\begin{enumerate}
    \item The relation
    \begin{align*}
    A\downarrow = \{ ((w;\bar{\alpha}_l),(v;\bar{\beta}_l))\mid (\exists m& \geq l)(\exists\bar{\gamma}_m)(\exists\bar{\delta}_m)((w;\bar{\gamma}_m)A(v;\bar{\delta}_m)\\
    &\&\,(\exists i_1,\ldots i_l \leq m)(\bigwedge^l_{j = 1}(\alpha_j = \gamma_{i_j}\,\&\,\beta_j = \delta_{i_j})))
    \}
    \end{align*} 
    is an $\mathsf{IL}$-asimulation from $(\mathcal{M}_1,w_1, \bar{a}_i)$ to
$(\mathcal{M}_2,w_2, \bar{b}_i)$ for every $i \leq n$.
    \item $A_{w_1,w_2} = \{ ((w;\bar{\alpha}_l),(v;\bar{\beta}_l))\mid
    (w;\bar{\alpha}_l)A(v;\bar{\beta}_l)\,\&\,(\exists i,
    j)(\{i,j\}= \{1,2\}\,\&\,w \in [W_i, w_i]\,\&\,v\in [W_j,
    w_j])\}$ is an $\mathsf{IL}$-asimulation from $(([\mathcal{M}_1,w_1],\bar{c}_n/\bar{a}_n), w_1, \bar{a}_n)$ to
$(([\mathcal{M}_2,w_2],\bar{c}_n/\bar{b}_n), w_2, \bar{b}_n)$; in
particular, for $n = 0$, we get that $A_{w_1,w_2}$ is an
$\mathsf{IL}$-asimulation from $([\mathcal{M}_1,w_1], w_1,
\bar{a}_n)$ to $([\mathcal{M}_2,w_2], w_2, \bar{b}_n)$. 
 \item As a
consequence, $(A_{w_1,w_2})\downarrow$ is an
$\mathsf{IL}$-asimulation from
$(([\mathcal{M}_1,w_1],\bar{c}_n/\bar{a}_n), w_1)$ to
$(([\mathcal{M}_2,w_2],\bar{c}_n/\bar{b}_n), w_2)$.
\end{enumerate}
\end{lemma}

Intuitionistic first-order formulas are known to be preserved
under $\mathsf{IL}$-asimulations \cite[Corollary 3.3]{o1}. More precisely, if
$(\mathcal{M}_1, w_1)$, and $(\mathcal{M}_2, w_2)$ are pointed
$\Theta$-models, $\bar{a}_n \in (A_1)^n_{w_1}$, $\bar{b}_n \in
(A_2)^n_{w_2}$, and $A$ is an $\mathsf{IL}$-asimulation from
$(\mathcal{M}_1,w_1, \bar{a}_n)$ to $(\mathcal{M}_2,w_2,
\bar{b}_n)$, then, for every tuple $\bar{c}_n$ of pairwise
distinct individual constants outside $\Theta$, we have
$Tp^+_{\mathsf{IL}}(\mathcal{M}_1,w_1, \bar{c}_n/\bar{a}_n)
\subseteq Tp^+_{\mathsf{IL}}(\mathcal{M}_2,w_2,
\bar{c}_n/\bar{b}_n)$. Moreover, preservation under
$\mathsf{IL}$-asimulations is known to semantically characterize
intuitionistic first-order logic as a fragment of classical
first-order logic, see \cite[Theorem 3.8]{o1} for the proof.

Next we consider asimulations between $\mathsf{IL}$-saturated
models. The following lemma states that such asimulations can be
defined in an easy and natural way:
\begin{lemma}\label{L:asimulations}
Let $(\mathcal{M}_1, w_1)$, $(\mathcal{M}_2, w_2)$ be pointed
$\Theta$-models, let $\bar{a}_n \in (A_1)^n_{w_1}$, and let
$\bar{b}_n \in (A_2)^n_{w_2}$, and let $\bar{c}_{n}$ be a tuple of
pairwise distinct individual constants outside $\Theta$. If we
have

\[ Tp^+_{\mathsf{IL}}(\mathcal{M}_1, w_1,
\bar{c}_{n}/\bar{a}_n)\subseteq Tp^+_{\mathsf{IL}}(\mathcal{M}_2,
w_2, \bar{c}_{n}/\bar{b}_n)
\]
for every tuple $\bar{c}_{n}$ of pairwise distinct individual
constants outside $\Theta$, and both $\mathcal{M}_1$ and
$\mathcal{M}_2$ are $\mathsf{IL}$-saturated, then the relation $A$
such that for all $\{ i,j \} = \{ 1,2 \}$, all $u \in W_i$, $s \in
W_j$, $\bar{\alpha}_k \in (A_i)^k_{u}$, and all $\bar{\beta}_k \in
(A_j)^k_{s}$ we have
$$
(u; \bar{\alpha}_k) \mathrel{A}(s; \bar{\beta}_k) \Leftrightarrow
Tp^+_{\mathsf{IL}}(\mathcal{M}_i, u,
\bar{c}_{k}/\bar{\alpha}_k)\subseteq
Tp^+_{\mathsf{IL}}(\mathcal{M}_j, s, \bar{c}_{k}/\bar{\beta}_k)
$$
for any given set $\{c_i\mid i < \omega\}$ of pairwise distinct
individual constants outside $\Theta$ is an
$\mathsf{IL}$-asimulation from $(\mathcal{M}_1, w_1, \bar{a}_n)$
to $(\mathcal{M}_2, w_2, \bar{b}_n)$.
\end{lemma}
\begin{proof}
The relation $A$, as defined in the lemma, obviously satisfies
conditions \eqref{E:c22}, \eqref{E:c11}, and \eqref{E:c33} given
in Definition \ref{D:asimulation}. We check the remaining
conditions.

\eqref{E:c44}. Assume that $(u; \bar{\alpha}_k) \mathrel{A}(s;
\bar{\beta}_k)$, so that we have
$$Tp^+_{\mathsf{IL}}(\mathcal{M}_i, u,
\bar{c}_{k}/\bar{\alpha}_k)\subseteq
Tp^+_{\mathsf{IL}}(\mathcal{M}_j, s, \bar{c}_{k}/\bar{\beta}_k),$$
and that for some $t \in W_j$ we have $s\mathrel{\prec_j}t$.
Consider an arbitrary pair $$(\Gamma, \Delta) \Subset
Tp_{\mathsf{IL}}(\mathcal{M}_j, t,
\bar{c}_{k}/(\mathbb{H}_j)_{st}\langle\bar{\beta}_k\rangle).$$ We
reason as follows:
\begin{align*}
    &([\mathcal{M}_j, t],
\bar{c}_{k}/(\mathbb{H}_j)_{st}\langle\bar{\beta}_k\rangle), t
\models_{\mathsf{IL}} (\Gamma,\Delta) &&\text{(by the choice of
$\Gamma, \Delta$)}\\
&[([\mathcal{M}_j, s], \bar{c}_{k}/\bar{\beta}_k),t], t
\models_{\mathsf{IL}} (\Gamma,\Delta) &&\text{(by Lemma \ref{L:cutoff-constants}.2)}\\
&([\mathcal{M}_j, s], \bar{c}_{k}/\bar{\beta}_k), t
\models_{\mathsf{IL}} (\Gamma,\Delta) &&\text{(by Lemma \ref{L:satisfaction}.2)}\\
&([\mathcal{M}_j, s], \bar{c}_{k}/\bar{\beta}_k), s
\not\models_{\mathsf{IL}} (\bigwedge\Gamma \to \bigvee\Delta)&&\text{(by def. of $\models_{\mathsf{IL}} $)}\\
&([\mathcal{M}_i, u], \bar{c}_{k}/\bar{\alpha}_k), u
\not\models_{\mathsf{IL}} (\bigwedge\Gamma \to
\bigvee\Delta)&&\text{(by $(u; \bar{\alpha}_k) \mathrel{A}(s;
\bar{\beta}_k)$)}
\end{align*}
From the last equation, by Corollary \ref{C:types-formulas}.1, we
infer that $Tp_{\mathsf{IL}}(\mathcal{M}_j, t,
\bar{c}_{k}/(\mathbb{H}_j)_{st}\langle\bar{\beta}_k\rangle)$ is a
$\bar{c}_{k}$-successor type of $(\mathcal{M}_i, u,
\bar{\alpha}_k)$. By the $\mathsf{IL}$-saturation of both
$\mathcal{M}_1$ and $\mathcal{M}_2$, it follows that for some $w
\in W_i$ such that $u\mathrel{\prec_i}w$ we have
$$Tp_{\mathsf{IL}}(\mathcal{M}_j, t,
\bar{c}_{k}/(\mathbb{H}_j)_{st}\langle\bar{\beta}_k\rangle)
\subseteq Tp_{\mathsf{IL}}(\mathcal{M}_i, w,
\bar{c}_{k}/(\mathbb{H}_i)_{uw}\langle\bar{\alpha}_k\rangle).$$
Since both $Tp_{\mathsf{IL}}(\mathcal{M}_j, t,
\bar{c}_{k}/(\mathbb{H}_j)_{st}\langle\bar{\beta}_k\rangle)$ and
$Tp_{\mathsf{IL}}(\mathcal{M}_i, w,
\bar{c}_{k}/(\mathbb{H}_i)_{uw}\langle\bar{\alpha}_k\rangle)$ are,
in fact, partitions of $IL_{\emptyset}(\Theta \cup\{\bar{c}_k\})$,
it follows that $$Tp_{\mathsf{IL}}(\mathcal{M}_j, t,
\bar{c}_{k}/(\mathbb{H}_j)_{st}\langle\bar{\beta}_k\rangle) =
Tp_{\mathsf{IL}}(\mathcal{M}_i, w,
\bar{c}_{k}/(\mathbb{H}_i)_{uw}\langle\bar{\alpha}_k\rangle),$$ and
thus also $Tp^+_{\mathsf{IL}}(\mathcal{M}_j, t,
\bar{c}_{k}/(\mathbb{H}_j)_{st}\langle\bar{\beta}_k\rangle) =
Tp^+_{\mathsf{IL}}(\mathcal{M}_i, w,
\bar{c}_{k}/(\mathbb{H}_i)_{uw}\langle\bar{\alpha}_k\rangle)$, or,
in other words, both $(w;
(\mathbb{H}_i)_{uw}\langle\bar{\alpha}_k\rangle) \mathrel{A}(t;
(\mathbb{H}_j)_{st}\langle\bar{\beta}_k\rangle)$ and $(t;
(\mathbb{H}_j)_{st}\langle\bar{\beta}_k\rangle) \mathrel{A}(w;
(\mathbb{H}_i)_{uw}\langle\bar{\alpha}_k\rangle)$.

\eqref{E:c55}. Assume that $(u; \bar{\alpha}_k) \mathrel{A}(s;
\bar{\beta}_k)$, so that we have
$$Tp^+_{\mathsf{IL}}(\mathcal{M}_i, u,
\bar{c}_{k}/\bar{\alpha}_k)\subseteq
Tp^+_{\mathsf{IL}}(\mathcal{M}_j, s, \bar{c}_{k}/\bar{\beta}_k),$$
and assume that $\alpha \in (A_i)_u$. Consider an arbitrary $\Xi
\Subset Tp^+_{\mathsf{IL}}(\mathcal{M}_i, u,
\bar{c}_{k+1}/(\bar{\alpha}_k)^\frown\alpha)$. We reason as follows:
\begin{align*}
    &([\mathcal{M}_i, u],
    \bar{c}_{k+1}/(\bar{\alpha}_k)^\frown\alpha),
    u\models_{\mathsf{IL}}\bigwedge\Xi &&\text{(by the choice of $\Xi$)}\\
    &([\mathcal{M}_i, u],
    \bar{c}_{k}/\bar{\alpha}_k),
    u\models_{\mathsf{IL}}\exists c_{k+1}\bigwedge\Xi &&\text{(by Corollary
    \ref{C:satisfaction-cutoff-constants})}\\
     &([\mathcal{M}_j, s],
    \bar{c}_{k}/\bar{\beta}_k),
    u\models_{\mathsf{IL}}\exists c_{k+1}\bigwedge\Xi &&\text{(by $(u; \bar{\alpha}_k) \mathrel{A}(s;
\bar{\beta}_k)$)}
\end{align*}
But then, by Corollary \ref{C:types-formulas}.2, we infer that
$Tp^+_{\mathsf{IL}}(\mathcal{M}_i, u,
\bar{c}_{k+1}/\bar{\alpha}_k^\frown\alpha)$ is a
$\bar{c}_{k+1}$-existential type of $(\mathcal{M}_j, s,
\bar{\beta}_k)$. By the $\mathsf{IL}$-saturation of both
$\mathcal{M}_1$ and $\mathcal{M}_2$, it follows that for some
$\beta \in (A_j)_s$ we have $Tp^+_{\mathsf{IL}}(\mathcal{M}_i, u,
\bar{c}_{k+1}/(\bar{\alpha}_k)^\frown\alpha) \subseteq
Tp^+_{\mathsf{IL}}(\mathcal{M}_j, s,
\bar{c}_{k+1}/(\bar{\beta}_k)^\frown\beta)$, or, in other words,
$(u; (\bar{\alpha}_k)^\frown\alpha) \mathrel{A}(s;
(\bar{\beta}_k)^\frown\beta)$.

\eqref{E:c66}. Assume that $(u; \bar{\alpha}_k) \mathrel{A}(s;
\bar{\beta}_k)$,
so that we have $$Tp^+_{\mathsf{IL}}(\mathcal{M}_i, u,
\bar{c}_{k}/\bar{\alpha}_k)\subseteq
Tp^+_{\mathsf{IL}}(\mathcal{M}_j, s, \bar{c}_{k}/\bar{\beta}_k),$$
and assume that for some $t \in W_j$ such that
$s\mathrel{\prec_j}t$ we have $\beta \in (A_j)_t$. Consider an
arbitrary $\Xi \Subset Tp^-_{\mathsf{IL}}(\mathcal{M}_j, t,
\bar{c}_{k+1}/(\mathbb{H}_j)_{st}\langle\bar{\beta}_k\rangle^\frown\beta)$.
We reason as follows:
\begin{align*}
    &([\mathcal{M}_j, t],
\bar{c}_{k+1}/(\mathbb{H}_j)_{st}\langle\bar{\beta}_k\rangle^\frown\beta),
t \not\models_{\mathsf{IL}} \bigvee\Xi &&\text{(by the choice of
$\Xi$)}\\
&([\mathcal{M}_j, s], \bar{c}_{k}/\bar{\beta}_k), s
\not\models_{\mathsf{IL}} \forall c_{k +1}\bigvee\Xi&&\text{(by
Corollary
    \ref{C:satisfaction-cutoff-constants})}\\
&([\mathcal{M}_i, u], \bar{c}_{k}/\bar{\alpha}_k), u
\not\models_{\mathsf{IL}} \forall c_{k +1}\bigvee\Xi&&\text{(by
$(u; \bar{\alpha}_k) \mathrel{A}(s; \bar{\beta}_k)$)}
\end{align*}
But then, by Corollary \ref{C:types-formulas}.3, we infer that
$Tp^-_{\mathsf{IL}}(\mathcal{M}_j, t,
\bar{c}_{k+1}/(\mathbb{H}_j)_{st}\langle\bar{\beta}_k\rangle^\frown\beta)$
is a $\bar{c}_{k+1}$-universal type of $(\mathcal{M}_i, u,
\bar{\alpha}_k)$. By the $\mathsf{IL}$-saturation of both
$\mathcal{M}_1$ and $\mathcal{M}_2$, it follows that for some $w
\in W_i$ such that $u\mathrel{\prec_i}w$ and for some $\alpha \in
(A_i)_w$ we have 
$$
Tp^-_{\mathsf{IL}}(\mathcal{M}_j, t,
\bar{c}_{k+1}/(\mathbb{H}_j)_{st}\langle\bar{\beta}_k\rangle^\frown\beta)
\subseteq Tp^-_{\mathsf{IL}}(\mathcal{M}_i, w,
\bar{c}_{k+1}/(\mathbb{H}_i)_{uw}\langle\bar{\alpha}_k\rangle^\frown\alpha),
$$
whence clearly $Tp^+_{\mathsf{IL}}(\mathcal{M}_i, w,
\bar{c}_{k+1}/(\mathbb{H}_i)_{uw}\langle\bar{\alpha}_k\rangle^\frown\alpha)
\subseteq Tp^+_{\mathsf{IL}}(\mathcal{M}_j, t,
\bar{c}_{k+1}/(\mathbb{H}_j)_{st}\langle\bar{\beta}_k\rangle^\frown\beta)$
or, in other words,
$(w;(\mathbb{H}_i)_{uw}\langle\bar{\alpha}_k\rangle^\frown\alpha)
\mathrel{A}(t;
(\mathbb{H}_j)_{st}\langle\bar{\beta}_k\rangle^\frown\beta)$.
\end{proof}

\subsection{Intuitionistic unravellings}\label{subS:Unravel}

Unravelling Kripke models is a well-known idea from classical
modal logic which, for any given $(\mathcal{M}, w)$, allows to
find an `equivalent' rooted $\Theta$-model. Here we adapt it to
the intuitionistic setting. For a given intuitionistic pointed
Kripke $\Theta$-model $(\mathcal{M}, w)$, the model
$\mathcal{M}^{un(w)}= \langle W^{un(w)}, \prec^{un(w)},
\mathfrak{A}^{un(w)}, \mathbb{H}^{un(w)}\rangle$, called the
intuitionistic unravelling of $\mathcal{M}$ around $w$ is defined
as follows.
\begin{itemize}
\item $W^{un(w)} = \{ \bar{u}_n \in W^n\mid u_1 = w,\,(\forall i <
n)(u_i\mathrel{\prec}u_{i + 1}) \}$;

\item $\prec^{un(w)}$ is the reflexive and transitive closure of
the following relation:
$$\{ (s,t) \in W^{un(w)} \times W^{un(w)} \mid
(\exists u \in W)(t = s^\frown u) \};
$$
\item For all $\bar{u}_n \in W^{un(w)}$, $A^{un(w)}_{\bar{u}_n} =
\{ (a;\bar{u}_n) \mid a \in A_{u_n} \}$;

\item For every $\bar{u}_n \in W^{un(w)}$, every $\bar{a}_m \in
A^m_{u_n}$, every $P \in \Theta_m$, and every $c \in Const_\Theta$:
$$
I^{un(w)}_{\bar{u}_n}(P)((a_1;\bar{u}_n),\ldots, (a_m;\bar{u}_n))
\Leftrightarrow I_{u_n}(P)(\bar{a}_m),
$$
and
$$
I^{un(w)}_{\bar{u}_n}(c) = (I_{u_n}(c);\bar{u}_n).
$$
\item For all $1 < k < n < \omega$, and all $\bar{u}_k, \bar{u}_n
\in W^{un(w)}$, and every $a \in  A_{u_k}$, we set
$\mathbb{H}^{un(w)}_{\bar{u}_k\bar{u}_n}(a; \bar{u}_k) :=
(\mathbb{H}_{u_ku_n}(a), \bar{u}_n)$.
\end{itemize}
In what follows, we will also write $\bar{a}_m\odot\bar{u}_n$
instead of $((a_1;\bar{u}_n),\ldots, (a_m;\bar{u}_n))$; using this
notation, the final two items in the above definition of
unravelling can be rephrased as follows:
$$
I^{un(w)}_{\bar{u}_n}(P)(\bar{a}_m\odot\bar{u}_n) \Leftrightarrow
I_{u_n}(P)(\bar{a}_m),
$$
and
$$
\mathbb{H}^{un(w)}_{\bar{u}_k\bar{u}_n}\langle\bar{a}_m\odot\bar{u}_k\rangle
= (\mathbb{H}_{u_ku_n}\langle\bar{a}_m\rangle\odot\bar{u}_n).
$$

The following lemma sums up the basic facts about intuitionistic
unravellings:
\begin{lemma}\label{unravellinglemma}
Let $(\mathcal{M}, w)$ be a pointed $\Theta$-model. Then:
\begin{enumerate}
\item $\mathcal{M}^{un(w)}$ is a $\Theta$-model;

\item Define
$$
B := \{\langle (w_n; \bar{a}_m), (\bar{w}_n;
\bar{a}_m\odot\bar{w}_n)\rangle, \langle (\bar{w}_n;
\bar{a}_m\odot\bar{w}_n), (w_n; \bar{a}_m)\rangle \mid \bar{w}_n
\in W^{un(w)},\,m\geq 0,\,\bar{a}_m \in A^m_{w_n}\}.
$$
Then $B$ is an $\mathsf{IL}$-asimulation both from $(\mathcal{M},
w_n, \bar{a}_m)$ to $(\mathcal{M}^{un(w)}, \bar{w}_n,
\bar{a}_m\odot\bar{w}_n)$ and from $(\mathcal{M}^{un(w)},
\bar{w}_n, \bar{a}_m\odot\bar{w}_n)$ to $(\mathcal{M}, w_n,
\bar{a}_m)$ for every $w_n \in W$ such that $w_n \succ w$ and
every tuple $\bar{a}_m \in A^m_{\bar{w}_n}$; 

\item $Tp_{\mathsf{IL}}(\mathcal{M}, v_k, \bar{a}_m) =
Tp_{\mathsf{IL}}(\mathcal{M}^{un(w)}, \bar{v}_k, \bar{a}_m\odot\bar{v}_k)$
for any $\bar{v}_k \in W^{un(w)}$, any $m\geq 0$, and any
$\bar{a}_m \in A^m_{v_k}$;
\item $Tp_{\mathsf{IL}}(\mathcal{M}^{un(w)}, \bar{u}_n,
\bar{a}_m\odot\bar{u}_n) = Tp_{\mathsf{IL}}(\mathcal{M}^{un(w)},
\bar{v}_k, \bar{a}_m\odot\bar{v}_k)$ for any $\bar{u}_n, \bar{v}_k
\in W^{un(w)}$ whenever $u_n = v_k$, $m\geq 0$, and $\bar{a}_m \in
A^m_{v_k}$.
\end{enumerate}
\end{lemma}
\begin{proof} (Part 1) By definition, $\prec^{un(w)}$ is reflexive and transitive. Moreover,
we note that for arbitrary $\bar{w}_k, \bar{v}_n \in W^{un(w)}$ we
have
$$
\bar{w}_k\mathrel{\prec^{un(w)}}\bar{v}_n \Leftrightarrow k \leq n
\wedge \bar{w}_k = \bar{v}_k.
$$
Therefore, to show antisymmetry, assume that for a given
$\bar{w}_k, \bar{v}_n \in W^{un(w)}$ we have

$$
\bar{w}_k\mathrel{\prec^{un(w)}}\bar{v}_n \wedge
\bar{v}_n\mathrel{\prec^{un(w)}}\bar{w}_k.
$$
By the above biconditional it immediately follows that $k \leq n
\wedge n \leq k$ so that we have $k = n$ and, further, that
$\bar{w}_k = \bar{v}_k$. Therefore, we get $\bar{w}_k = \bar{v}_n$
and thus $\prec^{un(w)}$ is shown to be antisymmetric.

Next, it is obvious that $A^{un(w)}_{\bar{u}_n} \cap
A^{un(w)}_{\bar{v}_k} = \emptyset$ for all $\bar{v}_k, \bar{u}_n
\in W^{un(w)}$ such that $\bar{v}_k \neq \bar{u}_n$. It remains to
show that $\mathbb{H}^{un(w)}$ is defined correctly.

If $\bar{u}_n, \bar{u}_k \in W^{un(w)}$ are such that
$\bar{u}_k\mathrel{\prec^{un(w)}}\bar{u}_n$, then we must have $k
\leq n$ and $u_k \prec u_n$ and so $\mathbb{H}_{u_ku_n}$ must be
defined. Now, if $\bar{a}_m \in A^m_{u_k}$, and $P \in \Theta_m$,
then we have:
\begin{align*}
I^{un(w)}_{\bar{u}_k}(P)(\bar{a}_m\odot\bar{u}_k) &\Leftrightarrow
I_{u_k}(P)(\bar{a}_m)\\
 &\Rightarrow
I_{u_n}(P)(\mathbb{H}_{u_ku_n}\langle\bar{a}_m\rangle)\\
 &\Leftrightarrow
I^{un(w)}_{\bar{u}_n}(P)(\mathbb{H}_{u_ku_n}\langle\bar{a}_m\rangle\odot\bar{u}_n)\\
&\Leftrightarrow
I^{un(w)}_{\bar{u}_n}(P)(\mathbb{H}^{un(w)}_{\bar{u}_k\bar{u}_n}\langle\bar{a}_m\odot\bar{u}_k\rangle).
\end{align*}
Similarly, if $c \in Const_\Theta$, we have that:
\begin{align*}
	\mathbb{H}^{un(w)}_{\bar{u}_k\bar{u}_n}(I^{un(w)}_{\bar{u}_k}(c)) &= \mathbb{H}^{un(w)}_{\bar{u}_k\bar{u}_n}((I_{u_k}(c); \bar{u}_k))\\ 
	&= (\mathbb{H}_{u_ku_n}(I_{u_k}(c)); \bar{u}_n)\\
	&= (I_{u_n}(c); \bar{u}_n) = I^{un(w)}_{\bar{u}_n}(c).
\end{align*}

Moreover, let  $\bar{u}_n, \bar{u}_k$, and $\bar{a}_m$ be as
above, let $l \geq n$, and let $\bar{u}_l \in W^{un(w)}$ be such
that $\bar{u}_n\mathrel{\prec^{un(w)}}\bar{u}_l$. Then we have,
first:
$$
\mathbb{H}^{un(w)}_{\bar{u}_k\bar{u}_k}((a_1; \bar{u}_k)) =
(\mathbb{H}_{u_ku_k}(a_1); \bar{u}_k) = (a_1; \bar{u}_k),
$$
and, second:
\begin{align*}
\mathbb{H}^{un(w)}_{\bar{u}_k\bar{u}_l}((a_1; \bar{u}_k)) &=
(\mathbb{H}_{u_ku_l}(a_1), \bar{u}_l)\\
&= (\mathbb{H}_{u_nu_l}(\mathbb{H}_{u_ku_n}(a_1)); \bar{u}_l)\\
&=
\mathbb{H}^{un(w)}_{\bar{u}_n\bar{u}_l}((\mathbb{H}_{u_ku_n}(a_1);
\bar{u}_n))\\
&=
\mathbb{H}^{un(w)}_{\bar{u}_n\bar{u}_l}(\mathbb{H}^{un(w)}_{\bar{u}_k\bar{u}_n}((a_1,
\bar{u}_k)))
\end{align*}
so that we get that $\mathbb{H}^{un(w)}_{\bar{u}_k\bar{u}_k} =
id_{A^{un(w)}_{\bar{u}_k}}$ and that
$\mathbb{H}^{un(w)}_{\bar{u}_k\bar{u}_l} =
\mathbb{H}^{un(w)}_{\bar{u}_n\bar{u}_l}\circ\mathbb{H}^{un(w)}_{\bar{u}_k\bar{u}_n}$.

Part 2 follows by a straightforward check of the conditions of
Definition \ref{D:asimulation}, and Parts 3 and 4 then follow from
Part 2 by the preservation of intuitionistic formulas under
asimulations.
\end{proof}
The most important result of an intuitionistic unravelling is, of course, the accessibility relation of the unravelled model which displays a number of useful properties. We will therefore call an arbitrary $(\mathcal{M}, w) \in Pmod_\Theta$ an unravelled model iff there is an $(\mathcal{N}, v) \in Pmod_\Theta$ and a bijection $h: W \to U^{un(v)}$ satisfying the condition \eqref{E:c1} of Definition \ref{D:isomorphism} such that $h(w) =v$; or, in other words, if the underlying frame $(W, \prec)$ of $(\mathcal{M}, w)$ is isomorphic to the underlying frame of $(\mathcal{N}^{un(v)}, v)$. It is clear that if $\mathcal{M}$ is an unravelled model, we may always assume that its underlying frame is just a copy of the underlying frame of the respective unravelled model. We will assume, therefore, for the sake of simplicity, that if $(\mathcal{M}, w) \in Pmod_\Theta$ is an unravelled model, then $\mathcal{M}$ is given in the form $\langle W^{un(w)}, \prec^{un(w)}, \mathfrak{A}, \mathbb{H}\rangle$.

\section{The family of standard intuitionistic logics}\label{S:Standard}

When equality symbol is allowed in intuitionistic first-order
logic, it can be interpreted either as \emph{intensional} identity or as \emph{extensional} equality.\footnote{These different kinds of equality can already be traced back to \cite{hey}.} In the context of a more common Kripke semantics of increasing domains this distinction corresponds to either adding to the first-order language a `true' equality, or a binary predicate that behaves as a congruence on the elements of the domain. 

However, in the context of the version of Kripke semantics employed in the present paper, the same distinction boils down to the question of whether the injectivity of canonical homomorphisms is enforced or omitted. This has the advantage of reading both types of equality as variants of `true' equality, which is important, given that the extensional equality is often viewed as the more natural intuitionistic equality notion.

Proceeding now to the more precise definitions, we stipulate that, for a given signature $\Sigma$, the language $IL^\equiv(\Sigma)$ is obtained by adding to the set of
atomic $IL$-formulas all formulas of the form $t_1\equiv t_2$ for all $\Sigma$-terms $t_1$ and $t_2$, 
and then closing under the
applications of connectives and quantifiers.

We then also need to define $\models_{\mathsf{IL}^\equiv}$ by
adding the following clause in the inductive definition of
$\models_{\mathsf{IL}}$:
$$
\mathcal{M}, w \models_{\mathsf{IL}^\equiv} t_i \equiv
t_j[\bar{a}_n] \Leftrightarrow \alpha\langle t_i\rangle = \alpha\langle t_j\rangle,
$$
where $\alpha$ is defined as in Lemma \ref{L:satisfaction}.1.

We now define four further logics based on the language $IL^\equiv$ and the relation $\models_{\mathsf{IL}^\equiv}$ applied to the general set of Kripke models, and the classes $Su$, $In$, and $Bi$, respectively, which we will denote by $\mathsf{IL}^\equiv$, $\mathsf{CD}^\equiv$, $\mathsf{In}^\equiv$, and $\mathsf{Bi}^\equiv$, respectively. We will interpret  $\mathsf{IL}^\equiv$ (resp. $\mathsf{CD}^\equiv$) as the basic (resp. constant-domain) first-order intuitionistic logic of extensional equality, and we will interpret $\mathsf{In}^\equiv$ (resp. $\mathsf{Bi}^\equiv$) as the basic (resp. constant-domain) first-order intuitionistic logic of intensional equality. Thus both extensional and intensional equality versions of $\mathsf{IL}$, $\mathsf{CD}$ can be viewed as extensions of these logics, as long as all we care about when handling logics are their sets of theorems; semantically, however, extensional and intensional equality must be viewed as super-imposed upon different (even though equivalent, in a sense) semantics for their underlying logics.

The two versions of equality also clearly lead to differences in the theorem sets in that we have $\not\models_{\mathsf{CD}^\equiv} (\forall x,y)(x \equiv y
\vee \neg(x \equiv y))$ (and thus also $\not\models_{\mathsf{IL}^\equiv} (\forall x,y)(x \equiv y
\vee \neg(x \equiv y))$) on the one hand, but $\models_{\mathsf{In}^\equiv} (\forall x,y)(x \equiv y
\vee \neg(x \equiv y))$ (and thus also $\models_{\mathsf{Bi}^\equiv} (\forall x,y)(x \equiv y
\vee \neg(x \equiv y))$) on the other hand.

All in all, these combinations of equality notions give us six
different logics. The set $\{ \mathsf{IL}, \mathsf{IL}^\equiv,
\mathsf{In}^\equiv, \mathsf{CD},
\mathsf{CD}^\equiv, \mathsf{Bi}^\equiv\}$ will be called in what follows the set of
standard intuitionistic logics, $StIL$, for short.

Although the extensional equality may sometimes be seen as `non-rigid',
this circumstance does not prevent it from displaying all the
benchmark properties that expected from equality by analogy with
classical logic. Namely, the following lemma holds:
\begin{lemma}\label{L:approx-properties}
For every signature $\Sigma$, every $(\mathcal{M},w) \in
Pmod_\Sigma$ and any $x, y, z, \bar{x}_n,\bar{y}_n \in Var$, it is
true that:
\begin{align*}
    \mathcal{M}, w \models_{\mathsf{IL}^\equiv} \{\forall x&(x \equiv x),
    (\forall x, y)(x \equiv y \to y \equiv x), (\forall x,y,z)(x
    \equiv y \wedge y \equiv z \to x \equiv z) \} \cup\\&\cup \{
    (\forall\bar{x}_n,\bar{y}_n)(\bigwedge^n_{i = 1}(x_i \equiv
    y_i) \to (P(\bar{x}_n) \leftrightarrow P(\bar{y}_n))) \mid n \in \omega,\,P
    \in \Sigma_n\}
\end{align*}
\end{lemma}

We can now extend the semantic machinery of the previous
subsection to these new logics, in particular, the notions of a
type, a theory, an asimulation and an elementary embedding. All of the
definitions of the previous sections can be easily adapted to any
of the logics in $StIL$ and for every proposition stated above
about $\mathsf{IL}$ an obvious analogue can be proven for any
standard intuitionistic logic. Note that, e.g., $\mathsf{CD}$-asimulation is just an $\mathsf{IL}$-asimulation between surjective models, and that, for the case of
$\mathsf{IL}^\equiv$ we have, among other things, that whenever $A$
is an $\mathsf{IL}^\equiv$-asimulation, and
$(w;\bar{\alpha}_l)A(v;\bar{\beta}_l)$, then $\bar{\alpha}_l
\mapsto \bar{\beta}_l$ is a bijection.

Finally, it is clear that the model theory of $\mathsf{IL}$,
developed in the previous sections, carries over to any other
logic in $StIL$, with minimal changes in formulations and proofs.
As an example, we formulate here the respective generalizations of
Lemmas \ref{unravellinglemma} and \ref{L:asimulations}:
\begin{lemma}\label{unravellinglemma-gen}
Let $\mathcal{L} \in StIL$ and let $(\mathcal{M}, w) \in
Pmod_\Theta(\mathcal{L})$. Then:
\begin{enumerate}
\item $\mathcal{M}^{un(w)} \in Pmod_\Theta(\mathcal{L})$;

\item Define
$$
B := \{\langle (w_n; \bar{a}_m), (\bar{w}_n;
\bar{a}_m\odot\bar{w}_n)\rangle, \langle (\bar{w}_n;
\bar{a}_m\odot\bar{w}_n), (w_n; \bar{a}_m)\rangle \mid \bar{w}_n
\in W^{un(w)},\,m\geq 0,\,\bar{a}_m \in A^m_{w_n}\}.
$$
Then $B$ is an $\mathcal{L}$-asimulation both from $(\mathcal{M},
w_n, \bar{a}_m)$ to $(\mathcal{M}^{un(w)}, \bar{w}_n,
\bar{a}_m\odot\bar{w}_n)$ and from $(\mathcal{M}^{un(w)},
\bar{w}_n, \bar{a}_m\odot\bar{w}_n)$ to $(\mathcal{M}, w_n,
\bar{a}_m)$ for every $w_n \in W$ such that $w_n \succ w$ and
every tuple $\bar{a}_m \in A^m_{\bar{w}_n}$;

\item $Tp_{\mathcal{L}}(\mathcal{M}, v_k, \bar{a}_m) =
Tp_{\mathcal{L}}(\mathcal{M}^{un(w)}, \bar{v}_k, \bar{a}_m\odot\bar{v}_k)$
for any $\bar{v}_k \in W^{un(w)}$, any $m\geq 0$, and any
$\bar{a}_m \in A^m_{v_k}$; 
\item $Tp_{\mathcal{L}}(\mathcal{M}^{un(w)}, \bar{u}_n,
\bar{a}_m\odot\bar{u}_n) = Tp_{\mathcal{L}}(\mathcal{M}^{un(w)},
\bar{v}_k, \bar{a}_m\odot\bar{v}_k)$ 

 for any $\bar{u}_n, \bar{v}_k
\in W^{un(w)}$ whenever $u_n = v_k$, $m\geq 0$, and $\bar{a}_m \in
A^m_{v_k}$. 
\end{enumerate}
\end{lemma}
\begin{lemma}\label{L:asimulations-gen}
Let $\mathcal{L} \in StIL$, let $(\mathcal{M}_1, w_1)$,
$(\mathcal{M}_2, w_2)$ be pointed $\Theta$-models, let $\bar{a}_n
\in (A_1)^n_{w_1}$, and let $\bar{b}_n \in (A_2)^n_{w_2}$, and let
$\bar{c}_{n}$ be a tuple of pairwise distinct individual constants
outside $\Theta$. If we have

\[ Tp^+_{\mathcal{L}}(\mathcal{M}_1, w_1,
\bar{c}_{n}/\bar{a}_n)\subseteq Tp^+_{\mathcal{L}}(\mathcal{M}_2,
w_2, \bar{c}_{n}/\bar{b}_n)
\]
for every tuple $\bar{c}_{n}$ of pairwise distinct individual
constants outside $\Theta$, and both $\mathcal{M}_1$ and
$\mathcal{M}_2$ are $\mathcal{L}$-saturated, then the relation $A$
such that for all $\{ i,j \} = \{ 1,2 \}$, all $u \in W_i$, $s \in
W_j$, $\bar{\alpha}_k \in (A_i)^k_{u}$, and all $\bar{\beta}_k \in
(A_j)^k_{s}$ we have
$$
(u; \bar{\alpha}_k) \mathrel{A}(s; \bar{\beta}_k) \Leftrightarrow
Tp^+_{\mathcal{L}}(\mathcal{M}_i, u,
\bar{c}_{k}/\bar{\alpha}_k)\subseteq
Tp^+_{\mathcal{L}}(\mathcal{M}_j, s, \bar{c}_{k}/\bar{\beta}_k)
$$
for any given set $\{c_i\mid i < \omega\}$ of pairwise distinct
individual constants outside $\Theta$ is an
$\mathcal{L}$-asimulation from $(\mathcal{M}_1, w_1, \bar{a}_n)$
to $(\mathcal{M}_2, w_2, \bar{b}_n)$.
\end{lemma}

\section{Abstract intuitionistic logics and their properties}\label{S:Abstract}

\subsection{Abstract intuitionistic logics}
An abstract intuitionistic logic $\mathcal{L}$ is a quadruple
$(Str_\mathcal{L}, L, \models_\mathcal{L}, \boxplus_\mathcal{L})$, where
$Str_\mathcal{L}$ is a function returning, for every signature
$\Theta$, the class of $\mathcal{L}$-admissible pointed
$\Theta$-models $Str_\mathcal{L}(\Theta) \subseteq Pmod_\Theta$.

This class is assumed to satisfy the following closure conditions:
\begin{itemize}
\item (Closure for isomorphisms) For any signature $\Theta$, any
$(\mathcal{M},w), (\mathcal{N},v) \in Pmod_\Theta$, if $(g,h):
(\mathcal{M},w)\cong (\mathcal{N},v)$, and $(\mathcal{M},w)
\in Str_\mathcal{L}(\Theta)$, then $(\mathcal{N},v) \in
Str_\mathcal{L}(\Theta)$.

\item (Closure for renamings) Let signatures $\Theta$ and
$\Sigma$, functions $f$ and $g$, and models $(\mathcal{M},w) \in
Pmod_\Theta$ and $(\mathcal{N},w) \in Pmod_\Sigma$ be given such
that $\Sigma$ is an $(f,g)$-renaming of $\Theta$ and
$(\mathcal{N},w)$ is an $(f,g)$-renaming of $(\mathcal{M},w)$.
Then, whenever $(\mathcal{M},w) \in Str_\mathcal{L}(\Theta)$, we
also have $(\mathcal{N},w) \in Str_\mathcal{L}(\Sigma)$.

\item (Closure for the model component) For any signature
$\Theta$, any $\mathcal{M}\in Mod_\Theta$, and any $w,v \in W$, if
$(\mathcal{M},w) \in Str_\mathcal{L}(\Theta)$, then
$(\mathcal{M},v) \in Str_\mathcal{L}(\Theta)$.

\item (Closure for reducts) For any signatures $\Theta$ and
$\Sigma$, any $(\mathcal{M},w)\in Pmod_\Theta$,
if $\Sigma \subseteq \Theta$ and $(\mathcal{M},w) \in
Str_\mathcal{L}(\Theta)$, then $(\mathcal{M}\upharpoonright\Sigma
,w) \in Str_\mathcal{L}(\Sigma)$. 

\item (Closure for constant extensions) For any signature
$\Theta$, any $(\mathcal{M},w) \in Pmod_\Theta$, any $n < \omega$,
any tuple $\bar{c}_n$ of pairwise distinct constants outside
$\Theta$, and any tuple $\bar{a}_n \in A^n_w$, if $(\mathcal{M},w)
\in Str_\mathcal{L}(\Theta)$, then there is at least one
$\mathcal{N}\in (\mathcal{M},w)\oplus(\bar{c}_n/\bar{a}_n)$ such
that $(\mathcal{N},w) \in Str_\mathcal{L}(\Theta\cup \bar{c}_n)$.

\item (Closure for the unions of countable model chains) For any signature
$\Theta$, and for any chain of models $\mathcal{M}_1 \subseteq,\ldots, \subseteq \mathcal{M}_n\subseteq,\ldots$ such that we have $(\mathcal{M}_i,w) \in
Str_\mathcal{L}(\Theta)$ for all $i > 0$, we also have $(\bigcup_{i > 0}\mathcal{M}_i, w) \in
Str_\mathcal{L}(\Theta)$.
\end{itemize}
The second component, $L$, is then a function returning, for a given signature $\Theta$,
the set $L(\Theta)$ of $\Theta$-sentences in $\mathcal{L}$; next,
$\models_\mathcal{L}$ is a class-relation such that, if
$\alpha\mathrel{\models_\mathcal{L}}\beta$, then there exists a
signature $\Theta$ such that $\alpha \in Str_\mathcal{L}(\Theta)$,
and $\beta\in L(\Theta)$; informally this is to mean that $\beta$
holds in $\alpha$. The relation $\models_\mathcal{L}$ is only assumed to be defined (i.e. to either hold or fail) for the elements of the class 
$\bigcup\{(Str_\mathcal{L}(\Theta), L(\Theta))\mid \Theta\text{ is a signature}\}$ 
and to be undefined otherwise.  


The fourth element in our quadruple, $\boxplus_\mathcal{L}$, is then a function, returning, for every $(\mathcal{M}, w)\in Str_\mathcal{L}(\Theta)$, every tuple $\bar{c}_n$ of pairwise distinct constants outside
$\Theta$, and every tuple $\bar{a}_n \in A^n_w$ a set $\emptyset \neq (\mathcal{M},
w)\boxplus_\mathcal{L}(\bar{c}_n/\bar{a}_n) \subseteq (\mathcal{M},w)\oplus(\bar{c}_n/\bar{a}_n)$ such that:
\begin{itemize}
	\item If $\mathcal{N} \in (\mathcal{M},
	w)\boxplus_\mathcal{L}(\bar{c}_n/\bar{a}_n)$, then, for all $v \in U$, $(\mathcal{N},v) \in Str_\mathcal{L}(\Theta\cup \bar{c}_n)$.
	
	\item If $\mathcal{N} \in (\mathcal{M},
	w)\boxplus_\mathcal{L}(\bar{c}_n/\bar{a}_n)$ and $\mathcal{N}' \in (\mathcal{M},w)\oplus(\bar{c}_n/\bar{a}_n)$ is such that for all
	 $\phi \in L(\Theta\cup \bar{c}_n)$ it is true that $\mathcal{N}', w\models_\mathcal{L} \phi \Leftrightarrow  \mathcal{N}, w\models_\mathcal{L} \phi$, then $\mathcal{N}' \in (\mathcal{M},
	w)\boxplus_\mathcal{L}(\bar{c}_n/\bar{a}_n)$. 
\end{itemize}

In order for such a quadruple $\mathcal{L}$ to count as an abstract
intuitionistic logic, we demand that the following conditions are
satisfied:
\begin{itemize}
\item (Occurrence). 
If $\phi \in L(\Theta)$ for some signature
$\Theta$, then there is a $\Theta_\phi \Subset \Theta$ such that
$\phi \in L(\Theta_\phi)$ and for every signature $\Theta'$, we have $\phi \in
L(\Theta')$ iff $\Theta_\phi \subseteq \Theta'$.

\item (Renaming). For all signatures $\Theta$, $\Theta'$, if
functions $f,g$ are such that $\Theta'$ is an $(f,g)$-renaming of
$\Theta$, then there is a bijection $\tau_{(f,g)}:L(\Theta) \to
L(\Theta')$ such that for every $\phi \in L(\Sigma)$ we have
$\Theta_{\tau_{(f,g)}(\phi)} = (f \cup g)(\Theta_\phi)$ and,
whenever  $(\mathcal{M}, w) \in Str_\mathcal{L}(\Theta)$ and
$(\mathcal{N}, w)$ is an $(f,g)$-renaming of $(\mathcal{M}, w)$,
we have:
$$
\mathcal{M}, w\models_\mathcal{L} \phi \Leftrightarrow
\mathcal{N}, w\models_\mathcal{L} \tau_{(f,g)}(\phi).
$$
\item (Isomorphism). If $(\mathcal{M}, w)\in
Str_\mathcal{L}(\Theta)$, $\phi \in L(\Theta)$, and $(g,
h):\mathcal{M}\cong\mathcal{N}$, then:
$$
\mathcal{M}, w\models_\mathcal{L} \phi \Leftrightarrow
\mathcal{N}, h(w) \models_\mathcal{L} \phi.
$$

\item (Expansion). Let $\Theta \subseteq \Theta'$ be signatures,
and let $(\mathcal{M}, w) \in Str_\mathcal{L}(\Theta')$. Then for
any $\phi \in L(\Theta)$, we have:
$$
\mathcal{M}, w\models_\mathcal{L} \phi \Leftrightarrow
\mathcal{M}\upharpoonright\Theta, w\models_\mathcal{L} \phi.
$$

\item (Propositional Closure). For every signature $\Theta$ we
have $\bot \in L(\Theta)$, and, if $\phi, \psi \in L(\Theta)$,
then we also have
\[
\phi \to \psi, \phi \wedge \psi, \phi \vee \psi \in L(\Theta),
\]
 that is to say, $\mathcal{L}$ is closed
under intuitionistic implication, conjunction and disjunction.

\item (Quantifier closure). For every signature $\Theta$, every
constant $c \notin Const_\Theta$, every $(\mathcal{M}, w)\in
Str_\mathcal{L}(\Theta)$, and every $\phi \in L(\Theta\cup
\{c\})$, there exist $\exists c\phi, \forall c\phi \in L(\Theta)$
such that:
$$
\mathcal{M}, w \models_\mathcal{L}\exists c\phi \Leftrightarrow
(\exists a \in A_w)(\exists\mathcal{N}\in (\mathcal{M},
w)\boxplus_\mathcal{L}(c/a))(\mathcal{N}, w \models_\mathcal{L} \phi),
$$
and:
$$
\mathcal{M}, w \models_\mathcal{L}\forall c\phi \Leftrightarrow
(\forall v\mathrel{\succ}w)(\forall a \in
A_v)(\exists\mathcal{N}\in (\mathcal{M},
w)\boxplus_\mathcal{L}(c/a))(\mathcal{N}, v \models_\mathcal{L} \phi).
$$
\end{itemize}

We further define that given a pair of abstract intuitionistic
logics $\mathcal{L}$ and $\mathcal{L}'$, we say that
$\mathcal{L}'$ extends $\mathcal{L}$ and write $\mathcal{L}
\sqsubseteq \mathcal{L}'$ iff all of the following holds:
\begin{itemize}
    \item $Str_\mathcal{L}
= Str_\mathcal{L'}$;
	\item If $(\mathcal{M}, w)\in
	Str_\mathcal{L}(\Theta)$, $\bar{c}_n$ is a tuple of pairwise distinct constants outside
	$\Theta$, and $\bar{a}_n \in A^n_w$, then we have $(\mathcal{M},
	w)\boxplus_{\mathcal{L}'}(\bar{c}_n/\bar{a}_n) \subseteq (\mathcal{M},
	w)\boxplus_\mathcal{L}(\bar{c}_n/\bar{a}_n)$

    \item For
    every $\phi \in L(\Theta_\phi)$ there exists a $\psi \in
    L'(\Theta_\phi)$ such that $\Theta_\phi = \Theta_\psi$, and
    such that for every $(\mathcal{M}, w) \in
    Str_\mathcal{L}(\Theta_\phi)$ we have:
    $$
\mathcal{M}, w\models_\mathcal{L}\phi \Leftrightarrow \mathcal{M},
w\models_{\mathcal{L}'}\psi.
    $$
\end{itemize}
If both $\mathcal{L} \sqsubseteq \mathcal{L}'$ and $\mathcal{L}'
\sqsubseteq \mathcal{L}$ holds, then we say that the logics
$\mathcal{L}$ and $\mathcal{L}'$ are \emph{expressively
equivalent} and write $\mathcal{L} \bowtie \mathcal{L}'$.

It is easy to see that all the systems in $StIL$ are abstract
intuitionistic logics; moreover, the injective variants of $\mathsf{IL}$ and $\mathsf{CD}$ have to be considered as abstract intuitionistic logics that are distinct from their non-injective variants. The logics in question can be presented along the lines of
the above definition as follows:
$$
\mathsf{IL} = (\Theta \mapsto Pmod_\Theta, \Theta \mapsto
IL_{\emptyset}(\Theta), \models_{\mathsf{IL}}, \oplus),
$$
$$
\mathsf{In} = (\Theta \mapsto Pmod_\Theta(In), \Theta \mapsto
IL_{\emptyset}(\Theta), \models_{\mathsf{IL}}, \oplus),
$$
$$
\mathsf{IL}^\equiv = (\Theta \mapsto Pmod_\Theta, \Theta \mapsto
IL^\equiv_{\emptyset}(\Theta), \models_{\mathsf{IL}^\equiv}, \oplus),
$$
$$
\mathsf{In}^\equiv = (\Theta \mapsto Pmod_\Theta(In), \Theta \mapsto
IL^\equiv_{\emptyset}(\Theta), \models_{\mathsf{IL}^\equiv}, \oplus),
$$
$$
\mathsf{CD} = (\Theta \mapsto Pmod_\Theta(Su), \Theta \mapsto
IL_{\emptyset}(\Theta), \models_{\mathsf{IL}}, \oplus),
$$
$$
\mathsf{Bi} = (\Theta \mapsto Pmod_\Theta(Bi), \Theta \mapsto
IL_{\emptyset}(\Theta), \models_{\mathsf{IL}}, \oplus),
$$
$$
\mathsf{CD}^\equiv = (\Theta \mapsto Pmod_\Theta(Su), \Theta
\mapsto IL^\equiv_{\emptyset}(\Theta),
\models_{\mathsf{IL}^\equiv}, \oplus),
$$
$$
\mathsf{Bi}^\equiv = (\Theta \mapsto Pmod_\Theta(Bi), \Theta
\mapsto IL^\equiv_{\emptyset}(\Theta),
\models_{\mathsf{IL}^\equiv}, \oplus),
$$

In addition, as one can easily see, we have  $\mathcal{L} \sqsubseteq
\mathcal{L}^\equiv$ for every $\mathcal{L} \in \{\mathsf{IL}, \mathsf{In}, \mathsf{CD}, \mathsf{Bi}\}$ but we also have $\mathsf{IL} \not\sqsubseteq \mathsf{In}^\equiv$ and $\mathsf{CD} \not\sqsubseteq \mathsf{Bi}^\equiv$ showing that the extensional equality can be easier represented as an extension of equality-free intuitionistic logic than the intensional equality.

\subsection{Abstract extensions of standard intuitionistic logics}
In this paper, our specific interest is in the extensions of
logics from $StIL$. If an abstract intuitionistic logic
$\mathcal{L}$ extends a logic $\mathcal{L}' \in StIL$, then
we must have $Str_\mathcal{L} = Str_{\mathcal{L}'}$ and we will
assume that every $\phi \in L'(\Theta)$ is also present in
$L(\Theta)$ with its usual meaning as given by $\models$-relation
in any of the two logics involved.

\begin{example}  Consider now the infinitary (but countable) disjunction $\bigvee$ and conjunction $\bigwedge$ with the obvious semantics. We can obtain infinitary expressive extensions of all the logics in $StIL$  by adding the aforementioned connectives along the lines of \cite{na}. All these can be accomodated to satisfy the definition of  an abstract intuitionistic logic.

\end{example}

For any given extension of standard intuitionistic logics an
obvious analogue of every definition given above for $\mathsf{IL}$
can be formulated; however, the situation here is different from
the logics in $StIL$ in that the same type of model-theoretic
treatment cannot be, generally speaking, given for any extension
of a standard intuitionistic logic. In particular, statements like
Lemma \ref{L:satisfaction}.2 and Lemma \ref{L:satisfaction}.3
might fail (for Lemma \ref{L:satisfaction}.2 consider, e.g., the
case of first-order bi-intuitionistic logic \cite{rau} which is an
extension of $\mathsf{CD}$, and for Lemma \ref{L:satisfaction}.3
consider an extension of $\mathsf{IL}$ with the Boolean negation connective). One result of such failures may be
that the analogues of Corollaries \ref{C:satisfaction-constants}
and \ref{C:types-and-constants} may fail, so that the quantifiers
will not have their intended meaning and the notion of the type,
as defined in the context of $\mathsf{IL}$ above, will be
inadequate.

The notions of asimulation and saturation, however, prove to be
more robust in that the following corollary to Lemma
\ref{L:asimulations-gen} still can be proven for arbitrary
extensions of the standard intuitionistic logics:

\begin{corollary}\label{L:asimulationscorollary}
Let $\mathcal{L}' \in StIL$, let $\mathcal{L}' \sqsubseteq
\mathcal{L}$, let $(\mathcal{M}_1, w_1), (\mathcal{M}_2, w_2) \in
Str_\mathcal{L}(\Theta)$, let $\bar{a}_n \in (A_1)^n_{w_1}$, and
let $\bar{b}_n \in (A_2)^n_{w_2}$. If

\[ Tp^+_{\mathcal{L}'}(\mathcal{M}_1, w_1,
\bar{a}_n)\subseteq Tp^+_{\mathcal{L}'}(\mathcal{M}_2, w_2,
\bar{b}_n)
\]
 and both $\mathcal{M}_1$ and
$\mathcal{M}_2$ are $\mathcal{L}$-saturated, then the relation $A$
such that for all $u \in W_i$, $s \in W_j$, $\bar{\alpha}_k \in
(A_i)^k_{u}$, and $\bar{\beta}_k \in (A_j)^n_{s}$, if $\{ i,j \} =
\{ 1,2 \}$, then
$$
(u; \bar{\alpha}_k) \mathrel{A}(s; \bar{\beta}_k) \Leftrightarrow
Tp^+_{\mathcal{L}'}(\mathcal{M}_i, u, \bar{\alpha}_k) \subseteq
Tp^+_{\mathcal{L}'}(\mathcal{M}_j, s, \bar{\beta}_k)
$$
is an $\mathcal{L}'$-asimulation from $(\mathcal{M}_1, w_1,
\bar{a}_n)$ to $(\mathcal{M}_2, w_2, \bar{b}_n)$.
\end{corollary}
To prove this, we just repeat the proof of the version of Lemma
\ref{L:asimulations} for the respective $\mathcal{L}'$, using the
fact that every $\mathcal{L}$-saturated model is of course
$\mathcal{L}'$-saturated.

However, it makes much more sense in general to focus on the more
well-behaved subclasses of extensions of any given logic in
$StIL$. The more regular behavior of logics in these subclasses is
typically due to the fact that these logics display useful
model-theoretic properties. We define some of the relevant
properties below.
\begin{definition}\label{D:properties}
Let $\mathcal{L}' \in StIL$ and let $\mathcal{L}$ be an
abstract intuitionistic logic. Then:
\begin{itemize}
\item $\mathcal{L}$ is \textbf{preserved under
$\mathcal{L}'$-asimulations}, iff for all signatures $\Theta$,
every tuple $\bar{c}_n$ of pairwise distinct constants outside
$\Theta$, all $(\mathcal{M}_1, w_1), (\mathcal{M}_2, w_2)\in
Str_\mathcal{L}(\Theta)$, whenever $A$ is an
$\mathcal{L}'$-asimulation from $(\mathcal{M}_1, w_1)$ to
$(\mathcal{M}_2, w_2)$, then for every $\phi \in L(\Theta)$ we
have:
$$
\mathcal{M}_1, w_1 \models_\mathcal{L}\phi \Rightarrow
\mathcal{M}_2, w_2 \models_\mathcal{L}\phi.
$$

\item $\mathcal{L}$ is \textbf{$\star$-compact}, iff for every
signature $\Theta$ and every pair $(\Gamma, \Delta)\subseteq
(L(\Theta), L(\Theta))$, there exists an $(\mathcal{M}, w)\in
Str_\mathcal{L}(\Theta)$ such that $\mathcal{M},
w\models_\mathcal{L} (\Gamma, \Delta)$, whenever for all $\Gamma'
\Subset \Gamma$ and $\Delta' \Subset \Delta$, there exists an
$(\mathcal{M}', w')\in Str_\mathcal{L}(\Theta)$ such that
$\mathcal{M}', w'\models_\mathcal{L} (\Gamma', \Delta')$.

\item $\mathcal{L}$ has the \textbf{Tarski Union Property (TUP)} iff
for every $\mathcal{L}$-elementary chain

$$
\mathcal{M}_0 \preccurlyeq_\mathcal{L},\ldots,
\preccurlyeq_\mathcal{L} \mathcal{M}_n
\preccurlyeq_\mathcal{L},\ldots
$$
it is true that:
$$
\mathcal{M}_n \preccurlyeq_\mathcal{L} \bigcup_{n \in
\omega}\mathcal{M}_n
$$

for all $n \in \omega$.
\end{itemize}
\end{definition}
Among these properties, the preservation under asimulations seems to
be particularly useful in that it immediately yields several
propositions that bring the model theory of the logic in question
to that of the logics in $StIL$. More precisely, we can prove the
following lemmas and corollaries.
\begin{lemma}\label{L:generated-submodels}
Let $\mathcal{L}$ be an abstract intuitionistic logic and let
$\mathcal{L}' \in StIL$ be such that  $\mathcal{L}$ is preserved
under $\mathcal{L}'$-asimulations. Let $\Theta$ be a signature,
and let $(\mathcal{M},w) \in Str_{\mathcal{L}}(\Theta)$. Then, for
every $\phi \in L(\Theta)$ and every $v \in [W,w]$ we have:
\begin{enumerate}
    \item $\mathcal{M},w \models_\mathcal{L} \phi \Leftrightarrow
[\mathcal{M},w],w \models_\mathcal{L} \phi$.
    \item $\mathcal{M},w \models_\mathcal{L} \phi \Rightarrow
\mathcal{M},v \models_\mathcal{L} \phi$.
\end{enumerate}
\end{lemma}
\begin{proof}
(Part 1). We define the relation $B$ setting:
$$
B:= \{ \langle (u; \bar{a}_k), (u; \bar{a}_k)\rangle \mid u \in
[W,w],\,k < \omega,\,\bar{a}_k \in A^k_u\}.
$$
It is then immediate to check that $B$ is an
$\mathcal{L}'$-asimulation both from $(\mathcal{M}, w)$ to
$([\mathcal{M},w], w)$ and from $([\mathcal{M},w], w)$ to
$(\mathcal{M}, w)$.

(Part 2). We define the relation $A$ setting:
$$
A:= \{\langle (w; \bar{a}_k), (v;
\mathbb{H}_{wv}\langle\bar{a}_k\rangle)\rangle\mid k < \omega,\,
\bar{a}_k \in A^k_w\}\cup \{ \langle (u; \bar{b}_k), (u;
\bar{b}_k)\rangle \mid u \in [W,v],\,k < \omega,\,\bar{a}_k \in
A^k_u\}.
$$
It is then immediate to check that $A$ is an
$\mathcal{L}'$-asimulation
 from $(\mathcal{M}, w)$ to
$(\mathcal{M}, v)$.

\end{proof}
Note that Lemma \ref{L:generated-submodels} holds also for the
logics that are not necessarily extensions of standard
intuitionistic logics as long as they are preserved under
$\mathcal{L}'$-asimulations for some $\mathcal{L}' \in StIL$. Some
of the consequences of this lemma for the logics that are, in
addition, the extensions of the respective standard intuitionistic
logics, are that the definition of the types now have to make
sense in the context of any such logic and the intutitionistic
simplifications of the quantifier clauses are available:
\begin{corollary}\label{C:generated-submodels}
Let
  $\mathcal{L}$ be an abstract intuitionistic logic and let
$\mathcal{L}' \in StIL$ be such that $\mathcal{L} \sqsupseteq
\mathcal{L}'$ and $\mathcal{L}$ is preserved under
$\mathcal{L}'$-asimulations. Let $\Theta$ be a signature and let
$(\mathcal{M},w) \in Str_{\mathcal{L}}(\Theta)$. Then:
\begin{enumerate}
\item $\boxplus_{\mathcal{L}} = \oplus$.
	
\item For every $\bar{a}_n \in A^n_w$, every tuple $\bar{c}_n$ of
pairwise distinct constants outside $\Theta$, and every $\phi \in
L(\Theta \cup \{\bar{c}_n\})$ we have:
\begin{align*}
(\exists\mathcal{N}\in (\mathcal{M},
w)\boxplus_\mathcal{L}(\bar{c}_n/\bar{a}_n))(\mathcal{N}, w
\models_\mathcal{L}\phi)&\Leftrightarrow (\forall\mathcal{N}\in
(\mathcal{M}, w)\boxplus_\mathcal{L}(\bar{c}_n/\bar{a}_n))(\mathcal{N}, w
\models_\mathcal{L}\phi)\\
&\Leftrightarrow ([\mathcal{M},w], \bar{c}_n/\bar{a}_n),
w\models_\mathcal{L}\phi\\
&\Leftrightarrow \phi \in
Tp^+_\mathcal{L}(\mathcal{M},w,\bar{a}_n)
\end{align*}

\item For every constant $c \notin Const_\Theta$ and every $\phi
\in L(\Theta\cup \{c\})$ it is true that:
$$
\mathcal{M}, w \models_\mathcal{L}\exists c\phi \Leftrightarrow
(\exists a \in A_w)(([\mathcal{M}, w], c/a),w \models_\mathcal{L}
\phi),
$$
and:
$$
\mathcal{M}, w \models_\mathcal{L}\forall c\phi \Leftrightarrow
(\forall v\mathrel{\succ}w)(\forall a \in A_v)(([\mathcal{M}, w],
c/a), v \models_\mathcal{L} \phi).
$$

\item	For all $\bar{b}_m
\in A^m_w$, and for every tuple $\bar{c}_{m + 1}$ of pairwise distinct individual constants outside $\Theta$, if
$\varphi\in L(\Theta \cup \{\bar{c}_{m+1}\})$,
 then we have:
	\begin{align*}
	&([\mathcal{M}, w], \bar{c}_m/\bar{b}_m), w\models_{\mathcal{L}}\exists c_{m + 1}\varphi
	\Leftrightarrow (\exists b \in A_w)(([\mathcal{M}, w], \bar{c}_{m + 1}/(\bar{b}_m)^\frown b), w \models_{\mathcal{L}} \varphi);\\
	&([\mathcal{M}, w], \bar{c}_m/\bar{b}_m), w \models_{\mathcal{L}} \forall
	c_{m + 1}\varphi \Leftrightarrow (\forall v \succ w)(\forall b
	\in A_v)(([\mathcal{M}, v], \bar{c}_{m + 1}/\mathbb{H}_{wv}\langle\bar{b}_m\rangle^\frown b), v \models_{\mathcal{L}}
	\varphi).
\end{align*}
\end{enumerate}
\end{corollary}
\begin{proof}
(Part 1). Just by definition, for every $(\mathcal{M}, w)\in
Str_\mathcal{L}(\Theta)$, every tuple $\bar{c}_n$ of fresh pairwise distinct constants, and every $\bar{a}_n \in A^n_w$, we have that $(\mathcal{M},
w)\boxplus_{\mathcal{L}}(\bar{c}_n/\bar{a}_n) \subseteq (\mathcal{M},
w)\oplus(\bar{c}_n/\bar{a}_n)$. In the other direction, note that if $\mathcal{N} \in (\mathcal{M},
w)\oplus(\bar{c}_n/\bar{a}_n)$ and $\mathcal{N}' \in (\mathcal{M},
w)\boxplus_{\mathcal{L}}(\bar{c}_n/\bar{a}_n)$, then we have, for every  $\phi \in L(\Theta\cup \bar{c}_n)$, that:
\begin{align*}
	\mathcal{N}, w\models_\mathcal{L} \phi &\Leftrightarrow  [\mathcal{N}, w], w\models_\mathcal{L} \phi&&\text{(by Lemma \ref{L:generated-submodels}.1)}\\
	&\Leftrightarrow ([\mathcal{M},w], \bar{c}_n/\bar{a}_n), w\models_\mathcal{L} \phi&&\text{(by Lemma \ref{L:cutoff})}\\
	 &\Leftrightarrow  [\mathcal{N}', w], w\models_\mathcal{L} \phi&&\text{(by Lemma \ref{L:cutoff})}\\
	  &\Leftrightarrow 	\mathcal{N}', w\models_\mathcal{L} \phi&&\text{(by Lemma \ref{L:generated-submodels}.1)}	
\end{align*}
Therefore, we must also have $\mathcal{N} \in (\mathcal{M},
w)\boxplus_{\mathcal{L}}(\bar{c}_n/\bar{a}_n)$. Since $\mathcal{N} \in (\mathcal{M},
w)\oplus(\bar{c}_n/\bar{a}_n)$ was chosen arbitrarily, and we have that  $(\mathcal{M},
w)\boxplus_{\mathcal{L}}(\bar{c}_n/\bar{a}_n) \neq \emptyset$, it follows that also $(\mathcal{M},
w)\boxplus_{\mathcal{L}}(\bar{c}_n/\bar{a}_n) \supseteq (\mathcal{M},
w)\oplus(\bar{c}_n/\bar{a}_n)$.

(Part 2). We clearly have:
\begin{align*}
(\forall\mathcal{N}\in (\mathcal{M},
w)\boxplus_\mathcal{L}(\bar{c}_n/\bar{a}_n))(\mathcal{N}, w
\models_\mathcal{L}\phi) &\Leftrightarrow (\forall\mathcal{N}\in (\mathcal{M},
w)\oplus(\bar{c}_n/\bar{a}_n))(\mathcal{N}, w
\models_\mathcal{L}\phi)\\
&\Rightarrow ([\mathcal{M},w],
\bar{c}_n/\bar{a}_n),
w\models_\mathcal{L}\phi\\
&\Leftrightarrow \phi \in
Tp^+_\mathcal{L}(\mathcal{M},w,\bar{a}_n)\\
 &\Rightarrow(\exists\mathcal{N}\in
(\mathcal{M}, w)\oplus(\bar{c}_n/\bar{a}_n))(\mathcal{N}, w
\models_\mathcal{L}\phi)\\
 &\Leftrightarrow(\exists\mathcal{N}\in
(\mathcal{M}, w)\boxplus_\mathcal{L}(\bar{c}_n/\bar{a}_n))(\mathcal{N}, w
\models_\mathcal{L}\phi)
\end{align*}
by Part 1 and since, of course, $([\mathcal{M},w], \bar{c}_n/\bar{a}_n) \in
(\mathcal{M}, w)\oplus(\bar{c}_n/\bar{a}_n)$ due to the fact that
the class of $\mathcal{L}$-admissible structures coincides with
the same class for $\mathcal{L}' \in StIL$. On the other hand, for
an arbitrary $\mathcal{N}\in (\mathcal{M}, w)\boxplus_\mathcal{L}(\bar{c}_n/\bar{a}_n)) = (\mathcal{M},
w)\oplus(\bar{c}_n/\bar{a}_n)$ we have:
\begin{align*}
\mathcal{N}, w \models_\mathcal{L}\phi &\Leftrightarrow
[\mathcal{N}, w], w \models_\mathcal{L}\phi &&\text{(by Lemma
\ref{L:generated-submodels}.1)}\\
&\Leftrightarrow ([\mathcal{M},w], \bar{c}_n/\bar{a}_n),
w\models_\mathcal{L}\phi &&\text{(by Lemma \ref{L:cutoff})}
\end{align*}

Parts 3 and 4 follow from Part 2 and the quantifier closure
conditions in the same way as they follow for $\mathsf{IL}$.
\end{proof}
We end this subsection with the lemma relating the positive
$\mathcal{L}$-types of tuples connected by an asimulation.
Although this lemma, again, only depends on the preservation under
asimulation and does not assume that the logic in question is an
extension of a standard intuitionistic logic, it only acquires its
standard meaning together with the latter assumption, since
otherwise the notion of an $\mathcal{L}$-type might end up having
a non-regular meaning.
\begin{lemma}\label{L:asimulation-l-types}
Let
  $\mathcal{L}$ be an abstract intuitionistic logic and let
$\mathcal{L}' \in StIL$ be such that  $\mathcal{L}$ is preserved
under $\mathcal{L}'$-asimulations. Let $\Theta$ be a signature,
and let $(\mathcal{M},w), (\mathcal{N},v)\in
Str_{\mathcal{L}}(\Theta)$. Then, if $n < \omega$ and $\bar{a}_n
\in A^n_w$, $\bar{b}_n \in B^n_v$ are such that $A$ is an
$\mathcal{L}'$-asimulation from $(\mathcal{M},w,\bar{a}_n)$ to
$(\mathcal{N},v,\bar{b}_n)$, then:
\[ Tp^+_{\mathcal{L}}(\mathcal{M}, w,
\bar{a}_n)\subseteq Tp^+_{\mathcal{L}}(\mathcal{N}, v, \bar{b}_n).
\]
\end{lemma}
\begin{proof}
By Lemma \ref{L:asimulations-generated}.3, for any given tuple
$\bar{c}_n$ of pairwise distinct constants outside $\Theta$,
$(A_{w,v})\downarrow$ is an $\mathcal{L}'$-asimulation from
$(([\mathcal{M},w],\bar{c}_n/\bar{a}_n), w)$ to
$(([\mathcal{N},v],\bar{c}_n/\bar{b}_n), v)$ so that, for an
arbitrary $\phi \in L(\Theta\cup \{\bar{c}_n\})$ we have:
\begin{align*}
    \phi \in Tp^+_{\mathcal{L}}(\mathcal{M}, w,
\bar{a}_n) &\Leftrightarrow ([\mathcal{M},w],\bar{c}_n/\bar{a}_n),
w \models_\mathcal{L}\phi\\
&\Rightarrow ([\mathcal{N},v],\bar{c}_n/\bar{b}_n),
v\models_\mathcal{L}\phi &&\text{(by $\mathcal{L}'$-asimulation preservation)}\\
&\Leftrightarrow \phi \in Tp^+_{\mathcal{L}}(\mathcal{N}, v,
\bar{b}_n).
\end{align*}
\end{proof}
More generally,  we claim that if $\mathcal{L}$ is an abstract intuitionistic logic, 
$\mathcal{L}' \in StIL$ is such that $\mathcal{L} \sqsupseteq
\mathcal{L}'$, and $\mathcal{L}$ is preserved under
$\mathcal{L}'$-asimulations then every statement formulated in the foregoing sections for $\mathsf{IL}$, except when the language is explicitly assumed to be $IL$, can also be proven for $\mathcal{L}$ with only minimal changes to the proof.

\subsection{Congruence relations}

In what follows, we will also use the notion of \emph{congruence}.
If $\Theta$ is a signature, and $\mathcal{M} \in Mod(\Theta)$,
then we say that the function $\approx: W \to
2^{\mathbb{A}\times\mathbb{A}}$ is a congruence on $\mathcal{M}$
iff by setting, for all $w \in W$, $I'_w(\approx) := \approx(w)$, we
obtain a $\Theta \cup \{ \approx\}$-expansion $\mathcal{M}'$ of
$\mathcal{M}$ such that every sentence given in Lemma
\ref{L:approx-properties} is $\models_{\mathsf{IL}}$-valid in
$\mathcal{M}'$. If we add $\equiv$ to the list of $\Theta$-predicates and demand that $\approx$ also acts as a congruence w.r.t. $\equiv$, then we simply get the diagonal relation which we may also call $\equiv$-congruence treating it as a borderline case of the more general notion of a congruence relation. 

Congruences can be used to define `factor-models' out of existing
models; more precisely, the following lemma is true:
\begin{lemma}\label{L:congr-mod}
Let $\Theta$ be a signature, let $\mathcal{L} \in StIL$, and let $(\mathcal{M},v) \in Str_\mathcal{L}(\Theta)$, and
assume that $\approx$ is a congruence on $\mathcal{M}$. Assume, finally, that, in case  $\mathcal{L}$ includes $\equiv$,  $\approx$ is even an $\equiv$-congruence. Then:
\begin{enumerate}
    \item $(\mathcal{M}_\approx, v) \in Str_\mathcal{L}(\Theta)$ where:
\begin{itemize}
    \item $W_\approx := W$, $\prec_\approx := \prec$;
    \item $(A_\approx)_w := \{ [a]_{\approx(w)} \mid a \in A_w\}$, where
    $[a]_{\approx(w)}:= \{ b \in A_w \mid a\mathrel{\approx(w)} b \}$;
    \item $(I_\approx)_w(P) := \{ [\langle\bar{a}_m\rangle]_{\approx(w)} \mid \bar{a}_m \in
    I_w(P)\}$ for all $P \in \Theta_m$;
    \item $(I_\approx)_w(c) = [I_w(c)]_{\approx(w)}$ for all $c \in
    Const_\Theta$;
    \item $(\mathbb{H}_\approx)_{wv}([a]_{\approx(w)}) =
    [\mathbb{H}_{wv}(a)]_{\approx(v)}$ for all $w,v \in W$ and all $a \in
    A_w$.
\end{itemize}
    \item The relation:
$$
B:= \{\langle (w; \bar{a}_m), (w; [\bar{a}_m]_{\approx(w)})\rangle,
\langle (w; [\bar{a}_m]_{\approx(w)}),(w; \bar{a}_m)\rangle \mid w
\in W,\,\bar{a}_m \in A^m_w\}
$$
is an $\mathcal{L}$-asimulation both from $(\mathcal{M}, v,
\bar{a}_n )$ to $(\mathcal{M}_\approx, v, [\bar{a}_n]_{\approx(v)})$ and
from $(\mathcal{M}_\approx,v, [\bar{a}_n]_{\approx(v)})$ to
$(\mathcal{M},v, \bar{a}_n)$ for every $v \in W$, every $n <
\omega$, and every $\bar{a}_n \in A^n_v$.
\end{enumerate}
\end{lemma}
\begin{proof}
First of all, if $\mathcal{L} \in StIL\setminus\{\mathsf{IL}, \mathsf{CD}\}$, 
 then $\approx$ is simply a diagonal relation, every $[a]_{\approx(w)}$ is just $\{a\}$, and $(id_W, \{(a, [a]_{\approx(w)})\mid a \in A_w\})$ is even an isomorphism between $\mathcal{M}$ and $\mathcal{M}_\approx$, so the only non-trivial cases arise when $\mathcal{L} \in \{\mathsf{IL}, \mathsf{CD}\}$. We consider these cases below.

(Part 1). Since $W_\approx = W$ and $\prec_\approx = \prec$, it is clear
that $\prec_\approx$ is a partial order on the non-empty set
$W_\approx$; moreover, it is also clear that the classical model
$(\mathfrak{A}_\approx)_w = ((A_\approx)_w, (I_\approx)_w)$ is well-defined
for any $w \in W_\approx$. Since $\mathcal{M}$ is a model, for all
$w, v \in W_\approx = W$ such that $w \neq v$ we must have $A_v \cap
A_w = \emptyset$, hence $2^{A_v} \cap 2^{A_w} = \emptyset$, but we
also have $(A_\approx)_w \subseteq 2^{A_w}$ and $(A_\approx)_v \subseteq
2^{A_v}$. Therefore, we get that $(A_\approx)_w \cap (A_\approx)_v =
\emptyset$. It remains to check the conditions involving
$\mathbb{H}_\approx$.

As for monotonicity, if $w,v \in W_\approx = W$ are such that $w
\mathrel{\prec_\approx} v$ (so also $w \mathrel{\prec} v$) then we have for every $m <
\omega$ and all $\bar{a}_m \in A^m_w$, $P \in \Theta_m$, and $c
\in Const_\Theta$:
\begin{align*}
    (I_\approx)_w(P)([\langle\bar{a}_m\rangle]_{\approx(w)}) &\Leftrightarrow
    I_w(P)(\bar{a}_m)\\
    &\Rightarrow I_v(P)(\mathbb{H}_{wv}\langle\bar{a}_m\rangle)\\
&\Leftrightarrow
(I_\approx)_v(P)([\langle\mathbb{H}_{wv}\langle\bar{a}_m\rangle\rangle]_{\approx(v)})\\
&\Leftrightarrow
(I_\approx)_v(P)((\mathbb{H}_\approx)_{wv}\langle[\langle\bar{a}_m\rangle]_{\approx(w)}\rangle)
\end{align*}
and, moreover:
\begin{align*}
    (\mathbb{H}_\approx)_{wv}((I_\approx)_w(c)) &=
    (\mathbb{H}_\approx)_{wv}([I_w(c)]_{\approx(w)})\\
    &= [\mathbb{H}_{wv}(I_w(c))]_{\approx(v)}\\
    &= [I_v(c)]_{\approx(v)}\\
    &= (I_\approx)_v(c)
\end{align*}
Next, it is easy to see that for every $w \in  W_\approx = W$ and
every $a \in A_w$ we have $(\mathbb{H}_\approx)_{ww}([a]_{\approx(w)}) =
[\mathbb{H}_{ww}(a)]_{\approx(w)} = [a]_{\approx(w)}$.

Finally, if $w,v,u \in W_\approx = W$ are such that $w \mathrel{\prec_\approx} v
\mathrel{\prec_\approx} u$ then also $w \mathrel{\prec} v \mathrel{\prec} u$, and for every $a \in
A_w$ we have:
\begin{align*}
    (\mathbb{H}_\approx)_{vu}((\mathbb{H}_\approx)_{wv}([a]_{\approx(w)})) &=
    (\mathbb{H}_\approx)_{vu}([\mathbb{H}_{wv}(a)]_{\approx(v)})\\
    &=
    [\mathbb{H}_{vu}(\mathbb{H}_{wv}(a))]_{\approx(u)}\\
    &= [\mathbb{H}_{wu}(a)]_{\approx(u)}\\
    &= (\mathbb{H}_\approx)_{wu}([a]_{\approx(w)})
\end{align*}
For the case when $\mathcal{M} \in Mod_\Theta(Su)$ we need to
observe, in addition, that if $\mathbb{H}_{wv}$ is a surjection for a given $v\in [W,w]$, then $(\mathbb{H}_\approx)_{wv}$ is clearly surjective. Indeed, we can choose then for an arbitrary $a' \in (A_\approx)_v$ an $a \in A_v$ such that $a' = [a]_{\approx(v)}$. But then, by surjectivity of $\mathbb{H}_{wv}$, there exists a $b \in A_w$ such that
$\mathbb{H}_{wv}(b) = a$. It follows then, that $a' = [a]_{\approx(v)} = [\mathbb{H}_{wv}(b)]_{\approx(v)} = (\mathbb{H}_\approx)_{wv}([b]_{\approx(w)})$, and we clearly have $[b]_{\approx(w)} \in (A_\approx)_w$.

(Part 2). By definition of $\mathcal{L}$-asimulation.
\end{proof}
\begin{corollary}\label{C:congr}
Let $\mathcal{L}\sqsupseteq \mathcal{L}' \in StIL$ be an abstract
intuitionistic logic, let $\Theta$ be a signature such that
$(\mathcal{M},w) \in Mod_\Theta$. Assume that $\mathcal{L}$ is
preserved under $\mathcal{L}'$-asimulations and that $\approx$ is a
congruence on $\mathcal{M}$ (and even an $\equiv$-congruence in case $\equiv$ is in the language of $\mathcal{L}'$). Then, for every $m < \omega$,
whenever $\bar{a}_m \in A^n_w$, we have:
$$
Tp_\mathcal{L}(\mathcal{M}, w, \bar{a}_m) =
 Tp_\mathcal{L}(\mathcal{M}_\approx, w, [\bar{a}_m]_{\approx(w)}).
$$
\end{corollary}
\begin{proof}
By Lemma \ref{L:asimulation-l-types} and Lemma
\ref{L:congr-mod}.2.
\end{proof}

\section{The main result}\label{S:main}

\subsection{The main result formulated}
We start by noting that all the standard intuitionistic logics
display a certain strong combination of desirable model-theoretic
properties:
\begin{lemma}\label{L:properties}
If $\mathcal{L} \in StIL$, then $\mathcal{L}$ is preserved under
$\mathcal{L}$-asimulations, $\star$-compact and has TUP.
\end{lemma}
\begin{proof}
We sketch the proof for the case $\mathcal{L} = \mathsf{IL}$, for
the other cases the proof is similar. Invariance under
asimulations follows from the main result of \cite{o1}. The
$\star$-compactness of $\mathsf{IL}$ is a well-known fact and is
established in many places, see e.g. \cite[Theorem
2.46]{chagrov} (this property is also easily established by compactness of classical fiirst-order logic and the standard
translation). Finally, we outline the proof of TUP following the
usual idea for such type of proofs.

It is straightforward to show that we have $\mathcal{M}_i
\subseteq \bigcup_{n \in \omega}\mathcal{M}_n = \mathcal{M}$ for
every $i \in \omega$, so it remains to verify that, for all $i, m <
\omega$, $w \in W_i$, $\bar{a}_m \in (A_i)^m_w$, and every tuple
$\bar{c}_m$ of pairwise distinct constants outside $\Theta$, we
have $Tp_{\mathsf{IL}}(\mathcal{M}_i, w, \bar{c}_m/\bar{a}_m) =
Tp_{\mathsf{IL}}(\mathcal{M}, w, \bar{c}_m/\bar{a}_m)$. The latter
is equivalent to showing that for every $i, m < \omega$, $w \in
W_i$, $\bar{a}_m \in (A_i)^m_w$, every tuple $\bar{c}_m$ of
pairwise distinct constants outside $\Theta$, and every $\varphi
\in IL_\emptyset(\Theta\cup\{\bar{c}_m\})$ we have:
$$
([\mathcal{M}_i,w],\bar{c}_m/\bar{a}_m), w \models_{\mathsf{IL}}
\varphi \Leftrightarrow ([\mathcal{M},w],\bar{c}_m/\bar{a}_m), w
\models_{\mathsf{IL}} \varphi.
$$
We show this by induction on the construction of $\varphi \in
IL_\emptyset(\Theta\cup\{\bar{c}_m\})$ in which the basis and
induction step cases for $\wedge$ and $\vee$ are trivial. We look
into the remaining cases:

\emph{Case 1}. $\varphi = \psi \to \chi$, where $\psi, \chi \in
IL_\emptyset(\Theta\cup\{\bar{c}_m\})$. ($\Leftarrow$). If
$([\mathcal{M}_i,w],\bar{c}_m/\bar{a}_m), w
\not\models_{\mathsf{IL}} \psi \to \chi$, then there is a $v \in
W_i$ such that $w\mathrel{\prec_i}v$ and
$([\mathcal{M}_i,w],\bar{c}_m/\bar{a}_m), v \models_{\mathsf{IL}}
(\{\psi\},\{\chi\})$, so also
$[([\mathcal{M}_i,w],\bar{c}_m/\bar{a}_m), v], v
\models_{\mathsf{IL}} (\{\psi\},\{\chi\})$ by Lemma
\ref{L:satisfaction}.2, whence
$$([\mathcal{M}_i,v],\bar{c}_m/(\mathbb{H}_i)_{wv}\langle\bar{a}_m\rangle),
v \models_{\mathsf{IL}} (\{\psi\},\{\chi\})$$ by Lemma
\ref{L:cutoff-constants}.2, so that we get
$([\mathcal{M},v],\bar{c}_m/(\mathbb{H}_i)_{wv}\langle\bar{a}_m\rangle),
v \models_{\mathsf{IL}} (\{\psi\},\{\chi\})$ by induction
hypothesis. Since $\mathcal{M}_i \subseteq \mathcal{M}$ and
$\bar{a}_m \in (A_i)^m_w$, we must have
$(\mathbb{H}_i)_{wv}\langle\bar{a}_m\rangle =
\mathbb{H}_{wv}\langle\bar{a}_m\rangle$ and $w \prec v$, whence
further
$([\mathcal{M},v],\bar{c}_m/\mathbb{H}_{wv}\langle\bar{a}_m\rangle),
v \models_{\mathsf{IL}} (\{\psi\},\{\chi\})$. Applying again Lemma
\ref{L:cutoff-constants}.2, we get that
$[([\mathcal{M},w],\bar{c}_m/\bar{a}_m), v], v
\models_{\mathsf{IL}} (\{\psi\},\{\chi\})$. Next, Lemma
\ref{L:satisfaction}.2 yields
$([\mathcal{M},w],\bar{c}_m/\bar{a}_m), v \models_{\mathsf{IL}}
(\{\psi\},\{\chi\})$, whence finally
$([\mathcal{M},w],\bar{c}_m/\bar{a}_m), w
\not\models_{\mathsf{IL}} \psi \to \chi$.

($\Rightarrow$). In the other direction, assume that
$([\mathcal{M},w],\bar{c}_m/\bar{a}_m), w
\not\models_{\mathsf{IL}} \psi \to \chi$. Then, for some $v \in W$
such that $w\mathrel{\prec}v$, it is true that
$([\mathcal{M},w],\bar{c}_m/\bar{a}_m), v \models_{\mathsf{IL}}
(\{\psi\},\{\chi\})$. Applying again Lemma \ref{L:satisfaction}.2
and Lemma \ref{L:cutoff-constants}.2, we infer that
$([\mathcal{M},v],\bar{c}_m/\mathbb{H}_{wv}\langle\bar{a}_m\rangle),
v \models_{\mathsf{IL}} (\{\psi\},\{\chi\})$. But then, for some
$j \geq i$, we must have $w,v \in W_j$, $w\mathrel{\prec_j}v$, and
$\bar{a}_m \in (A_j)^m_w$. Hence, by $\mathcal{M}_j \subseteq
\mathcal{M}$ and Definition \ref{D:submodel}, we get that
$(\mathbb{H}_j)_{wv}\langle\bar{a}_m\rangle =
\mathbb{H}_{wv}\langle\bar{a}_m\rangle$, whence it follows that
$([\mathcal{M},v],\bar{c}_m/(\mathbb{H}_j)_{wv}\langle\bar{a}_m\rangle),
v \models_{\mathsf{IL}} (\{\psi\},\{\chi\})$. By induction
hypothesis for $j$, we further get that
$([\mathcal{M}_j,v],\bar{c}_m/(\mathbb{H}_j)_{wv}\langle\bar{a}_m\rangle),
v \models_{\mathsf{IL}} (\{\psi\},\{\chi\})$. Applying one more
time Lemma \ref{L:satisfaction}.2 and Lemma
\ref{L:cutoff-constants}.2, we get that
$([\mathcal{M}_j,w],\bar{c}_m/\bar{a}_m), v \models_{\mathsf{IL}}
(\{\psi\},\{\chi\})$, so also
$([\mathcal{M}_j,w],\bar{c}_m/\bar{a}_m), w
\not\models_{\mathsf{IL}} \psi \to \chi$. By obvious transitivity
of $\preccurlyeq_{\mathsf{IL}}$ we get then that $\mathcal{M}_i
\preccurlyeq_{\mathsf{IL}} \mathcal{M}_j$, whence finally
$([\mathcal{M}_i,w],\bar{c}_m/\bar{a}_m), w
\not\models_{\mathsf{IL}} \psi \to \chi$.

\emph{Case 2}. $\varphi = \exists x\psi$, where $\psi \in
IL_x(\Theta\cup\{\bar{c}_m\})$. ($\Rightarrow$). If
$([\mathcal{M}_i,w],\bar{c}_m/\bar{a}_m), w \models_{\mathsf{IL}}
\exists x\psi$, then, by Corollary \ref{C:satisfaction-cutoff}, let $a
\in (A_i)_w \subseteq A_w$ and a new constant $c$ be such that
$([\mathcal{M}_i,w],(\bar{c}_m)^\frown c/(\bar{a}_m)^\frown a),w
\models_{\mathsf{IL}} \psi\binom{c}{x}$, whence, by induction hypothesis, we
also have $$([\mathcal{M},w],(\bar{c}_m)^\frown c/(\bar{a}_m)^\frown
a),w \models_{\mathsf{IL}} \psi\binom{c}{x}.$$ But then, again by
Corollary \ref{C:satisfaction-cutoff},
$([\mathcal{M},w],\bar{c}_m/\bar{a}_m), w \models_{\mathsf{IL}}
\exists x\psi$ clearly holds as well.

($\Leftarrow$). If $([\mathcal{M},w],\bar{c}_m/\bar{a}_m), w
\models_{\mathsf{IL}} \exists x\psi$, then, by Corollary
\ref{C:satisfaction-cutoff}, let $a \in A_w$ and a new
constant $c$ be such that
$([\mathcal{M},w],(\bar{c}_m)^\frown c/(\bar{a}_m)^\frown a),w
\models_{\mathsf{IL}} \psi\binom{c}{x}$. But then, for some $j \geq i$ we must
have $a \in (A_j)_w$ and also, by induction hypothesis and since
$\psi\binom{c}{x} \in IL_\emptyset(\Theta \cup \{(\bar{c}_m)^\frown
c\})$, that $([\mathcal{M}_j,w],(\bar{c}_m)^\frown
c/(\bar{a}_m)^\frown a),w \models_{\mathsf{IL}} \psi\binom{c}{x}$.
The latter means, by Corollary \ref{C:satisfaction-cutoff}, that
$$([\mathcal{M}_j,w],\bar{c}_m/\bar{a}_m), w \models_{\mathsf{IL}}
\exists x\psi.$$ By obvious transitivity of
$\preccurlyeq_{\mathsf{IL}}$ we get then that $\mathcal{M}_i
\preccurlyeq_{\mathsf{IL}} \mathcal{M}_j$, whence finally
$$([\mathcal{M}_i,w],\bar{c}_m/\bar{a}_m), w \models_{\mathsf{IL}}
\exists x\psi.$$

\emph{Case 3}. $\varphi = \forall x\psi$, where $\psi \in
IL_x(\Theta\cup\{\bar{c}_m\})$. ($\Leftarrow$). If $([\mathcal{M}_i,w],\bar{c}_m/\bar{a}_m), w
\not\models_{\mathsf{IL}} \forall x\psi$, then, by Corollary \ref{C:satisfaction-cutoff}, let $v \in W_i$,
$a \in (A_i)_v \subseteq A_v$, and a new constant $c$ be such that $w \prec_i v$ and $([\mathcal{M}_i,v],(\bar{c}_m)^\frown c/(\mathbb{H}_i)_{wv}\langle\bar{a}_m\rangle^\frown a), v \not\models_{\mathsf{IL}} \psi\binom{c}{x}$. Therefore, by induction hypothesis, we must also have $([\mathcal{M},v],(\bar{c}_m)^\frown c/(\mathbb{H}_i)_{wv}\langle\bar{a}_m\rangle^\frown a), v \not\models_{\mathsf{IL}} \psi\binom{c}{x}$; it follows, by $\mathcal{M}_i \subseteq \mathcal{M}$ and
$\bar{a}_m \in (A_i)^m_w$, that also $([\mathcal{M},v],(\bar{c}_m)^\frown c/\mathbb{H}_{wv}\langle\bar{a}_m\rangle^\frown a), v \not\models_{\mathsf{IL}} \psi\binom{c}{x}$. Applying again  Corollary \ref{C:satisfaction-cutoff}, we get that $([\mathcal{M},w],\bar{c}_m/\bar{a}_m), w
\not\models_{\mathsf{IL}} \forall x\psi$.

($\Rightarrow$). $([\mathcal{M},w],\bar{c}_m/\bar{a}_m), w
\not\models_{\mathsf{IL}} \forall x\psi$, then, by Corollary \ref{C:satisfaction-cutoff}, let $v \in W$,
$a \in A_v$, and a new constant $c$ be such that $w \prec v$ and  $([\mathcal{M},v],(\bar{c}_m)^\frown c/\mathbb{H}_{wv}\langle\bar{a}_m\rangle^\frown a), v \not\models_{\mathsf{IL}} \psi\binom{c}{x}$. But then, for some
$j \geq i$, we must have $w,v \in W_j$, $w\mathrel{\prec_j}v$,
$\bar{a}_m \in (A_j)^m_w$, and $a \in (A_j)_v$. Furthermore, by $\mathcal{M}_j \subseteq
\mathcal{M}$ and Definition \ref{D:submodel}, we get that
$(\mathbb{H}_j)_{wv}\langle\bar{a}_m\rangle =
\mathbb{H}_{wv}\langle\bar{a}_m\rangle$, so that, by induction hypothesis and
$\psi\binom{c}{x} \in IL_\emptyset(\Theta \cup \{(\bar{c}_m)^\frown
c\})$, we get that $$([\mathcal{M}_j,v],(\bar{c}_m)^\frown c/(\mathbb{H}_j)_{wv}\langle\bar{a}_m\rangle^\frown a), v \not\models_{\mathsf{IL}} \psi\binom{c}{x}.$$ Applying again Corollary \ref{C:satisfaction-cutoff}, we get that $([\mathcal{M}_j,w],\bar{c}_m/\bar{a}_m), w
\not\models_{\mathsf{IL}} \forall x\psi$. By obvious transitivity of
$\preccurlyeq_{\mathsf{IL}}$ we get then that $\mathcal{M}_i
\preccurlyeq_{\mathsf{IL}} \mathcal{M}_j$, whence finally
$([\mathcal{M}_i,w],\bar{c}_m/\bar{a}_m), w \not\models_{\mathsf{IL}}
\forall x\psi$.
\end{proof}

Our main theorem is then that no standard intuitionistic logic has
proper extensions that display the combination of useful
properties established in Lemma \ref{L:properties}. In other
words, we are going to establish the following:
\begin{theorem}\label{L:main}
Let $\mathcal{L}$ be an abstract intuitionistic logic and let
$\mathcal{L}' \in StIL$. If $\mathcal{L}' \sqsubseteq \mathcal{L}$
and $\mathcal{L}$ is preserved under $\mathcal{L}'$-asimulations,
$\star$-compact, and has TUP, then $\mathcal{L}' \bowtie
\mathcal{L}$.
\end{theorem}

\subsection{The main result proven}

Before we start with the proof, we need one more piece of
notation. If $\mathcal{L}$ is an abstract intuitionistic logic,
$\Theta$ a signature, and $\Gamma
\subseteq L(\Theta)$, then we let $Mod_\mathcal{L}(\Theta,
\Gamma)$ denote the class of pointed $\Theta$-models
$(\mathcal{N}, u)\in Str_{\mathcal{L}}(\Theta)$ such that
for every $\phi \in \Gamma$ it is
true that:
$$
\mathcal{N}, u \models_\mathcal{L} \phi.
$$
If $\Gamma = \{ \phi \}$ for some $\phi \in L(\Theta)$ then
instead of $Mod_\mathcal{L}(\Theta,\Gamma)$ we simply write
$Mod_\mathcal{L}(\Theta,\phi)$.

We now start by establishing a couple of technical facts first:

\begin{proposition}\label{L:proposition1}
Let $\mathcal{L}$ be a $\star$-compact  abstract intuitionistic
logic extending $\mathcal{L}' \in StIL$. Suppose that $\mathcal{L}'
\not\bowtie\mathcal{L}$. Then, there are $\phi \in L(\Theta_\phi)$ and pointed models $(\mathcal{M}_1, w_1)$,
$(\mathcal{M}_2, w_2) \in Str_{\mathcal{L}}(\Theta_\phi)$ such that
$Th^+_{\mathcal{L}'}(\mathcal{M}_1, w_1) \subseteq
Th^+_{\mathcal{L}'}(\mathcal{M}_2, w_2)$, while
$\mathcal{M}_1, w_1 \models_\mathcal{L} \phi$ and $\mathcal{M}_2,
w_2 \not\models_\mathcal{L} \phi$.
\end{proposition}

\begin{proof}
Suppose that for an arbitrary $\phi \in L(\Theta_\phi)$
we have shown that:

\begin{center}
\begin{itemize}
\item[(i)]  $Mod_\mathcal{L}(\Theta_\phi,\phi) =
\bigcup_{(\mathcal{N}, u) \in Mod_\mathcal{L}(\Theta_\phi,\phi)}
Mod_\mathcal{L}(\Theta_\phi,Th^+_{\mathcal{L}'}(\mathcal{N}, u))$,
\end{itemize}
\end{center}
Let  $(\mathcal{N}, u) \in Mod_\mathcal{L}(\Theta_\phi,\phi)$ be
arbitrary. The above implies that every $\Theta_\phi$-model of
$Th^+_{\mathcal{L}'}(\mathcal{N}, u)$ must be a model of
$\phi$. But then the theory $(Th^+_{\mathcal{L}'}(\mathcal{N}, u), \{ \phi \})$ is $\mathcal{L}$-unsatisfiable. By the
$\star$-compactness of $\mathcal{L}$, for some
$\Psi_{(\mathcal{N}, u)} \Subset
Th^+_{\mathcal{L}'}(\mathcal{N}, u)$ (and we can pick a
unique one using the Axiom of Choice), the pair
$(\Psi_{(\mathcal{N}, u)} ,\{ \phi\})$ is
$\mathcal{L}$-unsatisfiable. Hence, $\bigwedge \Psi_{(\mathcal{N},
u)} $ logically implies $\phi$ in $\mathcal{L}$. Then,
given (i), we get that
\begin{center}
\begin{itemize}
 \item [(ii)] $Mod_\mathcal{L}(\Theta_\phi,\phi) = \bigcup_{(\mathcal{N}, u) \in Mod_\mathcal{L}(\Theta_\phi,\phi)}Mod_\mathcal{L}(\Theta_\phi,\Psi_{(\mathcal{N}, u)})$.
\end{itemize}
\end{center}
However, this means that the theory $(\{ \phi \}, \{ \bigwedge
\Psi_{(\mathcal{N}, u)} \mid (\mathcal{N}, u) \in
Mod_\mathcal{L}(\Theta_\phi,\phi) \})$ is
$\mathcal{L}$-unsatisfiable and, by the
$\star$-compactness of $\mathcal{L}$, for some finite $$\Gamma
\subseteq  \{ \bigwedge \Psi_{(\mathcal{N}, u)} \mid
(\mathcal{N}, u) \in Mod_\mathcal{L}(\Theta_\phi,\phi) \}.$$ the
theory $(\{ \phi \}, \Gamma)$ is $\mathcal{L}$-unsatisfiable. This
means that  $ Mod_\mathcal{L}(\Theta_\phi,\phi) \subseteq
Mod_\mathcal{L}(\Theta_\phi,\bigvee\Gamma)$. So, using (ii), since
clearly $Mod_\mathcal{L}(\Theta_\phi,\bigvee\Gamma) \subseteq
\bigcup_{(\mathcal{N}, u) \in
Mod_\mathcal{L}(\Theta_\phi,\phi)}Mod_\mathcal{L}(\Theta_\phi,\Psi_{(\mathcal{N},u)})$,
 we get that:
\begin{center}
\begin{itemize}
 \item [(iii)] $Mod_\mathcal{L}(\Theta_\phi,\phi) = Mod_\mathcal{L}(\Theta_\phi,\bigvee\Gamma)$.
\end{itemize}
\end{center}
Now, $\bigvee \Gamma$ is obviously in $\mathcal{L}'(\Theta_\phi)$ involving only finitary conjunctions and disjunctions.
So we have shown that every $\phi \in L(\Theta_\phi)$ is
in $\mathcal{L}'(\Theta_\phi)$ and hence that
$\mathcal{L} \bowtie \mathcal{L}'$ which is in contradiction with
the hypothesis of the proposition.

Therefore, (i) must fail for at least one $\phi \in
L(\Theta_\phi)$, and clearly, for this $\phi$ it can only fail if
$$
\bigcup_{(\mathcal{N}, u) \in Mod_\mathcal{L}(\Theta_\phi,\phi)}
Mod_\mathcal{L}(\Theta_\phi,Th^+_{\mathcal{L}'}(\mathcal{N}, u)) \not\subseteq Mod_\mathcal{L}(\Theta_\phi,\phi).
$$
But the latter
means that for some pointed model
$(\mathcal{M}_1, w_1)\in Str_{\mathcal{L}}(\Theta_\phi)$ such that $\mathcal{M}_1, w_1
\models_\mathcal{L}\phi$ there is another model
$(\mathcal{M}_2, w_2)\in Str_{\mathcal{L}}(\Theta_\phi)$ such that $\mathcal{M}_2, w_2
\not\models_\mathcal{L}\phi$, and
$Th^+_{\mathcal{L}'}(\mathcal{M}_1, w_1)$ is satisfied
at $(\mathcal{M}_2, w_2)$. The latter means, in turn, that we have
$Th^+_{\mathcal{L}'}(\mathcal{M}_1, w_1) \subseteq
Th^+_{\mathcal{L}'}(\mathcal{M}_2, w_2)$ as desired.
\end{proof}

Assume $\mathcal{M}$ is a $\Theta$-model. We define that
$\Theta_\mathcal{M}$ is $\Theta \cup \{ P^+_a, P^-_a \mid a \in \mathbb{A}
\} $ where all of $P^+_a,
P^-_a$ are new and pairwise distinct unary predicate
letters, and we define that $\mathcal{M}^\ast = (W, \prec,
\mathfrak{A}^\ast, \mathbb{H})$ is the $\Theta_\mathcal{M}$-model,
such that $W$, $\prec$, and $\mathbb{H}$ are just borrowed from
$\mathcal{M}$, and, for every $w \in W$, $\mathfrak{A}^\ast$ is just
$(A_w, I_w^\ast)$. In other words, the only changes are in the
interpretation of predicates. As for $I_w^\ast$, we set $I_w^\ast(R) =
I_w(R)$ for all $R \in Pred_\Theta$, and $I_w^\ast(c) =
I_w(c)$ for all $c \in Const_\Theta$, and for all $w,v \in W$, every $b \in A_v$ and every $a \in A_w$ we set the following extensions:
\begin{align*}
   b \in I_v^\ast(P^+_a) &\Leftrightarrow (w \prec v\,\&\,b = \mathbb{H}_{wv}(a))\\
    b \in I_v^\ast(P^-_a) &\Leftrightarrow (v \not\prec w\,\text{or}\,a \neq \mathbb{H}_{vw}(b)).
\end{align*}
It is straightforward to check that $\mathcal{M}^\ast$ is a
$\Theta_\mathcal{M}$-model. For instance, if $v \prec u$ and $\mathbb{H}_{vu}(b) \notin I_u^\ast(P^-_a)$, then we must have both $u \prec w$ and $a = \mathbb{H}_{uw}(\mathbb{H}_{vu}(b)) = \mathbb{H}_{vw}(b)$, and, moreover, $v \prec w$ by transitivity. But then also $b \notin I_v^\ast(P^-_a)$. The case of $P^+_a$ is even easier.

We further observe that we can define in the resulting language  two further contexts, $Q^+_w$ and $Q^-_w$ by fixing some $a_w \in A_w$ for any given $w \in W$ and setting:
$$
Q^+_w := \exists xP^+_{a_w}(x);\,Q^-_w := \forall  x P^-_{a_w}(x).
$$ 
It follows then from the definitions given above that for any $w,v \in W$ we will have:
\begin{align*}
	\mathcal{M}^\ast, v \models_\mathcal{L} Q^+_w &\Leftrightarrow w \prec v\\
	\mathcal{M}^\ast, v \models_\mathcal{L} Q^-_w &\Leftrightarrow v \not\prec w.
\end{align*} 
It is also easy to check that, for every $\mathcal{L} \in StIL$, whenever $(\mathcal{M}, w) \in Str_\mathcal{L}(\Theta)$, then also $(\mathcal{M}^\ast, w) \in Str_\mathcal{L}(\Theta_\mathcal{M})$. 
\begin{lemma}\label{L:lemma1-cong}
Let $\mathcal{L}\sqsupseteq \mathcal{L}' \in StIL$ be an abstract
intuitionistic logic preserved under $\mathcal{L}'$-asimulations, let $(\mathcal{M},
w)$ be an unravelled $\Theta$-model, and let $(\mathcal{N}, v)$ be
a pointed $\Theta_{\mathcal{M}}$-model. Assume that
$Th_\mathcal{L}(\mathcal{N}, v) =
Th_\mathcal{L}(\mathcal{M}^\ast, w)$. Then the function
$\approx$ defined by setting, for all $u \in [U,v]$ and all $c,d \in
B_u$ that:

$$
c\mathrel{\approx(u)}d \Leftrightarrow \left\{%
\begin{array}{ll}
    (\forall a \in \mathbb{A})(c \in (J_u)(P^+_a)
\Leftrightarrow d \in (J_u)(P^+_a)), & \hbox{if $c \in (J_u)(P^+_b)$ for some $b \in \mathbb{A}$;} \\
    c = d, & \hbox{otherwise.} \\
\end{array}%
\right.   
$$
is a congruence on $[\mathcal{N}, v]$; moreover, in case  $\mathcal{L}'$ includes $\equiv$,  $\approx$ is even an $\equiv$-congruence on $[\mathcal{N}, v]$. 
\end{lemma}

\begin{proof} Assume that $Th_\mathcal{L}(\mathcal{N}, v) =
Th_\mathcal{L}(\mathcal{M}^\ast, w)$. By Lemma
\ref{L:generated-submodels}.1, we must then also have
$Th_\mathcal{L}([\mathcal{N},v], v) =
Th_\mathcal{L}(\mathcal{M}^\ast, w)$. Moreover, by
Lemma
\ref{L:generated-submodels}.2, we must further have
$Th^+_\mathcal{L}(\mathcal{M}^\ast, w) = Th^+_\mathcal{L}([\mathcal{N},v], v) \subseteq
Th^+_\mathcal{L}([\mathcal{N},v], u)$ for every $u \in [U,v]$.

We need to check that $\approx$ is a congruence on $[\mathcal{N}, v]$,
that is to say, if appending $\approx$ to $[\mathcal{N}, v]$ as the
interpretation for $\approx$ will indeed verify every sentence given in  Lemma
\ref{L:approx-properties}. 

We first consider the case when  $\mathcal{L}'$ includes $\equiv$. So let $u \in [U,v]$ and $c,d \in B_u$ be chosen arbitrarily. Then, if $c \notin (J_u)(P^+_b)$ for every $b \in \mathbb{A}$, we must have $c\mathrel{\approx(u)}d$ iff $c = d$ as desired. Otherwise, let $b \in \mathbb{A}$ be such that $c \in (J_u)(P^+_b)$. If $c = d$, then clearly $c\mathrel{\approx(u)}d$. Conversely, we must have, first, that $d \in (J_u)(P^+_b)$, and, second, that $\forall xy(P^+_b(x)\wedge P^+_b(y)\to x \equiv y) \in Th^+_\mathcal{L}(\mathcal{M}^\ast, w) \subseteq Th^+_\mathcal{L}([\mathcal{N},v], u)$, whence it clearly follows that $c = d$.

Turning now to the case when the language of $\mathcal{L}'$ lacks $\equiv$, we first check that
$\approx$ is a legitimate intuitionistic predicate extension, i.e.
that $\approx$ satisfies monotonicity. If $u \in  [U,v]$ and $c,d \in
B_u$ are such that $c\mathrel{\approx(u)}d$, then let $u' \in  [U,v]$ be
such that $u [\lhd,v] u'$. If $c \notin (J_u)(P^+_b)$ for every $b \in \mathbb{A}$, then we must have $c = d$, so also $\mathbb{G}_{uu'}(c)=\mathbb{G}_{uu'}(d)$, whence further $\mathbb{G}_{uu'}(c)\mathrel{\approx(u')}\mathbb{G}_{uu'}(d)$. On the other hand, if $c \in (J_u)(P^+_b)$ for some $b \in \mathbb{A}$, then also  $\mathbb{G}_{uu'}(c) \in (J_{u'})(P^+_b)$ for the same $b \in \mathbb{A}$ by monotonicity. We choose such a $b$. Note that we must also have both $d \in (J_u)(P^+_b)$ by $c\mathrel{\approx(u)}d$, and $\mathbb{G}_{uu'}(d) \in (J_{u'})(P^+_b)$ by monotonicity. Now, if it is not the case that
$\mathbb{G}_{uu'}(c)\mathrel{\approx(u')}\mathbb{G}_{uu'}(d)$, then
let $a \in \mathbb{A}$ be such that, wlog, $\mathbb{G}_{uu'}(c)
\in (J_{u'})(P^+_a)$, but $\mathbb{G}_{uu'}(d) \notin
(J_{u'})(P^+_a)$.
The following cases are then possible:

\emph{Case 1}. For some $\bar{w}_n, \bar{w}_l, \bar{w}_k \in W^{un(w)}$, it is true that $k, l \leq n$, $a \in A_{\bar{w}_k}$, $b \in A_{\bar{w}_l}$ and we have $\mathbb{H}_{\bar{w}_k\bar{w}_n}(a) = \mathbb{H}_{\bar{w}_l\bar{w}_n}(b)$. So the lines of homomorphic counterparts of $a$ and $b$ run together at some point, and since our accessibility order is discrete, we can choose the earliest point in which they start to coincide. In other words, we can choose an $r \in \omega$ such that $k,l \leq r \leq n$ and an $a_0 \in  A_{\bar{w}_r}$ such that $\mathbb{H}_{\bar{w}_k\bar{w}_r}(a) = \mathbb{H}_{\bar{w}_l\bar{w}_r}(b) = a_0$, and, whenever, $\mathbb{H}_{\bar{w}_k\bar{w}_s}(a) = \mathbb{H}_{\bar{w}_l\bar{w}_s}(b)$, then we have both $r \leq s$ and $$\mathbb{H}_{\bar{w}_r\bar{w}_s}(a_0) =\mathbb{H}_{\bar{w}_k\bar{w}_s}(a) = \mathbb{H}_{\bar{w}_l\bar{w}_s}(b).$$

But then note that we have, in fact:
\begin{align*}
\exists x(P^+_a(x) &\wedge P^+_b(x)) \to \exists zP^+_{a_0}(z),\\ &\exists z P^+_{a_0}(z) \to \forall y(P^+_b(y) \leftrightarrow P^+_a(y)) \in
Th^+_\mathcal{L}(\mathcal{M}^\ast, w) \subseteq
Th^+_\mathcal{L}([\mathcal{N}, v], u'). 
\end{align*} 
On the other hand, $c \in (J_u)(P^+_b)$ and monotonicity imply that $\mathbb{G}_{uu'}(c)
\in (J_{u'})(P^+_b)$, so that $\mathbb{G}_{uu'}(c)
\in (J_{u'})(P^+_a)\cap (J_{u'})(P^+_b)$. Therefore, we must have $$\forall y(P^+_b(y) \leftrightarrow P^+_a(y))
\in Th_\mathcal{L}([\mathcal{N}, v], u').$$ In addition, we have $\mathbb{G}_{uu'}(d)
\in (J_{u'})(P^+_b)$, hence also $\mathbb{G}_{uu'}(d)
\in (J_{u'})(P^+_a)$, contrary to our assumptions.

\emph{Case 2}. Case 1 does not obtain. But then we also
have
$$
\forall y(P^+_b(y) \to \neg P^+_a(y)) \in
Th^+_\mathcal{L}(\mathcal{M}^\ast, w) \subseteq
Th^+_\mathcal{L}([\mathcal{N}, v], u').
$$
Given that we have $\mathbb{G}_{uu'}(c) \in
(J_{u'})(P^+_b)$, we must also have $\mathbb{G}_{uu'}(c)
\notin (J_{u'})(P^+_a)$, which contradicts our assumptions.  

It is obvious, next, that $\approx(u)$ defines an equivalence on $B_u$ for
every $u \in [U,v]$, so once $\approx$ is shown to be monotonic, it is
straightforward to check, given its definition, that it satisfies
the intuitionistic theory of equivalence relations.

It remains to check that $\approx$-equivalent objects cannot be
distinguished by predicates available in
$\Theta_{\mathcal{M}}$. For the sake of simplicity, we only consider the case of unary predicates, the argument for the general case is the same. To show this, we first note that
we have
\begin{align*}
\forall x,y((P^+_a(x) \wedge P^+_a(y)) \to (R(x) \leftrightarrow R(y)))
&\in Th^+_\mathcal{L}(\mathcal{M}^\ast, w) =
Th^+_\mathcal{L}(\mathcal{N}, v) =\\ &=Th^+_\mathcal{L}([\mathcal{N}, v],
v) \subseteq Th^+_\mathcal{L}([\mathcal{N}, v], u).
\end{align*}
for all $a \in \mathbb{A}$, $R \in \Theta_{\mathcal{M}}$,
and $u \in [U,v]$. But then, let $c,d \in B_u$ be such that
$c\mathrel{\approx(u)}d$. We have either $c = d$ and then clearly also
$[\mathcal{N}, v], u \models_\mathcal{L} R(c) \leftrightarrow R(d)$
for every $R \in \Theta_{\mathcal{M}}$, or there is a $b
\in \mathbb{A}$ such that $c \in (J_u)(P^+_b)$. But then, by
$c\mathrel{\approx(u)}d$, we also get $d \in (J_u)(P^+_b)$, hence
$[\mathcal{N}, v], u \models_\mathcal{L} R(c) \leftrightarrow R(d)$
follows again for every $R \in \Theta_{\mathcal{M}}$.
\end{proof}
\begin{lemma}\label{L:lemma1}
Let $\mathcal{L}\sqsupseteq \mathcal{L}' \in StIL$ be an abstract
intuitionistic logic preserved under $\mathcal{L}'$-asimulations, let $(\mathcal{M},
w) \in Str_\mathcal{L}(\Theta)$ be an unravelled model, and let $(\mathcal{N}, v) \in Str_\mathcal{L}(\Theta_{\mathcal{M}})$. Assume that
$Th_\mathcal{L}(\mathcal{N}, v) =
Th_\mathcal{L}(\mathcal{M}^\ast, w)$. Let $\approx$ be the
congruence on $[\mathcal{N}, v]$
	defined as in Lemma
\ref{L:lemma1-cong}. Then we have that $(([\mathcal{N}, v])_\approx, v)  \in Str_\mathcal{L}(\Theta_{\mathcal{M}})$, and there exists an $\mathcal{L}$-elementary embedding
$(g,h)$ of $\mathcal{M}^\ast$ into $([\mathcal{N}, v])_\approx$
such that $h(w) = v$.
\end{lemma}
\begin{proof}
First, it is immediate to see that $(\mathcal{N}, v) \in Str_\mathcal{L}(\Theta_{\mathcal{M}})$ implies $([\mathcal{N}, v], v) \in Str_\mathcal{L}(\Theta_{\mathcal{M}})$. It follows then from Lemma \ref{L:congr-mod} and Lemma \ref{L:lemma1-cong}, that also  $(([\mathcal{N}, v])_\approx, v)  \in Str_\mathcal{L}(\Theta_{\mathcal{M}})$. 

As for the existence of an elementary embedding, we proceed by induction on $k \geq 1$ and define an increasing
chain of the form
$$
h_1 \subseteq\ldots \subseteq h_k \subseteq\ldots
$$
in such a way that, for each $k < \omega$, $h_k:\{\bar{u}_l \in
W^{un(w)}\mid l \leq k\} \to  ([U,v])_\approx$ is an injective
function. Once $h_k$ is defined, we also define, for any given
$\bar{w}_k \in W^{un(w)}$, the relation $g_{\bar{w}_k}$ by
setting:
$$
g_{\bar{w}_k}(a) := [b]_{\approx(h_k(\bar{w}_k))} \Leftrightarrow b \in
B_{h_k(\bar{w}_k)}\,\&\,\mathcal{N},
h_k(\bar{w}_k)\models_\mathcal{L} P^+_a(b).
$$
The idea is to define $h$ and $g$ mentioned in the Lemma by
collecting together all of the $h_k$'s and all of the
$g_{\bar{w}_k}$'s.

In the process of inductively defining $h_k$ for all $k < \omega$,
we show that our construction satisfies the following conditions
for all $1 \leq l,m \leq k < \omega$, and all $\bar{u}_l,
\bar{v}_m, \bar{w}_k \in W^{un(w)}$:
\begin{enumerate}
    \item[$(i)_k$] $g_{\bar{w}_k}$ is an injective function from $
    A_{\bar{w}_k}$ into $(([B, v])_\approx)_{h_k(\bar{w}_k)} = \{ [b]_{\approx(h_k(\bar{w}_k))}\mid b \in [B, v]_{h_k(\bar{w}_k)}\}$.

      \item[$(ii)_k$] If $a \in A_{\bar{w}_l}$,
    then $g_{\bar{w}_k}(\mathbb{H}_{\bar{w}_l\bar{w}_k}(a)) =
    (([\mathbb{G},v])_\approx)_{h_l(\bar{w}_l)h_k(\bar{w}_k)}(g_{\bar{w}_l}(a))$.

    \item[$(iii)_k$] $Th_\mathcal{L}(([\mathcal{N}, v])_\approx, h_k(\bar{w}_k)) =
Th_\mathcal{L}(\mathcal{M}^\ast, \bar{w}_k)$.

   \item[$(iv)_k$] $h_k(\bar{u}_l) ([\lhd, v])_\approx h_k(\bar{v}_m)\Leftrightarrow (\bar{u}_l = \bar{v}_l\,\&\,l\leq m)$.
\end{enumerate}

\emph{Induction basis}. $k = 1$. Then $\bar{w}_k = w$, so we have
to define $h$ and $g$ for this unique node. We set $h_1(w) := v$
thus getting an injective function, and thus, for every $a \in
A_{w}$ we get that $g_w(a) := [b]_{\approx(v)} \in (([B, v])_\approx)_v$ iff $\mathcal{N}, v\models_\mathcal{L} P^+_a(b)$. Under
these settings, it is true that
\begin{align*}
Th_\mathcal{L}(([\mathcal{N}, v])_\approx, h_1(w)) &= Th_\mathcal{L}(([\mathcal{N}, v])_\approx, v)\\
&= Th_\mathcal{L}([\mathcal{N}, v], v) &&\text{(by Lemma \ref{L:congr-mod}.2)}\\
&= Th_\mathcal{L}(\mathcal{N}, v) &&\text{(by Lemma \ref{L:generated-submodels}.1)}\\
&= Th_\mathcal{L}(\mathcal{M}^\ast, w)&&\text{(by the assumption of the lemma)}
\end{align*}
so
$(iii)_1$ is satisfied. Next, if $1 \leq l,m \leq 1$ and
$\bar{u}_l, \bar{w}_m \in W^{un(w)}$, then $l = m = 1$ and we get
both $\bar{u}_l = w = \bar{w}_m = \bar{w}_l$ and $h_1(\bar{u}_l) =
v \mathrel{ ([\lhd, v])_\approx} v= h_1(\bar{w}_m)$ by reflexivity.
Therefore, $(iv)_1$ is satisfied as well. Next, we check the
satisfaction of $(i)_1$. Note that, by $(iii)_1$, we also
have:
$$
\exists xP^+_a(x) \in Th^+_\mathcal{L}(\mathcal{M}^\ast, w) =
Th^+_\mathcal{L}(\mathcal{N}, v)
$$
for every $a \in A_{w}$ and thus  $\mathcal{N},
v\models_\mathcal{L} \exists xP^+_a(x)$ for every such $a$.
Therefore, $g_w$ is defined for every $a \in A_{w}$.
In addition, $(iii)_1$ implies that we have
$$
\forall x,y((P^+_a(x) \wedge P^+_a(y)) \to (P^+_c(x) \leftrightarrow
P^+_c(y)) \in Th^+_\mathcal{L}(\mathcal{M}^\ast, w) =
Th^+_\mathcal{L}(\mathcal{N}, v)
$$
for all $c \in \mathbb{A}$ and thus we have $\mathcal{N},
v\models_\mathcal{L} \forall x,y((P^+_a(x) \wedge P^+_a(y)) \to
(P^+_c(x) \leftrightarrow P^+_c(y))$. But then, if $d, e \in B_v =
([B,v])_v$ are such that both $\mathcal{N},
v\models_\mathcal{L} P^+_a(d)$, and $\mathcal{N},
v\models_\mathcal{L} P^+_a(e)$ then we get $\mathcal{N},
v\models_\mathcal{L} P^+_c(d)\leftrightarrow P^+_c(e)$ for all $c \in
\mathbb{A}$. Therefore, $d\mathrel{\approx(v)}e$, or, in other words,
$[d]_{\approx(v)} = [e]_{\approx(v)}$; thus $g_w$ is a function. Finally, if $a, b
\in A_{w}$ are such that $a \neq b$, then we must have
$g_w(a) = [d]_{\approx(v)}$ and $g_w(b) = [e]_{\approx(v)}$ where $d,e \in B_v  =
([B,v])_v$
are such that $\mathcal{N}, v\models_\mathcal{L} P^+_a(d)$ and
$\mathcal{N}, v\models_\mathcal{L} P^+_b(e)$. If now $[d]_{\approx(v)} =
[e]_{\approx(v)}$, then we must also have, for example, $\mathcal{N},
v\models_\mathcal{L} P^+_a(e)$, which cannot be the case since it
follows from $(iii)_1$ that we have
$$
\exists x(P^+_a(x) \wedge P^+_b(x)) \in
Th^-_\mathcal{L}(\mathcal{M}^\ast, w) =
Th^-_\mathcal{L}(\mathcal{N}, v)
$$
and thus $\mathcal{N}, v\not\models_\mathcal{L} \exists x(P^+_a(x)
\wedge P^+_b(x))$. Therefore, $d\mathrel{\approx(v)}e$ fails, whence we
have $g_w(a) = [d]_{\approx(v)} \neq [e]_{\approx(v)} = g_w(b)$, so that $g_w$
must be injective and $(i)_1$ must be satisfied. Finally, if $1
\leq l \leq 1$ and $a \in A_{w}$, then $l = 1$ so
$\bar{w}_l = w$ and we have:
\begin{align*}
g_{w}(\mathbb{H}_{ww}(a)) &= g_w(a) =  (([\mathbb{G},v])_\approx)_{h_1(w)h_1(w)}(g_{w}(a)),
\end{align*}
and $(ii)_1$ is trivially satisfied as well.

\emph{Induction step}. $k = l + 1$ for some $l \geq 1$. Then we
already have the chain of injections $h_1 \subseteq\ldots
\subseteq h_l$ defined and for all $m \leq l$, and all $\bar{u}_m
\in W^{un(w)}$ we have an injection $g_{\bar{u}_m}$ defined in
such a way that $(i)_m-(iv)_m$ are jointly satisfied by $h_m$ and
$g_{\bar{u}_m}$.

If $\bar{w}_k \in W^{un(w)}$ is chosen arbitrarily, then we must
have $\bar{w}_k = (\bar{w}_l)^\frown w_k$. We have then
\begin{align*}
Th_\mathcal{L}(([\mathcal{N}, v])_\approx, h_l(\bar{w}_l)) = Th_\mathcal{L}(\mathcal{M}^\ast, \bar{w}_l)
\end{align*}
by $(iii)_l$. Since $w$ is the root of
$\mathcal{M}$, we have, by $(iii)_l$, Lemma
\ref{L:generated-submodels}.2, and the preservation of $\mathcal{L}$ under
$\mathcal{L}'$-asimulations that
$$
Th^+_\mathcal{L}(\mathcal{M}^\ast, w) \subseteq
Th^+_\mathcal{L}(\mathcal{M}^\ast, \bar{w}_l) =
Th^+_\mathcal{L}(([\mathcal{N}, v])_\approx, h_l(\bar{w}_l)).
$$
By definition of $\mathcal{M}^\ast$ and by the closure of
$\mathcal{L}$ w.r.t. intuitionistic implication, we have that the
set:
$$
\Gamma_{\bar{w}_k} = \{ Q^+_{\bar{w}_k} \to \phi \mid
\phi \in Th^+_\mathcal{L}(\mathcal{M}^\ast, \bar{w}_k) \} \cup
\{ \psi \to Q^-_{\bar{w}_k} \mid \psi \in
Th^-_\mathcal{L}(\mathcal{M}^\ast, \bar{w}_k) \}
$$
is a subset of $Th^+_\mathcal{L}(\mathcal{M}^\ast, w)$, thus
also of $Th^+_\mathcal{L}(([\mathcal{N}, v])_\approx,
h_l(\bar{w}_l))$. We also know that, by
$\bar{w}_l\mathrel{\prec^{un(w)}}\bar{w}_k$, and by the fact that
the theory $(\{ Q^+_{\bar{w}_k} \},\{ Q^-_{\bar{w}_k}\})$ is $\mathcal{L}$-satisfied at
$(\mathcal{M}^\ast, \bar{w}_k)$, we have:
 $$
 \mathcal{M}^\ast, \bar{w}_l \not\models_\mathcal{L}
Q^+_{\bar{w}_k}\to Q^-_{\bar{w}_k}.
$$
Therefore, also $([\mathcal{N}, v])_\approx, h_l(\bar{w}_l)
\not\models_\mathcal{L} Q^+_{\bar{w}_k}\to Q^-_{\bar{w}_k}$. So fix any $v'_{\bar{w}_k} \in ([U,
v])_\approx$ such that $h_l(\bar{w}_l)\mathrel{([\lhd,
v])_\approx}v'_{\bar{w}_k}$ and $(\{ Q^+_{\bar{w}_k} \},\{ Q^-_{\bar{w}_k}\})$ is $\mathcal{L}$-satisfied at
$(([\mathcal{N}, v])_\approx, v'_{\bar{w}_k})$. We now set:
$$
h_k(\bar{u}_m) :=\left\{%
\begin{array}{ll}
    h_l(\bar{u}_m), & \hbox{if $m \leq l$;} \\
    v'_{\bar{u}_m}, & \hbox{if $m = k$.} \\
\end{array}%
\right.
$$
It is clear then that $h_k$ is a function from $\{\bar{u}_m \in
W^{un(w)}\mid m \leq k\}$ to $([U,
v])_\approx$, since $h_l$ is a
function and we have fixed a unique $v'_{\bar{w}_k} \in ([U,
v])_\approx$ for every $\bar{w}_k \in W^{un(w)}$. It is also immediate
to see that $(iii)_k$ is satisfied. Indeed, we have
$\Gamma_{\bar{w}_k} \subseteq Th^+_\mathcal{L}(([\mathcal{N}, v])_\approx, h_l(\bar{w}_l))$, and, by
$h_l(\bar{w}_l)\mathrel{([\lhd,
	v])_\approx}v'_{\bar{w}_k} =
h_k(\bar{w}_k)$, the preservation of $\mathcal{L}$ under
asimulations, and Lemma \ref{L:generated-submodels}.2, we also get that
$\Gamma_{\bar{w}_k} \subseteq
Th^+_\mathcal{L}(([\mathcal{N}, v])_\approx, h_k(\bar{w}_k))$, so that,
by $$([\mathcal{N}, v])_\approx, v'_{\bar{w}_k}
\models_\mathcal{L} (\{ Q^+_{\bar{w}_k} \},\{ Q^-_{\bar{w}_k}\}),$$ we further obtain that
$Th_\mathcal{L}(([\mathcal{N}, v])_\approx, h_k(\bar{w}_k)) =
Th_\mathcal{L}(\mathcal{M}^\ast, \bar{w}_k)$.

We now address the injectivity of $h_k$. Let $m \leq r \leq k$ and
let $\bar{u}_m, \bar{w}_r \in W^{un(w)}$ be such that $\bar{u}_m
\neq \bar{w}_r$. Then we have
$\bar{u}_m\mathrel{\not\prec^{un(w)}}\bar{w}_r$ or
$\bar{w}_r\mathrel{\not\prec^{un(w)}}\bar{u}_m$ by antisymmetry. In the first case,
we have $Q^+_{\bar{u}_m} \in
Th^+_\mathcal{L}(\mathcal{M}^\ast, \bar{u}_m)\setminus
Th^+_\mathcal{L}(\mathcal{M}^\ast, \bar{w}_r)$, so by
$(iii)_k$ we must have
 $$
 Q^+_{\bar{u}_m} \in
Th^+_\mathcal{L}(([\mathcal{N}, v])_\approx,  h_k(\bar{u}_m))
\setminus Th^+_\mathcal{L}(([\mathcal{N}, v])_\approx,
h_k(\bar{w}_r)),
$$
whence $h_k(\bar{u}_m) \neq h_k(\bar{w}_k)$. In the second case, we
reason symmetrically, using $Q^+_{\bar{w}_r}$ instead
of $Q^+_{\bar{u}_m}$. Thus $h_k$ is injective.

As for $(iv)_k$, if $\bar{u}_m, \bar{w}_r \in W^{un(w)}$ are such
that $m, r \leq k$, then the following cases are possible:

\emph{Case 1}. $m, r \leq l$. But then we have:
\begin{align*}
h_k(\bar{u}_m)\mathrel{([\lhd,
	v])_\approx}h_k(\bar{w}_r)\Leftrightarrow\,
&h_l(\bar{u}_m)\mathrel{([\lhd,
	v])_\approx}h_l(\bar{w}_r) &&\text{(by $h_l \subseteq h_k$)}\\
\Leftrightarrow\,&\bar{u}_m = \bar{w}_m\,\&\,m \leq r &&\text{(by $(iv)_l$)}
\end{align*}
and we are done.

\emph{Case 2}. One of $m,r$ is equal to $k$. We  consider the
($\Leftarrow$)-part first. If $\bar{u}_m = \bar{w}_m$ and $m \leq
r$, then, of course, $r = k$ and we have either $m = k$ or $m \leq
l$. If $m = k$, then we have $h_k(\bar{u}_m) = h_k(\bar{w}_m)
\mathrel{([\lhd,v])_\approx}h_k(\bar{w}_k)$ by reflexivity. If $m \leq l$,
then we have $\bar{u}_m =\bar{w}_m\mathrel{\prec^{un(w)}}\bar{w}_l$
and thus $h_k(\bar{u}_m) = h_k(\bar{w}_m) = h_l(\bar{w}_m)
\mathrel{([\lhd,
	v])_\approx}h_l(\bar{w}_l) = h_k(\bar{w}_l)$ by $(iv)_l$,
and $h_k(\bar{w}_l) = h_l(\bar{w}_l) \mathrel{([\lhd,v])_\approx}
h_k(\bar{w}_k)$ by the choice of $h_k(\bar{w}_k)$, hence also
$h_k(\bar{w}_m)\mathrel{([\lhd,v])_\approx} h_k(\bar{w}_k)$ by
transitivity. As for the ($\Rightarrow$)-part, if $h_k(\bar{u}_m)\mathrel{([\lhd,v])_\approx} h_k(\bar{w}_r)$, then we have $Q^+_{\bar{u}_m} \in Th^+_\mathcal{L}(\mathcal{M}^\ast,
\bar{u}_m) = Th^+_\mathcal{L}(([\mathcal{N}, v])_\approx,
h_k(\bar{u}_m))$ by $(iii)_k$, but then, by Lemma
\ref{L:generated-submodels}.2, preservation of $\mathcal{L}$ under
$\mathcal{L}'$-asimulations, and $(iii)_k$ again, we further get that
$$
Q^+_{\bar{u}_m} \in
Th^+_\mathcal{L}(([\mathcal{N}, v])_\approx, h_k(\bar{w}_r)) =
Th^+_\mathcal{L}(\mathcal{M}^\ast, \bar{w}_r),
$$
whence we
must have $\bar{u}_m\mathrel{\prec^{un(w)}}\bar{w}_r$ so that also
$\bar{u}_m = \bar{w}_m$.

Once $h_k$ is defined, we can show that $(i)_k$ holds, arguing
exactly as in the Induction Basis.

Finally, we address the satisfaction of $(ii)_k$. If $r \leq k$
and $a \in A_{\bar{w}_r}$, then, by  $(i)_k$, let $b \in B_{h_r(\bar{w}_r)}$ be such that $g_{\bar{w}_r}(a) =
[b]_{\approx(h_r(\bar{w}_r))}$ and thus also $\mathcal{N}, h_r(\bar{w}_r) \models_\mathcal{L}
P^+_a(b)$. In addition, we have
$\bar{w}_r\mathrel{\prec^{un(w)}}\bar{w}_k$, hence also
$h_r(\bar{w}_r) = h_k(\bar{w}_r)\mathrel{([\lhd,v])_\approx}h_k(\bar{w}_k)$ by $(iv)_k$. This means that we also have
$h_r(\bar{w}_r)\mathrel{\lhd}h_k(\bar{w}_k)$, since passing to
generated submodels and applying a congruence does not affect the
accessibility relation. Therefore, we must also have $\mathcal{N}, h_k(\bar{w}_k)
\models_\mathcal{L} P^+_a(\mathbb{G}_{h_r(\bar{w}_r)h_k(\bar{w}_k)}(b))$. On
the other hand, again  by  $(i)_k$, we can choose a $c \in B_{h_r(\bar{w}_k)}$ such that 
$g_{\bar{w}_k}(\mathbb{H}_{\bar{w}_r\bar{w}_k}(a)) = [c]_{\approx(h_k(\bar{w}_k))}$ and thus $\mathcal{N}, h_k(\bar{w}_k) \models_\mathcal{L}
P^+_{\mathbb{H}_{\bar{w}_r\bar{w}_k}(a)}(c)$, whence also
$\mathcal{N}, h_k(\bar{w}_k) \models_\mathcal{L} \exists
xP^+_{\mathbb{H}_{\bar{w}_r\bar{w}_k}(a)}(x)$. Next, note that we
have, by $h_k(\bar{w}_k) \in ([U,v])_\approx = [U,v]$, that:
$$
\exists xP^+_{\mathbb{H}_{\bar{w}_r\bar{w}_k}(a)}(x) \to \forall
y(P^+_a(y) \to P^+_{\mathbb{H}_{\bar{w}_r\bar{w}_k}(a)}(y)) \in
Th^+_\mathcal{L}(\mathcal{M}^\ast, w) =
Th^+_\mathcal{L}(\mathcal{N}, v) \subseteq
Th^+_\mathcal{L}(\mathcal{N}, h_k(\bar{w}_k)).
$$
It follows then that $\forall y(P^+_a(y) \to
P^+_{\mathbb{H}_{\bar{w}_r\bar{w}_k}(a)}(y)) \in
Th^+_\mathcal{L}(\mathcal{N}, h_k(\bar{w}_k))$ and thus also
$\mathcal{N}, h_k(\bar{w}_k) \models_\mathcal{L}
P^+_{\mathbb{H}_{\bar{w}_r\bar{w}_k}(a)}(\mathbb{G}_{h_r(\bar{w}_r)h_k(\bar{w}_k)}(b))$,
whence, further, $$[\mathcal{N},v], h_k(\bar{w}_k)
\models_\mathcal{L}
P^+_{\mathbb{H}_{\bar{w}_r\bar{w}_k}(a)}([\mathbb{G},v]_{h_r(\bar{w}_r)h_k(\bar{w}_k)}(b)).$$ 


But then note that we have:
\begin{align*}
\forall x,y(P^+_{\mathbb{H}_{\bar{w}_r\bar{w}_k}(a)}(x) &\wedge
P^+_{\mathbb{H}_{\bar{w}_r\bar{w}_k}(a)}(y) \to \forall z(P^+_d(z)
\leftrightarrow P^+_d(z))) \in
Th^+_\mathcal{L}(\mathcal{M}^\ast, w)=\\
&= Th^+_\mathcal{L}(\mathcal{N}, v) =
Th^+_\mathcal{L}([\mathcal{N}, v], v)\subseteq
Th^+_\mathcal{L}([\mathcal{N}, v], h_k(\bar{w}_k))
\end{align*}
for all $d \in \mathbb{A}$, therefore, we get that
$[\mathcal{N}, v], h_k(\bar{w}_k) \models_\mathcal{L}
P^+_d(([\mathbb{G},v])_{h_r(\bar{w}_r)h_k(\bar{w}_k)})(b)) \leftrightarrow
P^+_d(c)$ for every such $d$, whence further
$([\mathbb{G},v])_{h_r(\bar{w}_r)h_k(\bar{w}_k)})(b)\mathrel{\approx_{(h_k(\bar{w}_k))}}c$
so that:
\begin{align*}
(([\mathbb{G},v])_\approx)_{h_r(\bar{w}_r)h_k(\bar{w}_k)})(g_{\bar{w}_r}(a))&=(([\mathbb{G},v])_\approx)_{h_r(\bar{w}_r)h_k(\bar{w}_k)})([b]_{\approx(h_r(\bar{w}_r))})\\
&= [([\mathbb{G},v])_{h_r(\bar{w}_r)h_k(\bar{w}_k)}(b)]_{\approx(h_k(\bar{w}_k))}\\
&= [c]_{\approx(h_k(\bar{w}_k))}\\
&= g_{\bar{w}_k}(\mathbb{H}_{\bar{w}_r\bar{w}_k}(a))
\end{align*}
In this way, $(ii)_k$ is shown to hold and the Induction Step is proven.

We now define:
\begin{align*}
    g: =& \bigcup \{g_{\bar{w}_k}\mid k < \omega,\,\bar{w}_k \in
W^{un(w)}\}\\
h: =& \bigcup\{ h_k\mid k < \omega \}
\end{align*}
We will show that $(g,h)$ is an $\mathcal{L}$-elementary embedding of
$\mathcal{M}^\ast$ into $([\mathcal{N}, v])_\approx$. It is
obvious that $h: W^{un(w)} \to ([U,v])_\approx$ is an injection,
being a union of the countable increasing chain of injections such
that the sequence of the domains of these injections covers all of
$W^{un(w)}$. It is also clear, that  $g: \mathbb{A}\to
([\mathbb{B}, v])_\approx$ is an injection. Indeed, if $a \in
\mathbb{A}$, then for some $\bar{w}_k \in W^{un(w)}$ we
have $a \in A_{\bar{w}_k} = dom(g_{\bar{w}_k}) \subseteq
dom(g)$, so $g$ is everywhere defined on $\mathbb{A}$.
Moreover, $g$ is a union of a set of functions with pairwise
disjoint domains and hence itself a function. Finally, if
$\bar{u}_m,\bar{w}_k \in W^{un(w)}$ are such that $\bar{u}_m \neq
\bar{w}_k$, $h(\bar{u}_m) \neq h(\bar{w}_k)$, therefore also
$rang(g_{\bar{w}_k})$ is disjoint from $rang(g_{\bar{u}_m})$.
Therefore, $g$ is a union of a set of injections with both
pairwise disjoint domains and pairwise disjoint ranges and hence
itself an injection.

It remains to check that the conditions of Definition
\ref{D:embedding} are satisfied by $(g,h)$. The satisfaction of
condition \eqref{E:c2} follows from the fact that we have
$g\upharpoonright A_{\bar{w}_k} = g_{\bar{w}_k}$ for any
$\bar{w}_k \in W^{un(w)}$ and the satisfaction of $(i)_k$. As for
condition \eqref{E:c1}, assume that $\bar{u}_m,\bar{w}_k \in
W^{un(w)}$. Then we have
\begin{align*}
    \bar{u}_m\mathrel{\prec^{un(w)}}\bar{w}_k &\Leftrightarrow m \leq k\,\&\,\bar{u}_m =
    \bar{w}_m\\
    &\Leftrightarrow h_k(\bar{u}_m)\mathrel{([\lhd,
v])_\approx}h_k(\bar{w}_k) &&\text{(by $(iv)_k$)}\\
&\Leftrightarrow h(\bar{u}_m)\mathrel{([\lhd,
	v])_\approx}h(\bar{w}_k) &&\text{(by definition of $h$)}
\end{align*}
As for the condition \eqref{E:c2a}, assume that
$\bar{u}_m,\bar{w}_k \in W^{un(w)}$. If we have $
\bar{u}_m\mathrel{\prec^{un(w)}}\bar{w}_k$, then we must have both
$m \leq k$ and $\bar{u}_m = \bar{w}_m$. Now, if $a \in
A_{\bar{w}_m}$, it follows from $(ii)_k$ that:
\begin{align*}
    g(\mathbb{H}_{\bar{w}_m\bar{w}_k}(a)) &=g_{\bar{w}_k}(\mathbb{H}_{\bar{w}_m\bar{w}_k}(a))\\
     &= (([\mathbb{G},v])_\approx)_{h_m(\bar{w}_m)h_k(\bar{w}_k)}(g_{\bar{w}_m}(a))\\
&= (([\mathbb{G},v])_\approx)_{h(\bar{w}_m)h(\bar{w}_k)}(g(a))
\end{align*}
It remains to check that condition \eqref{E:c3} is satisfied. We
observe, first, that for any $\bar{u}_m\in W^{un(w)}$, any $a \in
A_{\bar{u}_m}$, and any $b \in (([B,v])_\approx)_{h(\bar{u}_m)}$ we have:
\begin{equation}\label{E:lm3}
([\mathcal{N}, v])_\approx, h(\bar{u}_m)\models_\mathcal{L} P^+_a(b)
\Leftrightarrow b = g(a)
\end{equation}
Indeed, ($\Leftarrow$), if $b = g(a)$, then $b =
g_{\bar{u}_m}(a)$, therefore, for some $c \in B_{h(\bar{u}_m)}$ we
must have both $\mathcal{N}, h(\bar{u}_m) \models_\mathcal{L}
P^+_a(c)$ and $b = [c]_{\approx(h(\bar{u}_m))}$. But then also $([\mathcal{N}, v])_\approx, h(\bar{u}_m)\models_\mathcal{L} P^+_a(b)$ by the definition of congruence. In the other direction, $(\Rightarrow)$, if
$([\mathcal{N}, v])_\approx, h(\bar{u}_m)\models_\mathcal{L}
P^+_a(b)$ then we know that for some $c \in B_{h(\bar{u}_m)}$ we
must have $b = [c]_{\approx(h(\bar{u}_m))}$, whence $([\mathcal{N}, v])_\approx,
h(\bar{u}_m)\models_\mathcal{L} P^+_a([c]_{\approx(h(\bar{u}_m))})$, so that, further,
$[\mathcal{N}, v], h(\bar{u}_m)\models_\mathcal{L} P^+_a(c)$ by
the definition of congruence, and finally $\mathcal{N},
h(\bar{u}_m)\models_\mathcal{L} P^+_a(c)$ by Lemma \ref{L:generated-submodels}.1. But then we have $b =
g_{\bar{u}_m}(a)$, and hence also $b = g(a)$.

Having established \eqref{E:lm3}, we now reason as follows, for
any given $\bar{u}_m\in W^{un(w)}$, $\bar{a}_n \in
A_{\bar{u}_m}^n$, any tuple of pairwise distinct fresh constants
$\bar{c}_n$, and any $\phi \in
L(\Theta_{\mathcal{M}}\cup \{\bar{c}_n\})$:
\begin{align*}
\phi \in Tp^+_{\mathcal{L}}(\mathcal{M}^\ast, &\bar{u}_m, \bar{c}_n/\bar{a}_n) \Leftrightarrow ([\mathcal{M}^\ast, \bar{u}_m], \bar{c}_n/\bar{a}_n), \bar{u}_m \models_\mathcal{L} \phi\\
&\Leftrightarrow (\forall \bar{u}_r \succ^{un(w)} \bar{u}_m)(([\mathcal{M}^\ast, \bar{u}_m], \bar{c}_n/\bar{a}_n), \bar{u}_r \models_\mathcal{L} \phi)\\
&\text{(by Lemma \ref{L:generated-submodels}.2)}\\
&\Leftrightarrow (\forall \bar{u}_r \succ^{un(w)} \bar{u}_m)([([\mathcal{M}^\ast, \bar{u}_m], \bar{c}_n/\bar{a}_n), \bar{u}_r], \bar{u}_r \models_\mathcal{L} \phi)\\
&\text{(by Lemma \ref{L:generated-submodels}.1)}\\
&\Leftrightarrow (\forall \bar{u}_r \succ^{un(w)} \bar{u}_m)(([\mathcal{M}^\ast, \bar{u}_r], \bar{c}_n/\mathbb{H}^{un(w)}_{\bar{u}_m\bar{u}_r}\langle\bar{a}_n\rangle), \bar{u}_r \models_\mathcal{L} \phi)\\
&\text{(by Lemma \ref{L:cutoff-constants}.2)}\\ 
&\Leftrightarrow (\forall \bar{u}_r \succ^{un(w)} \bar{u}_m)(\forall \bar{b}_n \in A_{\bar{u}_r})(\bar{b}_n = \mathbb{H}_{\bar{u}_m\bar{u}_r}^{un(w)}\langle\bar{a}_n\rangle \Rightarrow ([\mathcal{M}^\ast, \bar{u}_r], \bar{c}_n/\bar{b}_n), \bar{u}_r \models_\mathcal{L} \phi)\\
&\Leftrightarrow (\forall \bar{u}_r \succ^{un(w)} \bar{u}_m)(\forall \bar{b}_n \in A_{\bar{u}_r})(([\mathcal{M}^\ast, \bar{u}_r], \bar{c}_n/\bar{b}_n), \bar{u}_r \models_\mathcal{L}\bigwedge^n_{i = 1}P^+_{a_i}(c_i) \Rightarrow\\
&\qquad\qquad\qquad\Rightarrow ([\mathcal{M}^\ast, \bar{u}_r], \bar{c}_n/\bar{b}_n), \bar{u}_r \models_\mathcal{L} \phi)\\
&\text{(by def. of $\mathcal{M}^\ast$)}\\
&\Leftrightarrow (\forall \bar{u}_r \succ^{un(w)} \bar{u}_m)(\forall \bar{b}_n \in A_{\bar{u}_r})(([\mathcal{M}^\ast, \bar{u}_r], \bar{c}_n/\bar{b}_n), \bar{u}_r \models_\mathcal{L}\bigwedge^n_{i = 1}P^+_{a_i}(c_i) \to \phi)\\
&\Leftrightarrow \mathcal{M}^\ast, \bar{u}_m \models_\mathcal{L}\forall\bar{c}_n(\bigwedge^n_{i = 1}P^+_{a_i}(c_i) \to \phi)\\
&\text{(by Corollary \ref{C:generated-submodels}.3)}\\
&\Leftrightarrow [\mathcal{N}, v]_\approx, h_m(\bar{u}_m) \models_\mathcal{L}\forall\bar{c}_n(\bigwedge^n_{i = 1}P^+_{a_i}(c_i) \to \phi)\\
&\text{(by $(iii)_m$)}\\
&\Leftrightarrow (\forall v\mathrel{[\rhd, v]_\approx} h_m(\bar{u}_m))(\forall \bar{b}_n \in ([B,v]_\approx)_v)(([[\mathcal{N}, v]_\approx, h_m(\bar{u}_m)], \bar{c}_n/\bar{b}_n), v \models_\mathcal{L}\bigwedge^n_{i = 1}P^+_{a_i}(c_i) \to \phi)\\
&\text{(by Corollary \ref{C:generated-submodels}.3)}\\
&\Leftrightarrow (\forall v\mathrel{[\rhd, v]_\approx} h_m(\bar{u}_m))(\forall \bar{b}_n \in ([B,v]_\approx)_v)(([[\mathcal{N}, v]_\approx, h_m(\bar{u}_m)], \bar{c}_n/\bar{b}_n), v \models_\mathcal{L}\bigwedge^n_{i = 1}P^+_{a_i}(c_i) \Rightarrow\\
&\qquad\qquad\qquad\Rightarrow ([[\mathcal{N}, v]_\approx, h_m(\bar{u}_m)], \bar{c}_n/\bar{b}_n), v\models_\mathcal{L} \phi)\\
&\Leftrightarrow (\forall v\mathrel{[\rhd, v]_\approx} h_m(\bar{u}_m))(\forall \bar{b}_n \in ([B,v]_\approx)_v)(\bar{b}_n = g\langle\bar{a}_n\rangle \Rightarrow ([[\mathcal{N}, v]_\approx, h_m(\bar{u}_m)], \bar{c}_n/\bar{b}_n), v\models_\mathcal{L} \phi)\\
&\text{(by \eqref{E:lm3})}\\
 &\Leftrightarrow ([[\mathcal{N}, v]_\approx, h_m(\bar{u}_m)], \bar{c}_n/g\langle\bar{a}_n\rangle), h_m(\bar{u}_m) \models_\mathcal{L} \phi\\
&\text{(reasoning as above)}\\
&\Leftrightarrow \phi \in Tp^+_{\mathcal{L}}([\mathcal{N}, v]_\approx, h_m(\bar{u}_m), \bar{c}_n/g\langle\bar{a}_n\rangle)\\
&\Leftrightarrow \phi \in Tp^+_{\mathcal{L}}([\mathcal{N}, v]_\approx, h(\bar{u}_m), \bar{c}_n/g\langle\bar{a}_n\rangle)\\
&\text{(by definition of $h$)}
\end{align*}
It follows that $Tp^+_{\mathcal{L}}(\mathcal{M}^\ast, \bar{u}_m, \bar{c}_n/\bar{a}_n) = Tp^+_{\mathcal{L}}([\mathcal{N}, v]_\approx, h(\bar{u}_m), \bar{c}_n/g\langle\bar{a}_n\rangle)$, whence clearly also $Tp_{\mathcal{L}}(\mathcal{M}^\ast, \bar{u}_m, \bar{c}_n/\bar{a}_n) = Tp_{\mathcal{L}}([\mathcal{N}, v]_\approx, h(\bar{u}_m), \bar{c}_n/g\langle\bar{a}_n\rangle)$, so that condition  \eqref{E:c3} (for $\mathcal{L}$) is
satisfied.
 \end{proof}

\begin{proposition}\label{L:saturation}
Let $\mathcal{L}\sqsupseteq \mathcal{L}' \in StIL$ be an abstract intuitionistic logic which is
preserved under $\mathcal{L}'$-asimulations, $\star$-compact, and has TUP, and let $(\mathcal{M}, w) \in Str_\mathcal{L}(\Theta)$ be an unravelled model. Then there exists an unravelled model $(\mathcal{N}', w)\in Str_\mathcal{L}(\Theta)$ such that $\mathcal{M}\preccurlyeq_\mathcal{L} \mathcal{N}'$ and every $\mathcal{L}$-type of 
$\mathcal{M}$ is realized in $\mathcal{N}'$.
\end{proposition}
\begin{proof}
Assume the hypothesis of the Proposition. For the given $\mathcal{M}$, consider $\mathcal{M}^\ast$ and fix a countable set $\{c_i \mid i > 0\}$ of pairwise distinct fresh individual constants; we extend the signature of this model as
follows:
\begin{itemize}
    \item For every $\bar{v}_k \in W^{un(w)}$ and every $\bar{a}_n \in
    A^n_{\bar{v}_k}$, if $\Gamma, \Delta \subseteq
    L(\Theta_{\mathcal{M}}\cup \{\bar{c}_n\})$ are such that
    $(\Gamma, \Delta)$ is a $\bar{c}_n$-successor $\mathcal{L}$-type of
    $(\mathcal{M}^\ast,\bar{v}_k, \bar{a}_n)$, then we add two fresh
    $n$-ary predicate letters $R^+_{\Gamma,\Delta,\bar{v}_k,\bar{a}_n}$
    and $R^-_{\Gamma,\Delta,\bar{v}_k,\bar{a}_n}$.

    \item For every $\bar{v}_k \in W^{un(w)}$, every $\bar{a}_n \in
    A^n_{\bar{v}_k}$, if $\Xi \subseteq
    L(\Theta_{\mathcal{M}}\cup \{\bar{c}_{n+1}\})$ is a $\bar{c}_{n+1}$-existential $\mathcal{L}$-type of
    $(\mathcal{M}^\ast,\bar{v}_k,\bar{a}_n)$, then we add a fresh
    $(n + 1)$-ary predicate letter $R^{ex}_{\Xi,\bar{v}_k,\bar{a}_n}$.

    \item For every $\bar{v}_k \in W^{un(w)}$, every $\bar{a}_n \in
    A^n_{\bar{v}_k}$, if $\Xi \subseteq
    L(\Theta_{\mathcal{M}}\cup \{\bar{c}_{n+1}\})$ is a $\bar{c}_{n+1}$-universal $\mathcal{L}$-type of
    $(\mathcal{M}^\ast,\bar{v}_k,\bar{a}_n)$, then we add a fresh
    $(n + 1)$-ary predicate letter $R^{all}_{\Xi,\bar{v}_k,\bar{a}_n}$.
\end{itemize}
We will call the resulting signature $\Theta'$. Consider next the
following $L(\Theta')$-theory
$(\Upsilon,Th^-_\mathcal{L}(\mathcal{M}^\ast, w))$, where:
\begin{align*}
\Upsilon = &Th^+_\mathcal{L}(\mathcal{M}^\ast, w) \cup\\
&\cup
\{\forall\bar{c}_n(R^+_{\Gamma,\Delta,\bar{v}_k,\bar{a}_n}(\bar{c}_n) \to
\phi), \forall\bar{c}_n(\psi \to
R^-_{\Gamma,\Delta,\bar{v}_k,\bar{a}_n}(\bar{c}_n)),
\\
&\qquad\quad\forall\bar{c}_n((R^+_{\Gamma,\Delta,\bar{v}_k,\bar{a}_n}(\bar{c}_n) \to
R^-_{\Gamma,\Delta,\bar{v}_k,\bar{a}_n}(\bar{c}_n)) \to
\bigvee^n_{i = 1}P^-_{a_i}(c_i))\mid R^+_{\Gamma,\Delta,\bar{v}_k,\bar{a}_n},
R^-_{\Gamma,\Delta,\bar{v}_k,\bar{a}_n} \in \Theta', \phi \in \Gamma,
\psi \in \Delta,
\bar{v}_k \in W^{un(w)} \}\\
&\cup \{\forall\bar{c}_{n+1}(R^{ex}_{\Xi,\bar{v}_k,\bar{a}_n}(\bar{c}_{n+1}) \to \phi),
\forall\bar{c}_{n}(\bigwedge^n_{i = 1}P^+_{a_i}(c_i) \to \exists
c_{n+1}R^{ex}_{\Xi,\bar{v}_k,\bar{a}_n}(\bar{c}_{n+1}))\mid
R^{ex}_{\Xi,\bar{v}_k,\bar{a}_n} \in \Theta', \phi \in \Xi, \bar{v}_k \in W^{un(w)} \}\\
&\cup \{\forall\bar{c}_{n+1}(\psi \to
R^{all}_{\Xi,\bar{v}_k,\bar{a}_n}(\bar{c}_{n+1})),
\forall\bar{c}_{n}(\forall c_{n+1}R^{all}_{\Xi,\bar{v}_k,\bar{a}_n}(\bar{c}_{n+1}) \to \bigvee^n_{i = 1}P^-_{a_i}(c_i))\mid R^{all}_{\Xi,\bar{v}_k,\bar{a}_n} \in
\Theta', \phi \in \Xi, \bar{v}_k \in W^{un(w)} \},
\end{align*}

$(\Upsilon,Th^-_\mathcal{L}(\mathcal{M}^\ast, w))$ is itself
finitely $\mathcal{L}$-satisfiable, since given an arbitrary $(\Xi_0, \Omega_0) \Subset (\Upsilon, Th^-_\mathcal{L}(\mathcal{M}^\ast, w))$
 we know, wlog, that, for some $n < \omega$ we have:
$$
\Xi_0 \subseteq Th^+_\mathcal{L}(\mathcal{M}^\ast, w) \cup
T_1 \cup, \ldots, \cup T_n,
$$
where for every $1 \leq i \leq n$ one of the following cases
holds:

\emph{Case 1}. For some  $\bar{v}_k \in W^{un(w)}$ and some
$\bar{a}_n \in A^n_{\bar{v}_k}$, there exist  $\Gamma, \Delta \subseteq
L(\Theta_{\mathcal{M}}\cup \{\bar{c}_n\})$ such that
$(\Gamma, \Delta)$ is a $\bar{c}_n$-successor $\mathcal{L}$-type of
$(\mathcal{M}^\ast,\bar{v}_k, \bar{a}_n)$ and $\Gamma'
    \Subset\Gamma$, $\Delta' \Subset \Delta$ such that we have:

\begin{align*}
T_i &\subseteq
\{\forall\bar{c}_n(R^+_{\Gamma,\Delta,\bar{v}_k,\bar{a}_n}(\bar{c}_n) \to
\phi), \forall\bar{c}_n(\psi \to
R^-_{\Gamma,\Delta,\bar{v}_k,\bar{a}_n}(\bar{c}_n)),
\\
&\qquad\quad\forall\bar{c}_n((R^+_{\Gamma,\Delta,\bar{v}_k,\bar{a}_n}(\bar{c}_n) \to
R^-_{\Gamma,\Delta,\bar{v}_k,\bar{a}_n}(\bar{c}_n)) \to \bigvee^n_{i = 1}P^-_{a_i}(c_i))\mid \phi \in \Gamma', \psi \in \Delta'\}
\end{align*}
By Definition \ref{D:types}, $(\Gamma', \Delta')$ is
$\mathcal{L}$-satisfied at $(([\mathcal{M}^\ast, \bar{v}_m], \bar{c}_n/\mathbb{H}_{\bar{v}_k\bar{v}_m}\langle\bar{a}_n\rangle), \bar{v}_m)$
for some $\bar{v}_m \in W^{un(w)}$ and $m \geq k$. We set $\bar{b}_n := \mathbb{H}_{\bar{v}_k\bar{v}_m}\langle\bar{a}_n\rangle$ and claim that $T_i$
will be satisfied at $w$ in the expansion of
$\mathcal{M}^\ast$ in which
$R^+_{\Gamma,\Delta,\bar{v}_k,\bar{a}_n}(\bar{y}_n)$ and
$R^-_{\Gamma,\Delta,\bar{v}_k,\bar{a}_n}(\bar{y}_n)$ are identified with $\bigwedge^n_{i = 1}P^+_{b_i}(y_i)$ and $\bigvee^n_{i = 1}P^-_{b_i}(y_i)$, respectively.

Indeed, if $\phi \in \Gamma'$, then we obviously have $\mathcal{M}^\ast, w \models_{\mathcal{L}} \forall\bar{c}_n(\bigwedge^n_{i = 1}P^+_{b_i}(c_i) \to
\phi)$. If $\bar{u}_r \in W^{un(w)}$ and $\bar{e}_n \in A^n_{\bar{u}_r}$ are such that $([\mathcal{M}^\ast, \bar{u}_r], \bar{c}_n/\bar{e}_n), \bar{u}_r \models_{\mathcal{L}} \bigwedge^n_{i = 1}P^+_{b_i}(c_i)$, then we must have $\bar{e}_n = \mathbb{H}_{\bar{v}_m\bar{u}_r}\langle\bar{b}_n\rangle$, so, in particular, $\mathbb{H}_{\bar{v}_m\bar{u}_r}\langle\bar{b}_n\rangle$ must be defined and we must therefore also have $\bar{u}_r \succ^{un(w)} \bar{v}_m$. But then, since we have, by the choice of $\bar{v}_m$, that 
$([\mathcal{M}^\ast, \bar{v}_m], \bar{c}_n/\bar{b}_n), \bar{v}_m \models_{\mathcal{L}} \phi$, so that, by Lemma \ref{L:generated-submodels}.2,
$([\mathcal{M}^\ast, \bar{v}_m], \bar{c}_n/\bar{b}_n), \bar{u}_r \models_{\mathcal{L}} \phi$, whence, by Lemma \ref{L:generated-submodels}.1,
$[([\mathcal{M}^\ast, \bar{v}_m], \bar{c}_n/\bar{b}_n), \bar{u}_r], \bar{u}_r \models_{\mathcal{L}} \phi$. But then, by Lemma \ref{L:cutoff-constants}.2, $([\mathcal{M}^\ast, \bar{u}_r], \bar{c}_n/\bar{e}_n), \bar{u}_r \models_{\mathcal{L}} \phi$, as desired.

On the other hand, if  $\psi \in \Delta'$, then we have $\mathcal{M}^\ast, w \models_{\mathcal{L}} \forall\bar{c}_n(\psi \to
\bigvee^n_{i = 1}P^-_{b_i}(c_i))$. We argue by contraposition. If $\bar{u}_r \in W^{un(w)}$ and $\bar{e}_n \in A^n_{\bar{u}_r}$ are such that every formula of the form $P^-_{b_i}(c_i)$ for $1 \leq i \leq n$ fails at $(([\mathcal{M}^\ast, \bar{u}_r], \bar{c}_n/\bar{e}_n), \bar{u}_r)$, then we must have both $\bar{u}_r \prec^{un(w)} \bar{v}_m$ and 
$\mathbb{H}_{\bar{u}_r\bar{v}_m}\langle\bar{e}_n\rangle = \bar{b}_n$.

But then clearly $$([\mathcal{M}^\ast, \bar{v}_m], \bar{c}_n/\mathbb{H}_{\bar{u}_r\bar{v}_m}\langle\bar{e}_n\rangle), \bar{v}_m \not\models_{\mathcal{L}} \psi,$$ whence, by Lemma \ref{L:cutoff-constants}.2, $[([\mathcal{M}^\ast, \bar{u}_r], \bar{c}_n/\bar{e}_n), \bar{v}_m], \bar{v}_m \not\models_{\mathcal{L}} \psi$. But then, by Lemma \ref{L:generated-submodels}.1, also $([\mathcal{M}^\ast, \bar{u}_r], \bar{c}_n/\bar{e}_n), \bar{v}_m \not\models_{\mathcal{L}} \psi$, hence, by Lemma \ref{L:generated-submodels}.2, $([\mathcal{M}^\ast, \bar{u}_r], \bar{c}_n/\bar{e}_n), \bar{u}_r \not\models_{\mathcal{L}} \psi$, as desired.

Finally, we show that also 
$$
\mathcal{M}^\ast, w \models_{\mathcal{L}} \forall\bar{c}_n((\bigwedge^n_{i = 1}P^+_{b_i}(c_i) \to
\bigvee^n_{i = 1}P^-_{b_i}(c_i)) \to
\bigvee^n_{i = 1}P^-_{a_i}(c_i)).
$$
Again, we argue by contraposition. Indeed, let  $\bar{u}_r \in W^{un(w)}$ and $\bar{e}_n \in A^n_{\bar{u}_r}$ be such that  $([\mathcal{M}^\ast, \bar{u}_r], \bar{c}_n/\bar{e}_n), \bar{u}_r \not\models_{\mathcal{L}} \bigvee^n_{i = 1}P^-_{a_i}(c_i)$. This means that $\bar{u}_{r} \prec^{un(w)} \bar{v}_k$, and hence, by transitivity, $\bar{u}_{r} \prec^{un(w)} \bar{v}_m$. On the other hand, this also means that $\mathbb{H}_{\bar{u}_r\bar{v}_k}\langle\bar{e}_n\rangle = \bar{a}_n$ and thus also $\mathbb{H}_{\bar{u}_r\bar{v}_m}\langle\bar{e}_n\rangle = \mathbb{H}_{\bar{v}_k\bar{v}_m}\langle\mathbb{H}_{\bar{u}_r\bar{v}_k}\langle\bar{e}_n\rangle\rangle = \mathbb{H}_{\bar{v}_k\bar{v}_m}\langle\bar{a}_n\rangle = \bar{b}_n$. Furthermore, we also have $([\mathcal{M}^\ast, \bar{v}_m], \bar{c}_n/\bar{b}_n), \bar{v}_m \models_{\mathcal{L}} (\{\bigwedge^n_{i = 1}P^+_{b_i}(c_i)\}, \{\bigvee^n_{i = 1}P^-_{b_i}(c_i)\})$ by definition, whence, further $([\mathcal{M}^\ast, \bar{v}_m], \bar{c}_n/\mathbb{H}^{un(w)}_{\bar{u}_r\bar{v}_m}\langle\bar{e}_n\rangle), \bar{v}_m \models_{\mathcal{L}} (\{\bigwedge^n_{i = 1}P^+_{b_i}(c_i)\}, \{\bigvee^n_{i = 1}P^-_{b_i}(c_i)\})$, and, by Lemma \ref{L:cutoff-constants}.2, $[([\mathcal{M}^\ast, \bar{u}_r], \bar{c}_n/\bar{e}_n), \bar{v}_m], \bar{v}_m \models_{\mathcal{L}} (\{\bigwedge^n_{i = 1}P^+_{b_i}(c_i)\}, \{\bigvee^n_{i = 1}P^-_{b_i}(c_i)\})$.
It follows then, by Lemma \ref{L:generated-submodels}.1, that $([\mathcal{M}^\ast, \bar{u}_r], \bar{c}_n/\bar{e}_n), \bar{v}_m \models_{\mathcal{L}} (\{\bigwedge^n_{i = 1}P^+_{b_i}(c_i)\}, \{\bigvee^n_{i = 1}P^-_{b_i}(c_i)\})$ so that, by $\bar{u}_{r} \prec^{un(w)} \bar{v}_m$, it follows that $([\mathcal{M}^\ast, \bar{u}_r], \bar{c}_n/\bar{e}_n), \bar{u}_r \not\models_{\mathcal{L}} \bigwedge^n_{i = 1}P^+_{b_i}(c_i) \to
\bigvee^n_{i = 1}P^-_{b_i}(c_i)$ as desired.

\emph{Case 2}. For some $\bar{v}_k \in W^{un(w)}$ and some $\bar{a}_n \in
A^n_{\bar{v}_k}$, there exists a $\bar{c}_{n+1}$-existential $\mathcal{L}$-type $\Xi \subseteq
L(\Theta_{\mathcal{M}}\cup \{\bar{c}_{n+1}\})$  of
$(\mathcal{M}^\ast,\bar{v}_k,\bar{a}_n)$ and a $\Xi'
    \Subset\Xi$ such that we have:

\begin{align*}
T_i &\subseteq \{\forall\bar{c}_{n+1}(R^{ex}_{\Xi,\bar{v}_k,\bar{a}_n}(\bar{c}_{n+1}) \to \phi),
\forall\bar{c}_{n}(\bigwedge^n_{i = 1}P^+_{a_i}(c_i) \to \exists
c_{n+1}R^{ex}_{\Xi,\bar{v}_k,\bar{a}_n}(\bar{c}_{n+1}))\mid \phi
\in \Xi'\}
\end{align*}
By Definition \ref{D:types}, $\Xi'$ is
$\mathcal{L}$-satisfied at $(([\mathcal{M}^\ast, \bar{v}_k], \bar{c}_{n+1}/\bar{a}_{n+1}), \bar{v}_k)$
for some $a_{n+1} \in A_{\bar{v}_k}$. But then $T_i$ will be
satisfied at $w$ in the expansion of $\mathcal{M}^\ast$ in
which $R^{ex}_{\Xi,\bar{v}_k,\bar{a}_n}$ is identified with
$\bigwedge^{n + 1}_{i = 1}P^+_{a_i}(c_i)$. Indeed, we can show that $\mathcal{M}^\ast, w \models_{\mathcal{L}} \forall\bar{c}_{n+1}(\bigwedge^{n + 1}_{i = 1}P^+_{a_i}(c_i) \to
\phi)$ for every $\phi
\in \Xi'$ arguing as in Case 1.

Furthermore, we can show that we have:
$$
\mathcal{M}^\ast, w \models_{\mathcal{L}} \forall\bar{c}_{n}(\bigwedge^n_{i = 1}P^+_{a_i}(c_i) \to \exists
c_{n+1}\bigwedge^{n + 1}_{i = 1}P^+_{a_i}(c_i)).
$$
Indeed, let  $\bar{u}_r \in W^{un(w)}$ and $\bar{e}_n \in A^n_{\bar{u}_r}$ be such that  $([\mathcal{M}^\ast, \bar{u}_r], \bar{c}_n/\bar{e}_n), \bar{u}_r \models_{\mathcal{L}} \bigwedge^n_{i = 1}P^+_{a_i}(c_i)$. This means that $\bar{u}_{r} \succ^{un(w)} \bar{v}_k$, and that $\mathbb{H}_{\bar{v}_k\bar{u}_r}\langle\bar{a}_n\rangle = \bar{e}_n$. But then, we can also set $e_{n + 1}: = \mathbb{H}_{\bar{v}_k\bar{u}_r}(a_{n + 1})$, and we clearly have $([\mathcal{M}^\ast, \bar{u}_r], \bar{c}_{n + 1}/\bar{e}_{n + 1}), \bar{u}_r \models_{\mathcal{L}} \bigwedge^{n + 1}_{i = 1}P^+_{a_i}(c_i)$, thus also, by Corollary \ref{C:generated-submodels}.4, $([\mathcal{M}^\ast, \bar{u}_r], \bar{c}_{n}/\bar{e}_{n}), \bar{u}_r \models_{\mathcal{L}} \exists
c_{n+1}\bigwedge^{n + 1}_{i = 1}P^+_{a_i}(c_i)$, as desired.

\emph{Case 3}. For some $\bar{v}_k \in W^{un(w)}$ and some $\bar{a}_n \in
A^n_{\bar{v}_k}$, there exists a $\bar{c}_{n+1}$-universal $\mathcal{L}$-type $\Xi \subseteq
L(\Theta_{\mathcal{M}}\cup \{\bar{c}_{n+1}\})$ of
$(\mathcal{M}^\ast,\bar{v}_k, \bar{a}_n)$ and some $\Xi'
    \Subset\Xi$ such that we have:
\begin{align*}
T_i &\subseteq \{\forall\bar{c}_{n+1}(\psi \to
R^{all}_{\Xi,\bar{v}_k,\bar{a}_n}(\bar{c}_{n+1})),
\forall\bar{c}_{n}(\forall c_{n+1}R^{all}_{\Xi,\bar{v}_k,\bar{a}_n}(\bar{c}_{n+1}) \to \bigvee^n_{i = 1}P^-_{a_i}(c_i))\mid \phi \in \Xi'\}
\end{align*}
By Definition \ref{D:types}, there exist $m \geq k$,
$\bar{v}_m \in W^{un(w)}$, and $b \in A_{\bar{v}_m}$
such that $\Xi'$ is
$\mathcal{L}$-falsified at $(([\mathcal{M}^\ast, \bar{v}_m], \bar{c}_{n + 1}/\mathbb{H}_{\bar{v}_k\bar{v}_m}\langle\bar{a}_n\rangle^\frown b), \bar{v}_m)$. We set $\bar{b}_{n + 1} := \mathbb{H}_{\bar{v}_k\bar{v}_m}\langle\bar{a}_n\rangle^\frown b$ and claim that $T_i$ will be
satisfied at $w$ in the expansion of $\mathcal{M}^\ast$ in
which $R^{all}_{\Xi,\bar{v}_k,\bar{a}_n}$ is identified with
$\bigvee^{n + 1}_{i = 1}P^-_{b_i}(c_i)$. Indeed, we can show that $\mathcal{M}^\ast, w \models_{\mathcal{L}} \forall\bar{c}_{n+1}(\psi \to \bigvee^{n + 1}_{i = 1}P^-_{b_i}(c_i))$ for every $\psi
\in \Xi'$ arguing as in Case 1.

Furthermore, we can show that we have:
$$
\mathcal{M}^\ast, w \models_{\mathcal{L}} \forall\bar{c}_{n}(\forall c_{n+1}\bigvee^{n + 1}_{i = 1}P^-_{b_i}(c_i) \to \bigvee^n_{i = 1}P^-_{a_i}(c_i)).
$$
Again, we argue by contraposition. Indeed, let  $\bar{u}_r \in W^{un(w)}$ and $\bar{e}_n \in A^n_{\bar{u}_r}$ be such that  $([\mathcal{M}^\ast, \bar{u}_r], \bar{c}_n/\bar{e}_n), \bar{u}_r \not\models_{\mathcal{L}} \bigvee^n_{i = 1}P^-_{a_i}(c_i)$. This means that $\bar{u}_{r} \prec^{un(w)} \bar{v}_k$, and hence, by transitivity, $\bar{u}_{r} \prec^{un(w)} \bar{v}_m$. On the other hand, this also means that $\mathbb{H}_{\bar{u}_r\bar{v}_k}\langle\bar{e}_n\rangle = \bar{a}_n$ and thus also $\mathbb{H}_{\bar{u}_r\bar{v}_m}\langle\bar{e}_n\rangle = \mathbb{H}_{\bar{v}_k\bar{v}_m}\langle\mathbb{H}_{\bar{u}_r\bar{v}_k}\langle\bar{e}_n\rangle\rangle = \mathbb{H}_{\bar{v}_k\bar{v}_m}\langle\bar{a}_n\rangle = \bar{b}_n$. On the other hand, we also have $([\mathcal{M}^\ast, \bar{v}_m], \bar{c}_{n+1}/\bar{b}_{n+1}), \bar{v}_m \not\models_{\mathcal{L}} \bigvee^{n + 1}_{i = 1}P^-_{b_i}(c_i)$ by definition, whence further $([\mathcal{M}^\ast, \bar{v}_m], \bar{c}_{n+1}/\mathbb{H}_{\bar{u}_r\bar{v}_m}\langle\bar{e}_n\rangle^\frown b), \bar{v}_m \not\models_{\mathcal{L}} \bigvee^{n + 1}_{i = 1}P^-_{b_i}(c_i)$, and, by $\bar{u}_{r} \prec^{un(w)} \bar{v}_m$ and Corollary \ref{C:generated-submodels}.4, $([\mathcal{M}^\ast, \bar{u}_r], \bar{c}_n/\bar{e}_n), \bar{u}_r \not\models_{\mathcal{L}} \forall c_{n + 1}\bigvee^{n + 1}_{i = 1}P^-_{b_i}(c_i)$, as desired. 

Note, moreover, that for all $1 \leq i < j \leq n$, the set of
predicates that needs to be added to $\mathcal{M}^\ast$ in
order to get $T_i$ satisfied at $w$ is disjoint from the set of
predicates to be added to this same model in order to get $T_j$
satisfied at $w$. Therefore, we can take the union of the
expansions required by $T_1, \ldots, T_n$ and get (by Expansion property) an expansion
$\mathcal{M}'$ of $\mathcal{M}^\ast$ such that $(\Xi_0,
\Omega_0)$ is satisfied at $(\mathcal{M}', w)$. Since $(\Xi_0,
\Omega_0) \Subset (\Upsilon,Th^-_\mathcal{L}(\mathcal{M}^\ast,
w))$ was chosen arbitrarily, this means, by the
$\star$-compactness of $\mathcal{L}$, that
$(\Upsilon,Th^-_\mathcal{L}(\mathcal{M}^\ast, w))$ itself is
$\mathcal{L}$-satisfiable.

Let $(\mathcal{M}'_1, w_1) \in Str_\mathcal{L}(\Theta')$ be a pointed $\Theta'$-model
$\mathcal{L}$-satisfying
$(\Upsilon,Th^-_\mathcal{L}(\mathcal{M}^\ast, w))$. We set
$\mathcal{M}_1: = \mathcal{M}'_1\upharpoonright\Theta_{\mathcal{M}}$; it is obvious, that we also have $(\mathcal{M}_1, w_1) \in Str_\mathcal{L}(\Theta_{\mathcal{M}})$. We know, by Expansion property,
that $(\mathcal{M}_1, w_1)$ $\mathcal{L}$-satisfies
$Th_\mathcal{L}(\mathcal{M}^\ast, w)$, therefore, by Lemma
\ref{unravellinglemma}.3, $(\mathcal{M}_1^{un(w_1)}, w_1)\in Str_\mathcal{L}(\Theta_{\mathcal{M}})$ also
$\mathcal{L}$-satisfies $Th_\mathcal{L}(\mathcal{M}^\ast, w)$
and, by Lemma \ref{L:lemma1}, there must be an
$\mathcal{L}$-elementary embedding $(g,h)$ of $\mathcal{M}^\ast$ into
$([\mathcal{M}_1^{un(w_1)},w_1])_\approx =
(\mathcal{M}_1^{un(w_1)})_\approx\in Str_\mathcal{L}(\Theta_{\mathcal{M}})$, where $\approx$ is defined as in
Lemma \ref{L:lemma1-cong}, and for this elementary embedding we will have
with $g(w) = w_1$. Note, moreover, that, by Lemma \ref{L:embedding}.3, we have then $(g,h)(\mathcal{M}^\ast) \preccurlyeq_\mathcal{L} (\mathcal{M}_1^{un(w_1)})_\approx$ where $(g,h)(\mathcal{M}^\ast)$ is a $(g,h)$-isomorphic copy of $\mathcal{M}^\ast$. Therefore, by Lemma \ref{L:isomorphic-correction}, there must exist a $\Theta_{\mathcal{M}}$-model $\mathcal{N}$ and a pair of functions $g' \supseteq g$, $h' \supseteq h$ such that $\mathcal{M}^\ast \preccurlyeq_\mathcal{L} \mathcal{N}$, and $(g', h'): \mathcal{N} \cong (\mathcal{M}_1^{un(w_1)})_\approx$. In virtue of the latter isomorphism, we also have $(\mathcal{N}, w) \in Str_\mathcal{L}(\Theta_{\mathcal{M}})$. We now prove the following:

\emph{Claim}. $\mathcal{N}$ realizes every $\mathcal{L}$-type
of $\mathcal{M}^\ast$.

To prove this Claim, we again have to consider the three cases
outlined above, keeping fixed a set $\{c_i \mid i > 0\}$ of pairwise distinct fresh individual constants.

\emph{Case 1}. For some $\bar{u}_r \in W^{un(w)}$ and some $\bar{a}_n \in  A^n_{\bar{u}_r}$ the sets $\Gamma, \Delta \subseteq
L(\Theta_{\mathcal{M}^\ast}\cup \{\bar{c}_n\})$ are such that
$(\Gamma, \Delta)$ is a $\bar{c}_n$-successor $\mathcal{L}$-type of
$(\mathcal{M}^\ast, \bar{u}_r, \bar{a}_n)$. Just by definition, we will have
$([\mathcal{M}^\ast, \bar{u}_r], \bar{c}_n/\bar{a}_n), \bar{u}_r \not\models_\mathcal{L}
\bigvee^{n}_{i = 1}P^-_{a_i}(c_i)$, hence also
$$([(\mathcal{M}_1^{un(w_1)})_\approx, h(\bar{u}_r)], \bar{c}_n/g\langle\bar{a}_n\rangle), h(\bar{u}_r)
\not\models_\mathcal{L} \bigvee^{n}_{i = 1}P^-_{a_i}(c_i).$$ Moreover, for some $\bar{v}_k \in (W_1)^{un(w_1)}$ and some $\bar{\beta}_n \in (A_1)^n_{v_k}$, we will have both $\bar{v}_k = h(\bar{u}_r)$ and  $g\langle\bar{a}_n\rangle = [\langle\bar{\beta}_n\odot \bar{v}_k\rangle]_{\approx(\bar{v}_k)}$. It follows now from Corollary \ref{C:congr}, that $([\mathcal{M}_1^{un(w_1)}, \bar{v}_k], \bar{c}_n/\bar{\beta}_n\odot \bar{v}_k), \bar{v}_k
\not\models_\mathcal{L} \bigvee^{n}_{i = 1}P^-_{a_i}(c_i)$, whence, further, by Lemma \ref{unravellinglemma-gen}, $$([\mathcal{M}_1, v_k], \bar{c}_n/\bar{\beta}_n), v_k
\not\models_\mathcal{L} \bigvee^{n}_{i = 1}P^-_{a_i}(c_i).$$ Thus, by Expansion and Lemma \ref{L:cutoff-constants}.5, also $$([\mathcal{M}'_1, v_k], \bar{c}_n/\bar{\beta}_n), v_k
\not\models_\mathcal{L} \bigvee^{n}_{i = 1}P^-_{a_i}(c_i).$$ Note that we have, in fact, $w_1 = v_1$, so that also $w_1\mathrel{\prec'_1} v_k$, whence it follows, by $\mathcal{M}'_1, w_1 \models_\mathcal{L} \Upsilon$ and Lemma \ref{L:generated-submodels}.2, that also $\mathcal{M}'_1, v_k \models_\mathcal{L} \Upsilon$. Next, we get, by Lemma \ref{L:generated-submodels}.1, that $[\mathcal{M}'_1, v_k], v_k \models_\mathcal{L} \Upsilon$, and further, by Expansion Property and Lemma \ref{L:cutoff-constants}.4, that $([\mathcal{M}'_1, v_k], \bar{c}_n/\bar{\beta}_n), v_k
\models_\mathcal{L} \Upsilon$.
This means, in particular, that:
$$
([\mathcal{M}'_1, v_k], \bar{c}_n/\bar{\beta}_n), v_k
\models_\mathcal{L} (R^+_{\Gamma,\Delta,\bar{v}_k,\bar{a}_n}(\bar{c}_n) \to
R^-_{\Gamma,\Delta,\bar{v}_k,\bar{a}_n}(\bar{c}_n)) \to \bigvee^n_{i = 1}P^-_{a_i}(c_i),
$$
so that we have:
$$
([\mathcal{M}'_1, v_k], \bar{c}_n/\bar{\beta}_n), v_k
\not\models_\mathcal{L} R^+_{\Gamma,\bar{v}_k,\bar{a}_n}(\bar{c}_n) \to
R^-_{\Delta,\bar{v}_k,\bar{a}_n}(\bar{c}_n),
$$
which means, in turn, that for some $v' \in W'_1$ such that
$v_k\mathrel{\prec'_1}v'$ we have:
$$
([\mathcal{M}'_1, v_k], \bar{c}_n/\bar{\beta}_n), v'
\models_\mathcal{L} (\{R^+_{\Gamma,\Delta,\bar{v}_k,\bar{a}_n}(\bar{c}_n)\}, \{R^-_{\Gamma,\Delta,\bar{v}_k,\bar{a}_n}(\bar{c}_n)\}), 
$$
whence we obtain, successively, that:
\begin{align*}
	[([\mathcal{M}'_1, v_k], \bar{c}_n/\bar{\beta}_n), v'], v'
	&\models_\mathcal{L} (\{R^+_{\Gamma,\Delta,\bar{v}_k,\bar{a}_n}(\bar{c}_n)\}, \{R^-_{\Gamma,\Delta,\bar{v}_k,\bar{a}_n}(\bar{c}_n)\})&&\text{(by Lemma \ref{L:generated-submodels}.1)}\\
	([\mathcal{M}'_1, v'], \bar{c}_n/(\mathbb{H}_1)_{v_kv'}\langle\bar{\beta}_n\rangle), v'
	&\models_\mathcal{L} (\{R^+_{\Gamma,\Delta,\bar{v}_k,\bar{a}_n}(\bar{c}_n)\}, \{R^-_{\Gamma,\Delta,\bar{v}_k,\bar{a}_n}(\bar{c}_n)\})&&\text{(by Lemma \ref{L:cutoff-constants}.2)}
\end{align*}
On the other hand, it follows from $w_1\mathrel{\prec'_1}v_k\mathrel{\prec'_1}v'$ and Lemma \ref{L:generated-submodels}.2, that also $\mathcal{M}'_1, v'
\models_\mathcal{L} \Upsilon$, whence, by Lemma \ref{L:generated-submodels}.1, $[\mathcal{M}'_1, v'], v'
\models_\mathcal{L} \Upsilon$, so that, by Expansion Property and Lemma \ref{L:cutoff-constants}.4, $([\mathcal{M}'_1, v'], \bar{c}_n/(\mathbb{H}_1)_{v_kv'}\langle\bar{\beta}_n\rangle), v'
\models_\mathcal{L} \Upsilon$. The latter means, in particular, that:
\begin{align*}
([\mathcal{M}'_1, v'], \bar{c}_n/(\mathbb{H}_1)_{v_kv'}\langle\bar{\beta}_n\rangle), v'
\models_\mathcal{L} \{\forall\bar{c}_n&(R^+_{\Gamma,\bar{v}_k,\bar{a}_n}(\bar{c}_n) \to
\phi),\\ 
&\forall\bar{c}_n(\psi \to
R^-_{\Delta,\bar{v}_k,\bar{a}_n}(\bar{c}_n))\mid \phi \in \Gamma, \psi \in \Delta\}	
\end{align*}
We get then $([\mathcal{M}'_1, v'], \bar{c}_n/(\mathbb{H}_1)_{v_kv'}\langle\bar{\beta}_n\rangle), v'
\models_\mathcal{L} (\Gamma, \Delta)$. By Expansion Property and Lemma \ref{L:cutoff-constants}.5, this implies that $([\mathcal{M}_1, v'], \bar{c}_n/(\mathbb{H}_1)_{v_kv'}\langle\bar{\beta}_n\rangle), v'
\models_\mathcal{L} (\Gamma, \Delta)$. Note that, by $v_k\mathrel{\prec_1}v'$, we also have both $(\bar{v}_k)^\frown v' \in W^{un(w_1)}_1$, and $\bar{v}_k\mathrel{\prec^{un(w_1)}_1}(\bar{v}_k)^\frown v'$. But then, by Lemma \ref{unravellinglemma-gen}.3, also $([\mathcal{M}^{un(w_1)}_1, (\bar{v}_k)^\frown v'], \bar{c}_n/(\mathbb{H}_1)_{v_kv'}\langle\bar{\beta}_n\rangle\odot((\bar{v}_k)^\frown v')), (\bar{v}_k)^\frown v'
\models_\mathcal{L} (\Gamma, \Delta)$. This further means, by the definition of intuitionistic unravellings, that $$([\mathcal{M}^{un(w_1)}_1, (\bar{v}_k)^\frown v'], \bar{c}_n/(\mathbb{H}_1)^{un(w_1)}_{\bar{v}_k(\bar{v}_k)^\frown v'}\langle\bar{\beta}_n\odot\bar{v}_k\rangle), (\bar{v}_k)^\frown v'
\models_\mathcal{L} (\Gamma, \Delta).$$ From the latter equation, by Corollary \ref{C:congr}, we obtain that:
$$
([(\mathcal{M}^{un(w_1)}_1)_{\approx}, \bar{v}_k^\frown v'], \bar{c}_n/[\langle(\mathbb{H}_1)^{un(w_1)}_{\bar{v}_k(\bar{v}_k)^\frown v'}\langle\bar{\beta}_n\odot\bar{v}_k\rangle\rangle]_{\approx((\bar{v}_k)^\frown v')}), (\bar{v}_k)^\frown v'
\models_\mathcal{L} (\Gamma, \Delta),
$$
whence, by definition of a congruence-based model, we get that:
$$
([(\mathcal{M}^{un(w_1)}_1)_{\approx}, (\bar{v}_k)^\frown v'], \bar{c}_n/((\mathbb{H}_1)^{un(w_1)}_\approx)_{\bar{v}_k(\bar{v}_k)^\frown v'}\langle[\langle\bar{\beta}_n\odot\bar{v}_k\rangle]_{\approx(\bar{v}_k)}\rangle), (\bar{v}_k)^\frown v'
\models_\mathcal{L} (\Gamma, \Delta).
$$
In other words, for $\mathbf{v}: = (\bar{v}_k)^\frown v'\mathrel{\succ^{un(w_1)}_1}\bar{v}_k = h(\bar{u}_r)$ we have shown that:
$$
([(\mathcal{M}^{un(w_1)}_1)_{\approx}, \mathbf{v}], \bar{c}_n/((\mathbb{H}_1)^{un(w_1)}_\approx)_{h(\bar{u}_r)\mathbf{v}}\langle g\langle\bar{a}_n\rangle\rangle), \mathbf{v}
\models_\mathcal{L} (\Gamma, \Delta).
$$
We now use the fact that $g' \supseteq g$, $h' \supseteq h$  are such that $(g', h'): \mathcal{N} \cong (\mathcal{M}_1^{un(w_1)})_\approx$, and choose a $\mathbf{u} \in U$ such that both $\bar{u}_r\mathrel{\lhd} \mathbf{u}$ and $\mathbf{v} = h'(\mathbf{u})$. This allows us to rephrase the latter equation in the following form:
$$
([(\mathcal{M}^{un(w_1)}_1)_{\approx}, h'(\mathbf{u})], \bar{c}_n/((\mathbb{H}_1)^{un(w_1)}_\approx)_{h'(\bar{u}_r)h'(\mathbf{u})}\langle g'\langle\bar{a}_n\rangle\rangle), h'(\mathbf{u})
\models_\mathcal{L} (\Gamma, \Delta).
$$
By condition \eqref{E:ic2a}  of Definition \ref{D:isomorphism}, we get that $((\mathbb{H}_1)^{un(w_1)}_\approx)_{h'(\bar{u}_r)h'(\mathbf{u})}\langle g'\langle\bar{a}_n\rangle\rangle = g'\langle\mathbb{G}_{\bar{u}_r\mathbf{u}}\langle\bar{a}_n\rangle\rangle$, so that we can infer
$$
([(\mathcal{M}^{un(w_1)}_1)_{\approx}, h'(\mathbf{u})], \bar{c}_n/g'\langle\mathbb{G}_{\bar{u}_r\mathbf{u}}\langle\bar{a}_n\rangle\rangle), h'(\mathbf{u})
\models_\mathcal{L} (\Gamma, \Delta).
$$
In other words, we have that $(\Gamma, \Delta) \subseteq Tp_\mathcal{L}((\mathcal{M}^{un(w_1)}_1)_{\approx}, h'(\mathbf{u}), \bar{c}_n/g'\langle\mathbb{G}_{\bar{u}_r\mathbf{u}}\langle\bar{a}_n\rangle\rangle)$. And, since every isomorphism is an elementary embedding by Lemma \ref{L:embedding}.2, we can apply condition \eqref{E:c3} from Definition \ref{D:embedding} and get that $(\Gamma, \Delta) \subseteq Tp_\mathcal{L}(\mathcal{N}, \mathbf{u}, \bar{c}_n/\mathbb{G}_{\bar{u}_r\mathbf{u}}\langle\bar{a}_n\rangle)$. Thus we have shown, given that $\bar{u}_r\mathrel{\lhd} \mathbf{u}$, that $(\Gamma, \Delta)$ is realized in $\mathcal{N}$.

\textit{Case 2}. For some $\bar{u}_r \in W^{un(w)}$ and some $\bar{a}_n \in  A^n_{\bar{u}_r}$ the set $\Xi \subseteq
L(\Theta_{\mathcal{M}^\ast}\cup \{\bar{c}_{n + 1}\})$ is a $\bar{c}_{n + 1}$-existential $\mathcal{L}$-type of
$(\mathcal{M}^\ast, \bar{u}_r, \bar{a}_n)$. Just by definition, we will have
$$([\mathcal{M}^\ast, \bar{u}_r], \bar{c}_n/\bar{a}_n), \bar{u}_r \models_\mathcal{L}
\bigwedge^{n}_{i = 1}P^+_{a_i}(c_i),$$ hence also
$([(\mathcal{M}_1^{un(w_1)})_\approx, h(\bar{u}_r)], \bar{c}_n/g\langle\bar{a}_n\rangle), h(\bar{u}_r)
\models_\mathcal{L}
\bigwedge^{n}_{i = 1}P^+_{a_i}(c_i)$. Moreover, for some $\bar{v}_k \in (W_1)^{un(w_1)}$ and some $\bar{\beta}_n \in (A_1)^n_{v_k}$, we will have both $\bar{v}_k = h(\bar{u}_r)$ and  $g\langle\bar{a}_n\rangle = [\langle\bar{\beta}_n\odot \bar{v}_k\rangle]_{\approx(\bar{v}_k)}$. It follows now from Corollary \ref{C:congr}, that $([\mathcal{M}_1^{un(w_1)}, \bar{v}_k], \bar{c}_n/\bar{\beta}_n\odot \bar{v}_k), \bar{v}_k
\models_\mathcal{L}
\bigwedge^{n}_{i = 1}P^+_{a_i}(c_i)$, whence, further, by Lemma \ref{unravellinglemma-gen}, $([\mathcal{M}_1, v_k], \bar{c}_n/\bar{\beta}_n), v_k
\models_\mathcal{L}
\bigwedge^{n}_{i = 1}P^+_{a_i}(c_i)$. Thus, by Expansion and Lemma \ref{L:cutoff-constants}.5, also $([\mathcal{M}'_1, v_k], \bar{c}_n/\bar{\beta}_n), v_k
\models_\mathcal{L}
\bigwedge^{n}_{i = 1}P^+_{a_i}(c_i)$. Note that we have, in fact, $w_1 = v_1$, so that also $w_1\mathrel{\prec'_1} v_k$, whence it follows, by $\mathcal{M}'_1, w_1 \models_\mathcal{L} \Upsilon$ and Lemma \ref{L:generated-submodels}.2, that also $\mathcal{M}'_1, v_k \models_\mathcal{L} \Upsilon$. Next, we get, by Lemma \ref{L:generated-submodels}.1, that $[\mathcal{M}'_1, v_k], v_k \models_\mathcal{L} \Upsilon$, and further, by Expansion Property and Lemma \ref{L:cutoff-constants}.4, that $([\mathcal{M}'_1, v_k], \bar{c}_n/\bar{\beta}_n), v_k
\models_\mathcal{L} \Upsilon$.
This means, in particular, that:
$$
([\mathcal{M}'_1, v_k], \bar{c}_n/\bar{\beta}_n), v_k
\models_\mathcal{L} \bigwedge^n_{i = 1}P^+_{a_i}(c_i) \to \exists
c_{n+1}R^{ex}_{\Xi,\bar{v}_k,\bar{a}_n}(\bar{c}_{n+1}),
$$
so that we have:
$$
([\mathcal{M}'_1, v_k], \bar{c}_n/\bar{\beta}_n), v_k
\models_\mathcal{L} \exists
c_{n+1}R^{ex}_{\Xi,\bar{v}_k,\bar{a}_n}(\bar{c}_{n+1}),
$$
which means, in turn, by Corollary \ref{C:generated-submodels}.4, that for some $\beta_{n + 1} \in (A_1)_{v_k}$ we have:
$$
([\mathcal{M}'_1, v_k], \bar{c}_{n + 1}/\bar{\beta}_{n + 1}), v_k\models_\mathcal{L} R^{ex}_{\Xi,\bar{v}_k,\bar{a}_n}(\bar{c}_{n+1}).
$$
On the other hand, it follows from $([\mathcal{M}'_1, v_k], \bar{c}_n/\bar{\beta}_n), v_k
\models_\mathcal{L} \Upsilon$, Lemma \ref{L:cutoff-constants}.3, and Lemma \ref{L:cutoff-constants}.4, that also:
\begin{align*}
	([\mathcal{M}'_1, v_k], \bar{c}_{n + 1}/\bar{\beta}_{n + 1}), v_k
	\models_\mathcal{L} \{\forall\bar{c}_n(R^{ex}_{\Xi,\bar{v}_k,\bar{a}_n}(\bar{c}_n) \to
	\phi)\mid \phi \in \Xi\}	
\end{align*}
We get then that $([\mathcal{M}'_1, v_k], \bar{c}_{n + 1}/\bar{\beta}_{n + 1}), v_k
\models_\mathcal{L} \Xi$. By Expansion Property and Lemma \ref{L:cutoff-constants}.5, this implies that $([\mathcal{M}_1, v_k], \bar{c}_{n + 1}/\bar{\beta}_{n + 1}), v_k
\models_\mathcal{L} \Xi$. Note that, by $\beta_{n + 1} \in (A_1)_{v_k}$, we also have $(\beta_{n + 1}; \bar{v}_k) \in (A^{un(w_1)}_1)_{\bar{v}_k}$. But then, by Lemma \ref{unravellinglemma-gen}.3, also $$([\mathcal{M}^{un(w_1)}_1, \bar{v}_k], \bar{c}_{n + 1}/\bar{\beta}_{n + 1}\odot\bar{v}_k), \bar{v}_k
\models_\mathcal{L} \Xi.$$ From the latter equation, by Corollary \ref{C:congr}, we obtain that:
$$
([(\mathcal{M}^{un(w_1)}_1)_\approx, \bar{v}_k], \bar{c}_{n + 1}/[\langle\bar{\beta}_{n + 1}\odot\bar{v}_k\rangle]_{\approx(\bar{v}_k)}), \bar{v}_k
\models_\mathcal{L} \Xi.
$$
In other words, we have shown that:
$$
([(\mathcal{M}^{un(w_1)}_1)_\approx, h(\bar{u}_r)], \bar{c}_{n + 1}/g\langle\bar{a}_n\rangle^\frown(\beta_{n + 1}; \bar{v}_k)), h(\bar{u}_r)
\models_\mathcal{L} \Xi.
$$
We now use the fact that $g' \supseteq g$, $h' \supseteq h$  are such that $(g', h'): \mathcal{N} \cong (\mathcal{M}_1^{un(w_1)})_\approx$, and choose an $a_{n + 1} \in B_{\bar{u}_r}$ such that $(\beta_{n + 1}; \bar{v}_k)) = g'(a_{n + 1})$. This allows us to rephrase the latter equation in the following form:
$$
([(\mathcal{M}^{un(w_1)}_1)_\approx, h'(\bar{u}_r)], \bar{c}_{n + 1}/g'\langle\bar{a}_{n + 1}\rangle), h'(\bar{u}_r)
\models_\mathcal{L} \Xi.
$$
In other words, we have that $\Xi \subseteq Tp_\mathcal{L}((\mathcal{M}^{un(w_1)}_1)_{\approx}, h'(\bar{u}_r), \bar{c}_{n + 1}/g'\langle\bar{a}_{n + 1}\rangle)$. And, since every isomorphism is an elementary embedding by Lemma \ref{L:embedding}.2, we can apply condition \eqref{E:c3} from Definition \ref{D:embedding} and get that $\Xi \subseteq Tp_\mathcal{L}(\mathcal{N}, \bar{u}_r, \bar{c}_{n + 1}/\bar{a}_{n + 1})$. Thus we have shown, given that $a_{n + 1} \in B_{\bar{u}_r}$, that $\Xi$ is realized in $\mathcal{N}$.

\textit{Case 3}. For some $\bar{u}_r \in W^{un(w)}$ and some $\bar{a}_n \in  A^n_{\bar{u}_r}$ the set $\Xi \subseteq
L(\Theta_{\mathcal{M}^\ast}\cup \{\bar{c}_{n + 1}\})$ is a $\bar{c}_{n + 1}$-universal $\mathcal{L}$-type of
$(\mathcal{M}^\ast, \bar{u}_r, \bar{a}_n)$. Just by definition, we will have
$$([\mathcal{M}^\ast, \bar{u}_r], \bar{c}_n/\bar{a}_n), \bar{u}_r \not\models_\mathcal{L}
\bigvee^{n}_{i = 1}P^-_{a_i}(c_i),$$ hence also
$$([(\mathcal{M}_1^{un(w_1)})_\approx, h(\bar{u}_r)], \bar{c}_n/g\langle\bar{a}_n\rangle), h(\bar{u}_r)
\not\models_\mathcal{L} \bigvee^{n}_{i = 1}P^-_{a_i}(c_i).$$ Moreover, for some $\bar{v}_k \in (W_1)^{un(w_1)}$ and some $\bar{\beta}_n \in (A_1)^n_{v_k}$, we will have both $\bar{v}_k = h(\bar{u}_r)$ and  $g\langle\bar{a}_n\rangle = [\langle\bar{\beta}_n\odot \bar{v}_k\rangle]_{\approx(\bar{v}_k)}$. It follows now from Corollary \ref{C:congr}, that $([\mathcal{M}_1^{un(w_1)}, \bar{v}_k], \bar{c}_n/\bar{\beta}_n\odot \bar{v}_k), \bar{v}_k
\not\models_\mathcal{L} \bigvee^{n}_{i = 1}P^-_{a_i}(c_i)$, whence, further, by Lemma \ref{unravellinglemma-gen}, $$([\mathcal{M}_1, v_k], \bar{c}_n/\bar{\beta}_n), v_k
\not\models_\mathcal{L} \bigvee^{n}_{i = 1}P^-_{a_i}(c_i).$$ Thus, by Expansion and Lemma \ref{L:cutoff-constants}.5, also $([\mathcal{M}'_1, v_k], \bar{c}_n/\bar{\beta}_n), v_k
\not\models_\mathcal{L} \bigvee^{n}_{i = 1}P^-_{a_i}(c_i)$. Note that we have, in fact, $w_1 = v_1$, so that also $w_1\mathrel{\prec'_1} v_k$, whence it follows, by $\mathcal{M}'_1, w_1 \models_\mathcal{L} \Upsilon$ and Lemma \ref{L:generated-submodels}.2, that also $\mathcal{M}'_1, v_k \models_\mathcal{L} \Upsilon$. Next, we get that, by Lemma \ref{L:generated-submodels}.1, $[\mathcal{M}'_1, v_k], v_k \models_\mathcal{L} \Upsilon$, and further, by Expansion Property and Lemma \ref{L:cutoff-constants}.4, that $$([\mathcal{M}'_1, v_k], \bar{c}_n/\bar{\beta}_n), v_k
\models_\mathcal{L} \Upsilon.$$
This means, in particular, that:
$$
([\mathcal{M}'_1, v_k], \bar{c}_n/\bar{\beta}_n), v_k
\models_\mathcal{L} \forall c_{n+1}R^{all}_{\Xi,\bar{v}_k,\bar{a}_n}(\bar{c}_{n+1}) \to \bigvee^n_{i = 1}P^-_{a_i}(c_i),
$$
so that we have:
$$
([\mathcal{M}'_1, v_k], \bar{c}_n/\bar{\beta}_n), v_k
\not\models_\mathcal{L} \forall c_{n+1}R^{all}_{\Xi,\bar{v}_k,\bar{a}_n}(\bar{c}_{n+1}),
$$
which means, in turn, that for some $v' \in W'_1$ such that
$v_k\mathrel{\prec'_1}v'$ and for some $\beta_{n + 1} \in (A_1)_{v'}$ we have, by Corollary \ref{C:generated-submodels}.4, that:
$$
	([\mathcal{M}'_1, v'], \bar{c}_{n + 1}/(\mathbb{H}_1)_{v_kv'}\langle\bar{\beta}_{n}\rangle^\frown\beta_{n+1}), v' 
	\not\models_\mathcal{L} R^{all}_{\Xi,\bar{v}_k,\bar{a}_n}(\bar{c}_{n+1}).
$$
On the other hand, it follows from $w_1\mathrel{\prec'_1}v_k\mathrel{\prec'_1}v'$ and Lemma \ref{L:generated-submodels}.2, that also $\mathcal{M}'_1, v'
\models_\mathcal{L} \Upsilon$, whence, by Lemma \ref{L:generated-submodels}.1, $[\mathcal{M}'_1, v'], v'
\models_\mathcal{L} \Upsilon$, so that, by Expansion Property and Lemma \ref{L:cutoff-constants}.4, $([\mathcal{M}'_1, v'], \bar{c}_{n + 1}/(\mathbb{H}_1)_{v_kv'}\langle\bar{\beta}_{n}\rangle^\frown\beta_{n+1}), v'
\models_\mathcal{L} \Upsilon$. The latter means, in particular, that:
\begin{align*}
([\mathcal{M}'_1, v'], \bar{c}_{n + 1}/(\mathbb{H}_1)_{v_kv'}\langle\bar{\beta}_{n}\rangle^\frown\beta_{n+1}), v'
	\models_\mathcal{L} \{\forall\bar{c}_{n+1}(\psi \to
	R^{all}_{\Xi,\bar{v}_k,\bar{a}_n}(\bar{c}_{n+1}))\mid \psi \in \Xi\}	
\end{align*}
We get then $([\mathcal{M}'_1, v'], \bar{c}_{n + 1}/(\mathbb{H}_1)_{v_kv'}\langle\bar{\beta}_{n}\rangle^\frown\beta_{n+1}), v'
\models_\mathcal{L} (\emptyset, \Xi)$. By Expansion Property, this implies that $([\mathcal{M}_1, v'], \bar{c}_{n + 1}/(\mathbb{H}_1)_{v_kv'}\langle\bar{\beta}_{n}\rangle^\frown\beta_{n+1}), v'
\models_\mathcal{L} (\emptyset, \Xi)$. Note that, by $v_k\mathrel{\prec_1}v'$, we also have both $(\bar{v}_k)^\frown v' \in W^{un(w_1)}_1$, and $\bar{v}_k\mathrel{\prec^{un(w_1)}_1}(\bar{v}_k)^\frown v'$; moreover, we have $(\beta_{n + 1}; (\bar{v}_k)^\frown v') \in (A^{un(w_1)}_1)_{(\bar{v}_k)^\frown v'}$. But then, by Lemma \ref{unravellinglemma-gen}.3, also $$([\mathcal{M}^{un(w_1)}_1, (\bar{v}_k)^\frown v'], \bar{c}_{n + 1}/((\mathbb{H}_1)_{v_kv'}\langle\bar{\beta}_{n}\rangle^\frown\beta_{n+1})\odot (\bar{v}_k)^\frown v'), (\bar{v}_k)^\frown v'
\models_\mathcal{L} (\emptyset, \Xi).$$ This further means, by the definition of intuitionistic unravellings, that 
$$
([\mathcal{M}^{un(w_1)}_1, (\bar{v}_k)^\frown v'], \bar{c}_{n + 1}/(\mathbb{H}_1)^{un(w_1)}_{\bar{v}_k(\bar{v}_k)^\frown v'}\langle\bar{\beta}_n\odot\bar{v}_k\rangle^\frown(\beta_{n + 1}; (\bar{v}_k)^\frown v')), (\bar{v}_k)^\frown v'
\models_\mathcal{L} (\emptyset, \Xi).
$$

From the latter equation (having observed that  $[(\beta_{n + 1}; (\bar{v}_k)^\frown v')]_{\approx((\bar{v}_k)^\frown v')} \in ((A^{un(w_1)}_1)_\approx)_{(\bar{v}_k)^\frown v'}$), by Corollary \ref{C:congr}, we obtain that:
$$
([(\mathcal{M}^{un(w_1)}_1)_{\approx}, (\bar{v}_k)^\frown v'], \bar{c}_{n + 1}/([\langle(\mathbb{H}_1)^{un(w_1)}_{\bar{v}_k(\bar{v}_k)^\frown v'}\langle\bar{\beta}_n\odot\bar{v}_k\rangle\rangle]_{\approx((\bar{v}_k)^\frown v')})^\frown[(\beta_{n + 1}; (\bar{v}_k)^\frown v')]_{\approx((\bar{v}_k)^\frown v')}), (\bar{v}_k)^\frown v'
\models_\mathcal{L}(\emptyset, \Xi),
$$
whence, by definition of a congruence-based model, we get that:
$$
([(\mathcal{M}^{un(w_1)}_1)_{\approx}, (\bar{v}_k)^\frown v'], \bar{c}_{n + 1}/((\mathbb{H}_1)^{un(w_1)}_\approx)_{\bar{v}_k(\bar{v}_k)^\frown v'}\langle[\langle\bar{\beta}_n\odot\bar{v}_k\rangle]_{\approx(\bar{v}_k)}\rangle^\frown[(\beta_{n + 1}; (\bar{v}_k)^\frown v')]_{\approx((\bar{v}_k)^\frown v')}), (\bar{v}_k)^\frown v'
\models_\mathcal{L} (\emptyset, \Xi).
$$
In other words, for $\mathbf{v}: = (\bar{v}_k)^\frown v'\mathrel{\succ^{un(w_1)}_1}\bar{v}_k = h(\bar{u}_r)$ we have shown that:
$$
([(\mathcal{M}^{un(w_1)}_1)_{\approx}, \mathbf{v}], \bar{c}_{n + 1}/((\mathbb{H}_1)^{un(w_1)}_\approx)_{h(\bar{u}_r)\mathbf{v}}\langle g\langle\bar{a}_n\rangle\rangle^\frown[(\beta_{n + 1}; (\bar{v}_k)^\frown v')]_{\approx((\bar{v}_k)^\frown v')}), \mathbf{v}
\models_\mathcal{L}  (\emptyset, \Xi).
$$
We now use the fact that $g' \supseteq g$, $h' \supseteq h$  are such that $(g', h'): \mathcal{N} \cong (\mathcal{M}_1^{un(w_1)})_\approx$, and choose a $\mathbf{u} \in U$ such that both $\bar{u}_r\mathrel{\lhd} \mathbf{u}$ and $\mathbf{v} = h'(\mathbf{u})$ and an $a_{n + 1} \in B_{\mathbf{u}}$ such that $[(\beta_{n + 1}, (\bar{v}_k)^\frown v')]_{\approx((\bar{v}_k)^\frown v')} = g'(a_{n + 1})$. This allows us to rephrase the latter equation in the following form:
$$
([(\mathcal{M}^{un(w_1)}_1)_{\approx}, h'(\mathbf{u})], \bar{c}_{n + 1}/((\mathbb{H}_1)^{un(w_1)}_\approx)_{h'(\bar{u}_r)h'(\mathbf{u})}\langle g'\langle\bar{a}_n\rangle\rangle^\frown g'(a_{n + 1})), h'(\mathbf{u})
\models_\mathcal{L}  (\emptyset, \Xi).
$$
By condition \eqref{E:ic2a}  of Definition \ref{D:isomorphism}, we get that $((\mathbb{H}_1)^{un(w_1)}_\approx)_{h'(\bar{u}_r)h'(\mathbf{u})}\langle g'\langle\bar{a}_n\rangle\rangle = g'\langle\mathbb{G}_{\bar{u}_r\mathbf{u}}\langle\bar{a}_n\rangle\rangle$, so that we can infer
$$
([(\mathcal{M}^{un(w_1)}_1)_{\approx}, h'(\mathbf{u})], \bar{c}_{n + 1}/g'\langle\mathbb{G}_{\bar{u}_r\mathbf{u}}\langle\bar{a}_n\rangle\rangle^\frown a_{n + 1}\rangle), h'(\mathbf{u})
\models_\mathcal{L} (\emptyset, \Xi).
$$
In other words, we have that $(\emptyset, \Xi) \subseteq Tp_\mathcal{L}((\mathcal{M}^{un(w_1)}_1)_{\approx}, h'(\mathbf{u}), \bar{c}_n/g'\langle\mathbb{G}_{\bar{u}_r\mathbf{u}}\langle\bar{a}_n\rangle\rangle^\frown a_{n + 1}\rangle)$. And, since every isomorphism is an elementary embedding by Lemma \ref{L:embedding}.2, we can apply condition \eqref{E:c3} from Definition \ref{D:embedding} and get that $(\emptyset, \Xi) \subseteq Tp_\mathcal{L}(\mathcal{N}, \mathbf{u}, \bar{c}_n/\mathbb{G}_{\bar{u}_r\mathbf{u}}\langle\bar{a}_n\rangle^\frown a_{n + 1})$. Thus we have shown, given that both $\bar{u}_r\mathrel{\lhd} \mathbf{u}$ and $a_{n + 1} \in B_{\mathbf{u}}$, that $\Xi$ is realized in $\mathcal{N}$.

Our Claim is thus proven.

It remains to notice that, by Lemma \ref{L:type-realization}.2, $\mathcal{N}':= \mathcal{N}\upharpoonright\Theta$ must realize every $\mathcal{L}$-type of $\mathcal{M} = \mathcal{M}^\ast\upharpoonright\Theta$, and that, by $(g', h'):\mathcal{N} \cong ((\mathcal{M}_1^{un(w_1)})_\approx, w_1)$, $\mathcal{N}'$ is a $\Theta$-reduct of an unravelled model and hence an unravelled $\Theta$-model itself. Finally, note that $(\mathcal{N}, w) \in Str_\mathcal{L}(\Theta_{\mathcal{M}})$ also implies that $(\mathcal{N}', w) \in Str_\mathcal{L}(\Theta)$.
\end{proof}

\begin{corollary}\label{C:saturation}
Let $\mathcal{L}\sqsupseteq \mathcal{L}' \in StIL$ be an abstract intuitionistic logic which is
preserved under $\mathcal{L}'$-asimulations, $\star$-compact, and has TUP, and let $(\mathcal{M}, w)\in Str_\mathcal{L}(\Theta)$. Then there exists a $(\mathcal{N}, w) \in Str_\mathcal{L}(\Theta)$ such that $Th_\mathcal{L}(\mathcal{M}, w) = Th_\mathcal{L}(\mathcal{N}, w)$ and $\mathcal{N}$ is $\mathcal{L}$-saturated.	
\end{corollary}
\begin{proof}
	It follows from the preservation of $\mathcal{L}$ under $\mathcal{L}'$-asimulations and Lemma \ref{unravellinglemma-gen}.3 that $Th_\mathcal{L}(\mathcal{M}, w) = Th_\mathcal{L}(\mathcal{M}^{un(w)}, w)$. Next, aplying Proposition \ref{L:saturation} ($\omega$ times), we form an $\mathcal{L}$-elementary chain of submodels
	$$
	\mathcal{M}^{un(w)} = \mathcal{N}_1 \preccurlyeq_\mathcal{L}\ldots \preccurlyeq_\mathcal{L} \mathcal{N}_n \preccurlyeq_\mathcal{L}\ldots
	$$
	such that for every $i > 0$, we have $(\mathcal{N}_i, w) \in Str_\mathcal{L}(\Theta)$ and $\mathcal{N}_{i + 1}$ realizes every $\mathcal{L}$-type of $\mathcal{N}_{i}$. We then set $\mathcal{N} := \bigcup_{i > 0}\mathcal{N}_{i}$ so that $(\mathcal{N}, w) \in Str_\mathcal{L}(\Theta)$. Since $\mathcal{L}$ has TUP, we have then $\mathcal{N}_{i} \preccurlyeq_\mathcal{L} \mathcal{N}$ for every $i > 0$. In particular, we have $\mathcal{M}^{un(w)} = \mathcal{N}_1 \preccurlyeq_\mathcal{L} \mathcal{N}$, so that  $Th_\mathcal{L}(\mathcal{M}, w) = Th_\mathcal{L}(\mathcal{N}, w)$ follows. It remains to show the $\mathcal{L}$-saturation of $\mathcal{N}$, that is to say, that $\mathcal{N}$ realizes all of its $\mathcal{L}$-types. Since we have three sorts of types in $\mathcal{L}$, we have to consider three possible cases, having fixed a set  $\{c_i \mid i > 0\}$ of pairwise distinct fresh individual constants.
	
	\textit{Case 1}. For some $v \in U$ and some $\bar{a}_n \in B^n_{v}$
	the sets $\Gamma, \Delta \subseteq
	L(\Theta\cup \{\bar{c}_n\})$ are such that
	$(\Gamma, \Delta)$ is a $\bar{c}_n$-successor $\mathcal{L}$-type of
	$(\mathcal{N}, v, \bar{a}_n)$. Then, by Corollary \ref{C:types-formulas}.1, for all $\Gamma'\Subset \Gamma$ and $\Delta' \Subset \Delta$, we have $\bigwedge\Gamma'\to \bigvee\Delta' \in Tp_{\mathcal{L}}(\mathcal{N}, v, \bar{c}_n/\bar{a}_n)$. But then, we can choose a $j > 0$ such that $\mathcal{N}_{j}$ is the first model in the chain for which we have both $v \in U_j$ and $\bar{a}_n \in (B_j)^n_{v}$. Since we also have $\mathcal{N}_{j} \preccurlyeq_\mathcal{L} \mathcal{N}$, it follows that $Tp_{\mathcal{L}}(\mathcal{N}_j, v, \bar{c}_n/\bar{a}_n) = Tp_{\mathcal{L}}(\mathcal{N}, v, \bar{c}_n/\bar{a}_n)$ and thus, for all $\Gamma'\Subset \Gamma$ and $\Delta' \Subset \Delta$, we have $\bigwedge\Gamma'\to \bigvee\Delta' \in Tp_{\mathcal{L}}(\mathcal{N}_j, v, \bar{c}_n/\bar{a}_n)$. But then, by Corollary \ref{C:types-formulas}.1, 	$(\Gamma, \Delta)$ must be a $\bar{c}_n$-successor $\mathcal{L}$-type of
	$(\mathcal{N}_j, v, \bar{a}_n)$, and, therefore, $(\Gamma, \Delta)$ must be realized in $\mathcal{N}_{j + 1}$. But, since we also have $\mathcal{N}_{j + 1}\preccurlyeq_\mathcal{L} \mathcal{N}$, it follows from Lemma \ref{L:type-realization}.1, that $(\Gamma, \Delta)$ also must be realized in $\mathcal{N}$ itself.
	
	\textit{Case 2} and \textit{Case 3}, where we assume that our $\mathcal{L}$-type is either an existential or a universal $\mathcal{L}$-type, respectively, are solved in the same manner. The only difference from Case 1 is that we need to apply Corollary \ref{C:types-formulas}.2 and Corollary \ref{C:types-formulas}.3, respectively, in place of Corollary \ref{C:types-formulas}.1. 	 
\end{proof}

We are now in a position to prove Theorem \ref{L:main}. Indeed,
assume the hypothesis of the theorem, and assume, for
contradiction, that $\mathcal{L}\not\bowtie\mathcal{L}'$. By
Proposition \ref{L:proposition1}, there must be a $\phi \in
L(\Theta_\phi)$ and
$(\mathcal{M}_1, w_1)$, $(\mathcal{M}_2, w_2) \in Str_\mathcal{L}(\Theta_\phi)$ such that
$Th^+_{\mathcal{L}'}(\mathcal{M}_1, w_1) \subseteq Th^+_{\mathcal{L}'}(\mathcal{M}_2,
w_2)$ while $\mathcal{M}_1, w_1 \models_\mathcal{L} \phi$ and
$\mathcal{M}_2, w_2 \not\models_\mathcal{L} \phi$. By Corollary
\ref{C:saturation}, take $\mathcal{L}$-saturated
$\Theta_\phi$-models $\mathcal{N}_1$ and $\mathcal{N}_2$ such that
$(\mathcal{N}_i, w_i) \in Str_\mathcal{L}(\Theta_\phi)$, and $Th_\mathcal{L}(\mathcal{M}_i, w_i) =
Th_\mathcal{L}(\mathcal{N}_i, w_i)$ for $i \in \{ 1, 2 \}$. We
will have then, of course, that $\mathcal{N}_1, w_1
\models_\mathcal{L} \phi$, but $\mathcal{N}_2, w_2
\not\models_\mathcal{L} \phi$. On the other hand, we will still
have
\[
Th^+_{\mathcal{L}'}(\mathcal{N}_1, w_1) \subseteq
Th^+_{\mathcal{L}'}(\mathcal{N}_2, w_2),
\]
 whence, by Corollary
\ref{L:asimulationscorollary}, there must be an $\mathcal{L}'$-asimulation $A$
from $(\mathcal{N}_1, w_1, \Lambda)$ to $(\mathcal{N}_2, w_2, \Lambda)$, but then,
since $\mathcal{L}$ is preserved under $\mathcal{L}'$-asimulations, we must also
have \[
Th^+_\mathcal{L}(\mathcal{N}_1, w_1) \subseteq
Th^+_\mathcal{L}(\mathcal{N}_2, w_2).
\]
 Now, since $\phi \in
Th^+_\mathcal{L}(\mathcal{N}_1, w_1)$, we can see that we must also have $\phi \in
Th^+_\mathcal{L}(\mathcal{N}_2, w_2)$, so that $\mathcal{N}_2, w_2
\models_\mathcal{L} \phi$, which is a contradiction.

\section{Conclusion}\label{S:conclusion}

Our central result (Theorem \ref{L:main} from Section \ref{S:main}) established that no standard first-order intuitionistic logic has
proper extensions with TUP, compactness and preservation under asimulations. Hence, any such extension would have to lack one of these properties (e.g. infinitary extensions fail to have compactness and extensions obtained by adding Boolean negation fail to be preserved under asimulations).

The most important open question in this paper is whether there are other nice characterizations of intuitionistic first-order logic. That is, can we hope, for example, to replace the Tarski Union Property by some other natural condition? In the case of modal logic, TUP can be replaced by the so called ``relativization property'' \cite{vanB}. It would be interesting to see if some similar idea works in our present setting. However, such a difficult task goes beyond the scope of the present (already sufficiently long) contribution.
\section*{Acknowledgments}

Grigory Olkhovikov is supported by Deutsche
Forschungsgemeinschaft (DFG), project  OL 589/1-1.


\begin{thebibliography}{9}
\bibitem{baok}
 G. Badia and G. Olkhovikov.
\newblock A Lindstr\"om theorem for intuitionistic propositional logic. 
\newblock {\em Notre Dame Journal of Formal Logic},  61 (1):  11--30 (2020).


\bibitem{bar1}

   J. Barwise. Axioms for abstract model theory.
   \emph{Annals of Mathematical Logic}  7: 221--265   (1974).

\bibitem{barfer}
J. Barwise, S. Feferman, eds.. \emph{Model-Theoretic Logics},
Springer-Verlag (1985).

\bibitem{BdRV}
P. Blackburn, M. de Rijke, Y. Venema. \emph{Modal Logic}, CUP
(2001).

\bibitem{chagrov}
A. Chagrov and M. Zakharyaschev.   \emph{Modal Logic},
   Clarendon Press (1997).



\bibitem{rijke}
          M. De Rijke. A Lindstr\"om Theorem for Modal Logic. In: \emph{Modal Logic and Process Algebra}. Lecture Notes 53, CSLI Publications, Stanford (1995).




\bibitem{enqvist}
S. Enqvist. A General Lindstr\"om Theorem for Some Normal Modal
Logics. \emph{Logica Universalis}, 7(2): 233--264 (2013)

\bibitem{fe}

S. Feferman. Two notes on abstract model theory. I. Properties
invariant on the range of definable relations between structures.
\emph{ Fundamenta Mathematicae}, 82: 153--165 (1974) .



  \bibitem{flum}
          J. Flum. First-order logic and its extensions. In: \emph{Proceedings of the International Summer Institute and Logic Colloquium, Kiel 1974}, Springer-Verlag, pp. 248--310 (1975).


\bibitem{GabbayMaksimova}
 D. Gabbay and L. Maksimova. \emph{Interpolation and Definability: Modal and Intuitionistic
 Logics}, OUP (2005).

\bibitem{mm}

M. Garc\'ia-Matos and J. V\"a\"an\"anen. Abstract Model Theory as
a Framework for Universal Logic. In: \emph{Logica Universalis:
Towards a General Theory of Logic}, Jean-Yves Beziau (ed.),
Birkh\"{a}user (2005).


\bibitem{hey}
A. Heyting,  \emph{Die formalen Regeln der intuitionistischen Mathematik II}, Sitz. Preuss. Akad. Wiss., phys.-math. Kl.: 57--71 (1930).


 \bibitem{kurto}
        N. Kurtonina and M. de Rijke.
          Simulating without negation. \emph{Journal of Logic and Computation}, 7 (4): 501--522 (1997).


 \bibitem{lin}
          P. Lindstr\"om. On extensions of elementary logic. \emph{ Theoria} 35: 1--11 (1969).

 \bibitem{lin2}
P. Lindstr\"om.  A Characterization of Elementary Logic. In: \emph{Modality, Morality, and
Other Problems of Sense and Nonsense}, pp. 189--191. Lund: Gleerups Bokforlag (1973).

 \bibitem{lin3}
P. Lindstr\"om.   On Characterizing Elementary Logic. In S. Stenlund (ed.), \emph{Logical Theory and Semantic Analysis}, Synthese Library vol. 63, pp. 129--146. Dordrecht:
D. Reidel (1974).

 \bibitem{lin4}
P. Lindstr\"om.  Omitting Uncountable Types and Extensions of Elementary Logic.
\emph{Theoria}, 44: 152--156 (1978).

 \bibitem{na}
M. Nadel.
\newblock  Infinitary intuitionistic logic from a classical point of view.
\newblock \emph{Annals of Mathematical Logic} 14 (2): 159-191 (1978).

\bibitem{o}
        G. K. Olkhovikov.
          Model-theoretic Characterization of Intuitionistic Propositional Formulas. \emph{The Review of Symbolic Logic}, 6 (2): 348--365 (2013).

\bibitem{o1}
G. K.~Olkhovikov.
\newblock Model-theoretic characterization of intuitionistic predicate
  formulas.
\newblock {\em Journal of Logic and Computation}, 24: 809--829 (2014).

\bibitem{o2}
G. K.~{Olkhovikov}.
\newblock {Expressive power of basic modal intuitionistic logic as a fragment
  of classical FOL}.
\newblock {\em Journal of Applied Logic}, http://dx.doi.org/10.1016/j.jal.2016.11.036.

\bibitem{o3}
G. K.~{Olkhovikov}.
\newblock {On generalized Van-Benthem-type characterizations}.
\newblock {Annals of Pure and Applied Logic
},  168 (9): 1643--1691 (2017).

\bibitem{otto}
R. Piro and M. Otto. A Lindstr\"{o}m Characterisation of the
Guarded Fragment and of Modal Logic With a Global Modality. In:
\emph{Advances in Modal Logic} 7, AiML 2008, C. Areces and R.
Goldblatt (eds), pp. 273--288.



\bibitem{rau}

 Cecylia Rauszer. Model Theory for an Extension of Intuitionistic Logic.
\emph{Studia Logica}, 36:
1/2, pp. 73--87 (1977).


\bibitem{vanB}

J. van Benthem. A new modal Lindstr\"om theorem. \emph{Logica
Universalis} 1 (1):125--138 (2007).

\bibitem{zopo}

R. Zoghifard and M. Pourmahdian. First-order modal logic: frame definability and a Lindstr\"om theorem. \emph{Studia Logica}. 106 (4): 699-720 (2018).


\end{thebibliography}
\end{document}